\documentclass{mcom-l}
\usepackage{graphicx}
\usepackage{float} 
\usepackage{subfigure}
\usepackage{caption}
\usepackage{amsmath,amssymb,amstext}
\usepackage{mathrsfs}
\usepackage{multirow}
\usepackage{algorithm}
\usepackage{algorithmic}
\usepackage{threeparttable}
\usepackage[multiple]{footmisc}
\usepackage{hyperref}
\usepackage{xcolor}
\usepackage[top=2.5cm,bottom=2.5cm,right=2.5cm,left=2.5cm]{geometry}
\usepackage{booktabs}

\newtheorem{theorem}{Theorem}[section]
\newtheorem{lemma}[theorem]{Lemma}

\theoremstyle{definition}
\newtheorem{definition}[theorem]{Definition}

\newtheorem{assumption}[theorem]{Assumption}
\theoremstyle{remark}
\newtheorem{remark}[theorem]{Remark}

\numberwithin{equation}{section}

\begin{document}
	
	\title[Zeroth-Order Algorithms for Nonconvex Minimax Problems with Coupled linear Constraints]{Zeroth-Order primal-dual Alternating Projection Gradient Algorithms for Nonconvex Minimax Problems with Coupled linear Constraints}
	\thanks{Zi Xu is the corresponding author. This work is supported by National Key  R \& D Program of China (Nos. 2021YFA1000300 and 2021YFA1000301), the NSFC grants (Nos. 12471294, 12021001 and 92473208), and the Postdoctoral Fellowship Program of China Postdoctoral Science Foundation (CPSF) under Grant Numbers GZB20240802 and 2024M763470.}

	\author{Huiling Zhang}
	\address{Department of Mathematics,  College of Sciences, Shanghai University, Shanghai 200444, P.R.China;
		LSEC, ICMSEC, Academy of Mathematics and Systems Science,Chinese Academy of Sciences, Beijing 100190, China.}
	\email{zhanghl1209@shu.edu.cn}

	\author{Zi Xu}
	\address{Department of Mathematics, College of Sciences, Shanghai University, Shanghai 200444, P.R.China;
		Newtouch Center for Mathematics of Shanghai University, Shanghai 200444, P.R.China.}
	\email{xuzi@shu.edu.cn}
	
	\author{Yu-Hong Dai}
	\address{State Key Laboratory of Mathematical Sciences, Academy of Mathematics and Systems Science, Chinese Academy of Sciences, Beijing 100190, China;  School of Mathematical Sciences, University of Chinese Academy of Sciences, Beijing 100049, China.}
	\email{dyh@lsec.cc.ac.cn}
	
	\subjclass[2020]{Primary 90C47, 90C26,  90C30}
	%
	
	
	\keywords{nonconvex minimax problem, zeroth-order algorithm, primal-dual alternating gradient projection algorithm, zeroth-order regularized momentum primal-dual projected gradient algorithm, iteration complexity, machine learning}
	
	\begin{abstract}
		In this paper, we study zeroth-order algorithms for nonconvex minimax problems with coupled linear constraints under the deterministic and stochastic settings, which have attracted wide attention in machine learning, signal processing and  many other fields in recent years, e.g., adversarial attacks in resource allocation problems and network flow problems etc. We propose two single-loop algorithms, namely the zeroth-order primal-dual alternating projected gradient (ZO-PDAPG) algorithm and the zeroth-order regularized momentum primal-dual projected gradient algorithm (ZO-RMPDPG), for solving deterministic and stochastic nonconvex-(strongly) concave minimax problems with coupled linear constraints. The iteration complexity of the two proposed algorithms to obtain an $\varepsilon$-stationary point are proved to be $\mathcal{O}(\varepsilon ^{-2})$ (resp. $\mathcal{O}(\varepsilon ^{-4})$) for solving nonconvex-strongly concave (resp. nonconvex-concave) minimax problems with coupled linear constraints under deterministic settings and $\tilde{\mathcal{O}}(\varepsilon ^{-3})$ (resp. $\tilde{\mathcal{O}}(\varepsilon ^{-6.5})$)  under stochastic settings respectively. To the best of our knowledge, they are the first two zeroth-order algorithms with iterative complexity guarantees for solving nonconvex-(strongly) concave minimax problems with coupled linear constraints under the deterministic and stochastic settings. {\color{black}The proposed ZO-RMPDPG algorithm, when specialized to stochastic nonconvex-concave minimax problems without coupled constraints, outperforms all existing zeroth-order algorithms by achieving a better iteration complexity, thus setting a new state-of-the-art.}
	\end{abstract}

	\maketitle
	
	\section{Introduction}\label{sec1}
	In this paper, we consider the following minimax optimization problem with coupled linear constraints under the deterministic setting, i.e.,
	\begin{equation}\label{problem:1}
		\min_{x\in\mathcal{X}}\max\limits_{\substack{ y\in\mathcal{Y}\\Ax+By\unlhd c}}f(x,y),\tag{P}
	\end{equation}
	and the stochastic setting with the objective function being an expectation function, i.e.,
	\begin{equation}\label{problem:s}
		\min_{x\in\mathcal{X}}\max\limits_{\substack{ y\in\mathcal{Y}\\Ax+By\unlhd c}}g(x,y)=\mathbb{E}_{\zeta\sim D}[G(x,y,\zeta)],\tag{P-S}
	\end{equation}
	where $\mathcal{X}\subseteq \mathbb{R}^{d_x}$ and $\mathcal{Y}\subseteq \mathbb{R}^{d_y}$ are both nonempty convex and compact sets,  $A\in\mathbb{R}^{p\times d_x}$, $B\in\mathbb{R}^{p\times d_y}$ and $c\in\mathbb{R}^{p}$, $\unlhd$ denotes $\le$ or $=$, $f(x,y),G(x,y,\zeta): \mathcal{X} \times \mathcal{Y} \rightarrow \mathbb{R}$ are smooth functions, possibly nonconvex in $x$ and (strongly) concave in $y$, $\zeta$ is a random variable following an unknown
	distribution $D$, and $\mathbb{E}$ denotes the expectation function. We use $[Ax+By]_i\unlhd [c]_i$ to denote the $i$th constraint, and define $\mathcal{K}:= \{1,\dots , p\}$ as the index set of the constraints. Note that when $A=B=c=0$, \eqref{problem:1} and \eqref{problem:s} degenerate into the following minimax optimization problems:
	\begin{equation}\label{min}
		\min \limits_{x\in \mathcal{X}}\max \limits_{y\in \mathcal{Y}}f(x,y),
	\end{equation}
	and 
	\begin{equation}\label{min-s}
		\min \limits_{x\in \mathcal{X}}\max \limits_{y\in \mathcal{Y}}g(x,y)=\mathbb{E}_{\zeta\sim D}[G(x,y,\zeta)].
	\end{equation}
	
	Recently, widespread applications  can be modeled as \eqref{problem:1}, such as adversarial attacks in resource allocation problems and network flow problems \cite{tsaknakis2021minimax}, generalized absolute value equations and generalized linear projection equations \cite{dai2022optimality}. There are many applications that can be expressed as \eqref{min}, including adversarial attacks on deep neural networks (DNN), reinforcement learning, robust training, hyperparameter adjustment \cite{chen2017zoo,finlay2021scaleable,qian2019robust,snoek2012practical} and many others \cite{Chen20,Giannakis,Liao,Mateos}.

	We call \eqref{min} a (non)convex-(non)concave minimax problem if $f(x, y)$ is (non)convex in $x$ and (non)concave in $y$. There are some first-order algorithms for the (strongly) convex-(strongly) concave {\color{black}minimax problem \eqref{min}, such as the accelerated mirror-prox method \cite{Chen}, the prox-type method \cite{Nemirovski} and the dual extrapolation algorithm \cite{Nesterov2007}, which can achieve the optimal iteration complexity among the first-order algorithms for solving convex-concave minimax problems, namely $\mathcal{O}(\varepsilon^{-1})$ \cite{Ou21}. For some related results, see \cite{Gidel,Mokhtari,Ou15,Tom}
		
		For the nonconvex-(strongly) concave minimax problem \eqref{min}, Zhang et al. \cite{Zhang21} proposed a general acceleration framework to solve the nonconvex-strongly concave minimax problem, which can obtain the best known iterative complexity $\tilde{\mathcal{O}}\left(\sqrt{\kappa} \varepsilon ^{-2} \right)$, where the condition number $\kappa$ is defined as $\kappa=L/\mu_y$, $L$ represents the Lipschitz constant of the objective function, and $\mu_y$ is a strongly concave constant, $\tilde{\mathcal{O}}(\cdot)$ means that only absolute constants and logarithmic terms are ignored. For general nonconvex-concave minimax problems, there are two types of algorithms: multi-loop algorithms and single-loop algorithms. The MINIMAX-PPA algorithm proposed in \cite{lin2020near} achieves the best known iteration complexity of $\tilde{\mathcal{O}}\left(\varepsilon ^{-2.5} \right)$ among multi-loop algorithms. Single-loop algorithms, such as the alternating gradient projection (AGP) algorithm \cite{Xu21} and the smoothed GDA algorithm \cite{Zhang}, can achieve an iteration complexity of $\mathcal{O}(\varepsilon^{-4})$. For more related results, see \cite{Boroun, Kong, Lin2019, Lu, Pan}.
		
		Yang et al. \cite{Yang} proved that the alternating gradient descent algorithm converges globally at a linear rate for a subclass of nonconvex- nonconcave objective functions that satisfy the bilateral Polyak-{\L}ojasiewicz (PL) condition. Song et al. \cite{SongOptDE} proposed an optimistic dual extrapolation method with an iteration complexity of $\mathcal{O}\left( \varepsilon^{-2} \right)$ under certain conditions. Zhang et al. \cite{Zhang23} proposed a doubly smoothed gradient descent ascent method to solve the nonconvex-nonconcave minimax problem with the one-sided Kurdyka-\L ojasiewicz (KL) condition, achieving convergence with an iteration complexity of $\mathcal{O}\left( \varepsilon^{-4} \right)$. For more recent results, see \cite{Chinchilla,Doan,Grimmer,Hajizadeh,Jiang}.} 
	
	Problem \eqref{problem:1} is more challenging than \eqref{min}, and it is NP-hard to find its globally optimal solutions even if $f(x, y)$ is strongly convex with respect to $x$ and strongly concave with respect to $y$ \cite{tsaknakis2021minimax}. Few algorithms have been proposed to solve problem \eqref{problem:1}. Tsaknakis et al. \cite{tsaknakis2021minimax} proposed a family of efficient algorithms named multiplier gradient descent (MGD) for solving \eqref{problem:1}, which achieves the iteration complexity of $\tilde{\mathcal{O}}\left( \varepsilon ^{-2} \right)$ when $f(x,y)$ is strongly convex with respect to $x$ and strongly concave with respect to $y$.
	Dai et al. \cite{dai2022optimality} proposed a proximal gradient multi-step ascent descent method (PGmsAD) for nonsmooth convex-concave bilinear coupled minimax problems with linearly {\color{black}equality} constraints when $\mathcal{X}=\mathbb{R}^n$ and $\mathcal{Y}=\mathbb{R}^m$, which achieves the iteration complexity bound of $\mathcal{O}(\varepsilon^{-2}\log(1/\varepsilon))$ to obtain an $\varepsilon$-stationary point. Zhang et al. \cite{zhang2022primal,Zhang24} proposed a primal-dual alternating proximal gradient (PDAPG) algorithm and a primal-dual proximal gradient (PDPG-L) algorithm  for solving nonsmooth nonconvex-(strongly) concave and nonconvex-linear minimax problems \eqref{problem:1}, respectively. The iteration complexity of the two algorithms are proved to be $\mathcal{O}\left( \varepsilon ^{-2} \right)$ (resp. $\mathcal{O}\left( \varepsilon ^{-4} \right)$) under nonconvex-strongly concave (resp. nonconvex-concave) setting and $\mathcal{O}\left( \varepsilon ^{-3} \right)$ under nonconvex-linear setting to obtain an $\varepsilon$-stationary point, respectively. Lu et al. \cite{lu23} proposed a first-order augmented Lagrangian method which achieves the complexity bound of $\mathcal{O}(\varepsilon^{-4}\log(1/\varepsilon))$ for finding an $\varepsilon$-KKT solution of the nonconvex-concave constrained minimax problems.
	
	{\color{black} 	In this paper, we focus on {\color{black}zeroth-order} algorithms to solve black-box \eqref{problem:1} and \eqref{problem:s} problems.
		``Zeroth-order optimization" is closely related to the well-established field of ``derivative-free optimization" (DFO), which has indeed been studied for decades \cite{Audet,Conn,Rios}. Traditional DFO methods are broadly categorized into direct search-based and model-based approaches \cite{Liu20}. Our proposed algorithms belong to the class of zeroth-order methods, which solve optimization problems without explicit first-order or higher-order derivatives (e.g., gradients or Hessians). Instead, these methods approximate gradients using only function evaluations, often employing techniques like finite differences. The motivation for applying zeroth-order methods to machine learning (ML) problems, even when derivatives theoretically exist, stems primarily from practical scenarios involving black-box models or settings, including adversarial attacks\cite{chen2017zoo} on deep neural networks, hyperparameter tuning\cite{snoek2012practical}, and data poisoning\cite{Huang,xu21zeroth} against logistic regression. In these critical ML applications, zeroth-order methods become powerful and practical tools precisely because they circumvent the need for internal derivative information.
		%
		%
	}
	
	We focus on zeroth-order algorithms for solving minimax optimization problems under the nonconvex settings. For problem \eqref{min}, Xu et al. \cite{xu21zeroth} proposed a zeroth-order alternating randomized gradient projection (ZO-AGP) algorithm for smooth nonconvex-concave minimax problems, which can obtain the iteration complexity $\mathcal{O}(\varepsilon^{-4})$, and the number of function value estimation is bounded by $\mathcal{O}(d_x+d_y)$ per iteration.
	Shen et al. \cite{Shen} proposed another zeroth-order alternating randomized gradient projection algorithm for smooth nonconvex-linear minimax problems, which can obtain the iteration complexity $\mathcal{O}(\varepsilon^{-3})$, and the number of function value estimation per iteration is bounded by $\mathcal{O}(\varepsilon^{-2})$. Xu et al. \cite{xu22zeroth} proposed a zeroth-order alternating gradient descent ascent (ZO-AGDA)
	algorithm and a zeroth-order variance reduced alternating gradient descent ascent (ZO-VRAGDA) algorithm for solving {\color{black}nonconvex-Polyak- \L ojasiewicz (NC-PL) minimax problem (the objective function $f(x, y)$ satisfies the PL condition with respect to $y$, i.e., $\forall x, y$, there exists $\mu_y>0$ such that $\left\|\nabla_y f(x, y)\right\|^2 \geqslant 2 \mu_y\left[\max _y f(x, y)-f(x, y)\right]$)} under the deterministic and the stochastic setting, respectively, which can obtain the iteration complexity $\mathcal{O}(\varepsilon^{-2})$ and $\mathcal{O}(\varepsilon^{-3})$, respectively. For more details, we refer to \cite{Beznosikov,Huang,Liu,Sadiev,Wang,Xu20}. For the stochastic setting that is problem \eqref{min-s}, few zeroth-order algorithms have been proposed to solve this class of problems. Liu et al. \cite{Liu} proposed a ZO-Min-Max algorithm for nonconvex-strongly concave problem with $\mathcal{O}(\kappa^{6}\varepsilon^{-6})$ iteration complexity. Wang et al. \cite{Wang} proposed a ZO-SGDMSA which can obtain 
	$\tilde{\mathcal{O}}(\kappa^{2}\varepsilon^{-4})$ iteration complexity. Huang et al. \cite{Huang} proposed an accelerated zeroth-order momentum descent ascent method for nonconvex-strongly concave problem, and they proved the iteration complexity is $\tilde{\mathcal{O}}(\kappa^{4.5}\varepsilon^{-3})$. An et al. \cite{An} proposed a zeroth-order gradient descent extragradient ascent (ZO-GDEGA) algorithm, and proved the iteration complexity is $\mathcal{O}(\varepsilon^{-8})$ for nonconvex-concave problem \eqref{min-s}.
	
	To the best of our knowledge, there is no zeroth-order algorithm with theoretical complexity guarantee for both  deterministic and stochastic nonconvex-(strongly) concave minimax problems with coupled linear constraints up to now.
	
	In this paper, we propose a zeroth-order primal-dual alternating projection gradient (ZO-PDAPG) algorithm and a zeroth-order regularized momentum primal-dual projected gradient algorithm (ZO-RMPDPG) for solving problem \eqref{problem:1} and problem \eqref{problem:s} under nonconvex-(strongly) concave settings, respectively. They are both single-loop algorithms. Moreover, we prove that the iteration complexity of the proposed ZO-PDAPG algorithm is upper bounded by $\mathcal{O}\left( \varepsilon ^{-2} \right)$ (resp. $\mathcal{O}\left( \varepsilon ^{-4} \right)$) for deterministic nonconvex-strongly concave (resp. nonconvex-concave) settings, the iteration complexity of the proposed ZO-RMPDPG algorithm is upper bounded by $\tilde{\mathcal{O}}\left( \varepsilon ^{-3} \right)$ (resp. $\tilde{\mathcal{O}}\left( \varepsilon ^{-6.5} \right)$) for stochastic nonconvex-strongly concave (resp. nonconvex-concave) settings. To the best of our knowledge, they are the first two zeroth-order algorithms with theoretically guaranteed iteration complexity results for these classes of minimax problems. {\color{black} It is worth noting that in the special case of $A=0$, $B=0$, and $c=0$, ZO-PDAPG algorithm attains the same complexity as ZO-GDEGA algorithm for deterministic nonconvex-strongly concave problems. Whereas, under the stochastic nonconvex-concave setting, ZO-RMPDPG algorithm achieves a tighter iteration complexity than the $\mathcal{O}(\varepsilon^{-8})$ bound of ZO-GDEGA algorithm, thereby establishing a superior performance guarantee.}
	
	{\color{black}It should be noted that the ZO-PDAPG algorithm is different from the alternating gradient projection (AGP) algorithm proposed in \cite{Xu}, which is a first-order algorithm for solving deterministic nonconvex minimax problems without coupled linear constraints. The ZO-PDAPG and ZO-RMPDPG algorithms proposed in this paper are zeroth-order algorithms for solving deterministic and stochastic nonconvex minimax problems with linear coupled constraints respectively. They are also different from the ZO-AGDA and ZO-VRAGDA algorithms proposed in \cite{xu22zeroth}, which are zeroth-order algorithms for solving deterministic and stochastic nonconvex minimax problems without coupled linear constraints, respectively. Compared with the ZO-VRAGDA algorithm, the ZO-RMPDPG algorithm proposed in this paper not only uses the variance reduction technique, but also adds a momentum step for acceleration, while the ZO-VRAGDA algorithm only uses the variance reduction technique. In addition, the ZO-AGDA and ZO-VRAGDA algorithms use uniform smooth zeroth-order gradient estimators, while the ZO-PDAPG and ZO-RMPDPG algorithms proposed in this paper use finite difference zeroth-order gradient estimators. For more detailed comparison information of the two proposed algorithms with existing related algorithms, please refer to {\color{black}Table} \ref{tab1}.
		\begin{table}[!ht]
			\centering
			\caption{Comparison of zeroth-order and first-order methods for minimax problems with coupled constraints and existing zeroth-order methods for nonconvex minimax problems without coupled constraints.}\label{tab1}
			\begin{threeparttable}
				\begin{tabular}{p{3cm}p{2cm}p{1cm}p{2cm}p{2cm}p{2cm}}
					\toprule[1.5pt]
					\textbf{Algorithm} &\textbf{Coupled Constraint}\tnote{1}  &\textbf{Order} &\textbf{Setting}\tnote{2}  &\textbf{Loop(s)}&\textbf{Complexity} \\ \midrule[1.5pt]
					ZO-AGP \cite{xu21zeroth} & N  &Zeroth& D/NC-C & Single-loop &$\mathcal{O}(\varepsilon^{-4})$ \\\midrule[0.5pt]
					ZO-Min-Max \cite{Liu} & N &Zeroth& S/NC-SC & Single-loop& $\mathcal{O}(\kappa^6\varepsilon^{-6})$\\\midrule[0.5pt]
					ZO-SGDMSA \cite{Wang} & N &Zeroth& S/NC-SC & Nested-loop & $\tilde{\mathcal{O}}(\kappa^2\varepsilon^{-4})$\\\midrule[0.5pt]
					Acc-ZOMDA \cite{Huang} & N &Zeroth& S/NC-SC & Single-loop& $\tilde{\mathcal{O}}(\kappa^{4.5}\varepsilon^{-3})$\\\midrule[0.5pt]
					\multirow{2}*{ZO-GDEGA \cite{An}} & \multirow{2}*{N} &\multirow{2}*{Zeroth}& D/NC-SC  & \multirow{2}*{Single-loop}& $\mathcal{O}(\kappa^2\varepsilon^{-2})$ \\
					& & &S/NC-C& &$\mathcal{O}(\varepsilon^{-8})$\\
					\midrule[1.5pt]
					MGD \cite{tsaknakis2021minimax} & Y &  First & D/SC-SC & Nested-loop & $\tilde{\mathcal{O}}\left( \varepsilon ^{-2} \right)$ \\\midrule[0.5pt]
					PGmsAD \cite{dai2022optimality} & Y &  First & D/C-C & Nested-loop & $\tilde{\mathcal{O}}\left( \varepsilon ^{-2} \right)$ \\\midrule[0.5pt]
					\multirow{2}*{PDAPG \cite{zhang2022primal}} & \multirow{2}*{Y} &  \multirow{2}*{First} & D/NC-SC & \multirow{2}*{Single-loop} & $\mathcal{O}\left( \kappa^2\varepsilon ^{-2} \right)$ \\
					&  &   & D/NC-C &  &$\mathcal{O}(\varepsilon^{-4})$\\ \midrule[0.5pt]
					Algorithm in \cite{lu23} & Y &  First & D/NC-C & Nested-loop & $\tilde{\mathcal{O}}\left( \varepsilon ^{-4} \right)$ \\\midrule[1.5pt]
					\multirow{2}*{\textbf{ZO-PDAPG}}& \multirow{2}*{Y}&  \multirow{2}*{Zeroth} & D/NC-SC & \multirow{2}*{Single-loop} & $\mathcal{O}(\kappa^2\varepsilon^{-2})$\\
					&  &   & D/NC-C &  &$\mathcal{O}(\varepsilon^{-4})$\\\midrule[0.5pt]
					\multirow{2}*{\textbf{ZO-RMPDPG}}& \multirow{2}*{Y}&  \multirow{2}*{Zeroth} & S/NC-SC & \multirow{2}*{Single-loop} & $\tilde{\mathcal{O}}\left(\kappa^{4.5} \varepsilon ^{-3} \right)$\\
					&  &   & S/NC-C &  &$\tilde{\mathcal{O}}\left( \varepsilon ^{-6.5} \right)$\\
					\bottomrule[1.5pt]
				\end{tabular}
				\begin{tablenotes}
					\footnotesize
					\item[1] In the ``Couple Constraint" column, N means ``No Coupling Constraint" and Y means ``With Coupling Constraint".
					\item[2] In the ``Setting" column, V in V/W indicates whether the problem is stochastic or deterministic (V=D: deterministic; V=S: stochastic), and W indicates the convexity of $f(x,y)$ (W=NC-SC for nonconvex-strongly concave; W=NC-C for nonconvex-concave; W=SC-SC for strongly convex-strongly concave; W=C-C for convex-concave).
				\end{tablenotes}
			\end{threeparttable}
		\end{table}
	}
	
	\textbf{Notations}
	For vectors, we use $\|\cdot\|$ to represent the Euclidean norm and its induced matrix norm; $\left\langle x,y \right\rangle$ denotes the inner product of two vectors of $x$ and $y$. $[x]_i$ denotes the $i$th component of vector $x$. Let $\mathbb{R}^p$ denote the Euclidean space of dimension $p$ and $\mathbb{R}_{+}^p$ denote the nonnegative orthant in $\mathbb{R}^p$. We use $\nabla_{x} f(x,y,z)$ (or $\nabla_{y} f(x, y,z)$, $\nabla_{z} f(x, y,z)$) to denote the partial derivative of $f(x,y,z)$ with respect to $x$ (or $y$, $z$) at point $(x, y,z)$, respectively.
	We use the notation $\mathcal{O} (\cdot)$ to hide only absolute constants which do not depend on any problem parameter, and $\tilde{\mathcal{O}}(\cdot)$ notation to hide only absolute constants and log factors. {\color{black}$\mathbb{E}\|\cdot\|^{2}$ represents the expected value of the squared norm.} A continuously differentiable function $f(\cdot)$ is called $L$-smooth if there exists a constant $L> 0$ such
	that for any given $x,y \in \mathcal{X}$,
	\begin{equation*}
		\|\nabla f(x)-\nabla f(y)\|\le L\|x-y\|.
	\end{equation*}
	A continuously differentiable function $f(\cdot)$ is called $\mu$-strongly concave if there exists a constant $\mu> 0$ such
	that for any $x,y \in \mathcal{X}$,
	\begin{equation*}
		f(y)\le f(x)+\langle \nabla f(x),y-x \rangle-\frac{\mu}{2}\|y-x\|^2.
	\end{equation*}
	
	The paper is organized as follows. In Section \ref{sec2}, we propose a zeroth-order primal-dual alternating projection gradient (ZO-PDAPG) algorithm for solving deterministic nonconvex-(strongly) concave minimax problem with coupled linear constraints, and then prove its iteration complexity. In Section \ref{sec3}, we propose a zeroth-order regularized momentum primal-dual projected gradient algorithm (ZO-RMPDPG) algorithm for stochastic nonconvex-(strongly) concave minimax problem with coupled linear constraints, and also establish its iteration complexity. Numerical results in Section \ref{sec4} show the efficiency of the two proposed algorithms. Some conclusions are made in the last section.
	
	\section{Deterministic Nonconvex-(Strongly) Concave Minimax Problems with Coupled Linear Constraints.}\label{sec2}
	By using the Lagrangian function of problem \eqref{problem:1}, we obtain the dual problem of \eqref{problem:1}. Firstly, we make the following assumption.
	\begin{assumption}\label{fea}
		When $\unlhd$ is $\le$ in \eqref{problem:1}, for every $x\in\mathcal{X}$, there exists a vector $y\in relint(\mathcal{Y})$ such that $Ax+By-c\le0$, where $relint(\mathcal{Y})$ is the set of all relative interior points of $\mathcal{Y}$.
	\end{assumption}
	
	By strong duality shown in {\color{black}Theorem 1} in \cite{zhang2022primal}, instead of solving \eqref{problem:1}, we solve the following dual problem, i.e.,
	\begin{align}\label{problem:2}
		\min_{\lambda\in\Lambda} \min_{x\in\mathcal{X}}\max_{y\in\mathcal{Y}}\left\{\mathcal{L}(x,y,\lambda)\right\},\tag{D}
	\end{align}
	where $$\mathcal{L}(x,y,\lambda):=f(x,y)-\lambda^\top(Ax+By-c),$$ and $\Lambda=\mathbb{R}_{+}^p$ if $\unlhd$ is $\le$, $\Lambda=\mathbb{R}^p$ if $\unlhd$ is $=$.

	{\color{black}For completeness, we give the strong duality theorem in \cite{zhang2022primal} below.
		\begin{theorem}(Theorem 1 in \cite{zhang2022primal})\label{dual}
			Suppose $f(x,y)$ is a concave function with respect to $y$, $\mathcal{Y}$ is a convex and compact set. Then the strong duality of problem \eqref{problem:1} with respect to $y$ holds, i.e.,
			\begin{align*}
				&\min_{x\in\mathcal{X}}\max_{\substack{y\in\mathcal{Y}\\Ax+By\unlhd c}}f(x,y)
				=\min_{\lambda\in\Lambda}\min_{x\in\mathcal{X}}\max_{y\in\mathcal{Y}}\mathcal{L}(x,y,\lambda).
			\end{align*}
	\end{theorem}}
	{\color{black}
		\begin{proof}
			By Proposition 5.3.1 in \cite{Bertsekas}, we first obtain
			\begin{align}\label{dual4}
				\max_{y\in\mathcal{Y}}\min_{\lambda\in\Lambda}\mathcal{L}(x,y,\lambda)
				=&\max_{\substack{y\in\mathcal{Y}\\Ax+By\unlhd  c}}f(x,y),
			\end{align}
			which further implies that 
			\begin{align}
				\min_{x\in\mathcal{X}}\max\limits_{\substack{ y\in\mathcal{Y}\\Ax+By\unlhd c}}f(x,y)=\min\limits_{x\in\mathcal{X}}\max\limits_{y\in\mathcal{Y}}\min\limits_{\lambda\in\Lambda}\mathcal{L}(x,y,\lambda).\label{thm2.2:1}
			\end{align}
			Moreover, since $f(x,y)$ is a concave function with respect to $y$, by Sion’s minimax theorem (Corollary 3.3 in \cite{Sion}), $\forall x\in\mathcal{X}$, we have
			\begin{align}\label{dual:5}
				&\max_{y\in\mathcal{Y}}\min_{\lambda\in\Lambda}\mathcal{L}(x,y,\lambda)
				=\min_{\lambda\in\Lambda}\max_{y\in\mathcal{Y}}\mathcal{L}(x,y,\lambda).
			\end{align}
			The proof is then completed by combining \eqref{dual4} and \eqref{dual:5}.
		\end{proof}
	}
	
	\subsection{A Zeroth-Order primal-dual Alternating Projection Gradient Algorithm}
	In this subsection, we propose a zeroth-order primal-dual alternating projection gradient (ZO-PDAPG) algorithm for solving \eqref{problem:1}. Based on the idea of the PDAPG algorithm in \cite{zhang2022primal}, at each iteration of the proposed ZO-PDAPG algorithm, it performs two projection ``gradient'' steps for a regularized version of $\mathcal{L}(x,y,\lambda)$, i.e., 
	\begin{align}
		\tilde{\mathcal{L}}_k(x,y,\lambda)&=\mathcal{L}(x,y,\lambda)-\frac{\rho_k}{2}\|y\|^2,\label{zosc:1}
	\end{align}
	to update $y$ and $x$, where the ``gradient" of function $f\left(x,y\right)$ at $k$th iteration is computed by zeroth-order gradient estimators : $\mathbb{R}^{d_{x}} \times \mathbb{R}^{d_{y}} \rightarrow \mathbb{R}$, which are defined as
	\begin{align}
		\widehat{\nabla}_{x} f_k\left(x,y\right) &= \sum_{i=1}^{d_x} \frac{\left[f\left(x+\theta_{1,k} u_{i},y\right) - f(x,y)\right]}{\theta_{1,k}} u_{i}, \label{x_esti}\\
		\widehat{\nabla}_{y} f_k\left(x,y\right) &= \sum_{i=1}^{d_y} \frac{\left[f(x,y+\theta_{2,k} v_{i})-f(x,y)\right]}{\theta_{2,k}} v_{i},\label{y_esti}
	\end{align}
	where  $\theta_{1,k}, \theta_{2,k} > 0$ are smoothing parameters, $\{u_{i}\}_{i=1}^{d_x}$ is a standard basis in $\mathbb{R}^{d_{x}}$ {\color{black}(coordinate direction)}, and $\{v_{i}\}_{i=1}^{d_{y}}$ is a standard basis in $\mathbb{R}^{d_{y}}$. This means that for each $u_{i}$, its $i$th component is 1, while all other components are 0, and the same is true for each $v_{i}$. Define the following projection operator,
		\begin{align*}
			\mathcal{P}_{\mathcal{Z}}(\upsilon)&:=\arg\min\limits_{z\in \mathcal{Z}} \|z-\upsilon \|^2,
	\end{align*}
	where $\mathcal{Z}$ is a convex compact set. Then, at the $k$th iteration of the proposed algorithm, the update for $y_k$ is as follows,
	\begin{align*}
		y_{k+1}&={\arg\max}_{y\in \mathcal{Y}}\langle  \widehat{\nabla} _y \tilde{\mathcal{L}}_k\left( x_{k},y_k,\lambda_k \right),y-y_k \rangle  -\frac{\beta}{2}\| y-y_k \| ^2\nonumber\\
		&= {\color{black}\mathcal{P}_{\mathcal{Y}}}\left( y_k+\frac{1}{\beta} \widehat{\nabla}_y \tilde{\mathcal{L}}_k\left( x_{k},y_k,\lambda_k \right) \right),
	\end{align*}
	where $	\widehat{\nabla}_{y} \tilde{\mathcal{L}}_k\left(x,y,\lambda\right) = \widehat{\nabla}_{y} f_k\left(x,y\right) -B^\top\lambda-\rho_ky$ and $\beta >0$ is a parameter which will be defined later.
	The update for $x_k$ is as follows,
	\begin{align*}
		x_{k+1}&={\arg\min}_{x\in \mathcal{X}}\langle  \widehat{\nabla} _x \tilde{\mathcal{L}}_k\left( x_{k},y_{k+1},\lambda_k \right) ,x-x_k \rangle  +\frac{\alpha_k}{2}\| x-x_k \| ^2\nonumber\\
		&= {\color{black}\mathcal{P}_{\mathcal{X}}}\left( x_k - \frac{1}{\alpha_k} \widehat{\nabla} _x\tilde{\mathcal{L}}_k\left( x_{k},y_{k+1},\lambda_k \right) \right),
	\end{align*}
	where $	\widehat{\nabla}_{x} \tilde{\mathcal{L}}_k(x,y,\lambda) =  \widehat{\nabla}_{x} f_k(x,y) -A^\top\lambda$ and $\alpha_k > 0$ is a parameter which will be defined later.
	The update for $\lambda_k$ is as follows,
	\begin{align*}
		\lambda_{k+1}&=\mathcal{P}_{\Lambda}\left(\lambda_k-\gamma_k(Ax_{k+1}+By_{k+1}-c)\right).
	\end{align*}
	
	To ensure that the step size in the algorithm is well-defined, we first make the following assumptions about the smoothness of $f(x,y)$.
	\begin{assumption}\label{zoass:Lip}
		{\color{black}The gradients $\nabla_{x} f(x,y)$ and $\nabla_{y} f(x,y)$ are $L$-Lipschitz continuous with respect to $x$ and $y$, respectively, i.e.}
		\begin{align*}
			\| \nabla_{x} f\left(x_1, y\right)-\nabla_{x} f\left(x_2, y\right)\| &\leq L\|x_{1}-x_{2}\|,\\
			\|\nabla_{x} f\left(x, y_1\right)-\nabla_{x} f\left(x, y_2\right)\| &\leq L\|y_{1}-y_{2}\|,\\
			\|\nabla_{y} f\left(x, y_1\right)-\nabla_{y} f\left(x, y_2\right)\| &\leq L\|y_{1}-y_{2}\|,\\
			\|\nabla_{y} f\left(x_1, y\right)-\nabla_{y} f\left(x_2, y\right)\| &\leq L\|x_{1}-x_{2}\|.
		\end{align*}
	\end{assumption}
	The proposed ZO-PDAPG algorithm is formally stated in Algorithm \ref{zoalg:1}.
	\begin{algorithm}[t]
		\caption{A zeroth-order primal-dual alternating projection gradient algorithm (ZO-PDAPG) }
		\label{zoalg:1}
		\begin{algorithmic}
			\STATE{\textbf{Step 1}: Input $x_1,y_1,\lambda_1,{\color{black}\beta>3L}$, {\color{black}$\mu$}; Set $k=1$.}
			{\color{black}\STATE{\textbf{Step 2}: Compute $\alpha_k$ and $\gamma_k$:
					\STATE\quad  \textbf{(a)}: \textbf{If} {\color{black}$f(x, \cdot)$ is strongly concave ($\mu>0$)}, \textbf{then choose} 
					\begin{align*}
						\alpha_k=\alpha>5L+\frac{7L(\mu+2\beta)^2}{\mu^2}+\frac{L^2}{\mu},\quad
						\frac{1}{\gamma_k}=\frac{1}{\gamma}>\frac{10\|B\|^2(\mu+2\beta)^2}{L\mu^2}+\frac{L^2}{\mu}+L.
					\end{align*}
					\STATE\quad  \textbf{(b)}: \textbf{If} {\color{black}$f(x, \cdot)$ is simply concave ($\mu=0$)}, \textbf{then choose} 
					\begin{align*}
						\alpha_k&\ge16Lk^{1/2}+31L,\quad
						\frac{1}{\gamma_k}\ge\frac{\|B\|^2(12k^{1/2}+21)}{L}+9Lk^{1/2}+15L.
			\end{align*}}}
			\STATE{\textbf{Step 3}: Perform the following update for $y_k$:  	
				\qquad \begin{equation}
					y_{k+1}={\color{black}\mathcal{P}_{\mathcal{Y}}}\left( y_k+\frac{1}{\beta} \widehat{\nabla} _y \tilde{\mathcal{L}}_k\left( x_{k},y_k,\lambda_k \right) \right).\label{zoupdate-y}
			\end{equation}}		
			\STATE{\textbf{Step 4}: Perform the following update for $x_k$:
				\qquad	\begin{equation}
					x_{k+1}={\color{black}\mathcal{P}_{\mathcal{X}}} \left( x_k - \frac{1}{{\alpha_k}}  \widehat{\nabla} _x\tilde{\mathcal{L}}_k\left( x_{k},y_{k+1},\lambda_k \right) \right).\label{zoupdate-x}
			\end{equation}}
			\STATE{\textbf{Step 5}: Perform the following update for $\lambda_k$:
				\qquad	\begin{equation}
					\lambda_{k+1}=\mathcal{P}_{\Lambda}(\lambda_k+\gamma_k(Ax_{k+1}+By_{k+1}-c)).\label{zoupdate-lambda}
			\end{equation}}
			\STATE{\textbf{Step 6}: If some stationary condition is satisfied, stop; otherwise, set $k=k+1, $ go to Step 2.}
		\end{algorithmic}
	\end{algorithm}
	
	Note that although problem \eqref{problem:2} is a nonconvex-concave minimax problem, ZO-PDAPG cannot be regarded as a special case of ZO-BAPG in \cite{xu21zeroth} since that $\Lambda$ in \eqref{problem:2} is not a compact set, which is a necessary condition to ensure the convergence of ZO-BAPG. This is also the main difficulty in analyzing the ZO-PDAPG algorithm. In order to overcome this difficulty, we need to construct a new potential function. 
	
		For some given $\alpha>0$, $\beta>0$ and $\gamma>0$, denote 	\begin{equation*}
			\nabla \hat{\mathcal{G}}^{\alpha,\beta,\gamma}\left( x,y,\lambda \right) :=\left( \begin{array}{c}
				\alpha\left( x-{\color{black}\mathcal{P}_{\mathcal{X}}}\left( x-\frac{1}{\alpha}\nabla _x\mathcal{L}\left( x,y,\lambda  \right) \right) \right)\\
				\beta\left( y-{\color{black}\mathcal{P}_{\mathcal{Y}}}\left( y+ \frac{1}{\beta} \nabla _y\mathcal{L}\left( x,y,\lambda \right) \right) \right)\\
				\frac{1}{\gamma}\left( \lambda-\mathcal{P}_{\Lambda}\left(\lambda-\gamma\nabla_{\lambda}\mathcal{L}(x,y,\lambda)\right)\right)\\
			\end{array} \right) .
		\end{equation*}
		We define the stationarity gap as the termination criterion as follows.
		\begin{definition}\label{zogap-f}
			For given $\alpha>0$, $\beta>0$, and $\gamma>0$, we call a point $(x,y)$ a stationary point of problem \eqref{problem:1} if there exists a vector $\lambda\in\Lambda$ such that $\nabla \hat{\mathcal{G}}^{\alpha,\beta,\gamma}\left( x,y,\lambda \right)=0$. In addition, for any given $\varepsilon>0$, we call a point $(x,y)$ a {\color{black}$\varepsilon$-stationary} point of problem \eqref{problem:1} if there exists a vector $\lambda\in\Lambda$ such that $\|\nabla \hat{\mathcal{G}}^{\alpha,\beta,\gamma}\left( x,y,\lambda \right)\|\le \varepsilon$.
		\end{definition}
	
	Similar to Definition 2 in \cite{lu23}, below we give the definitions of KKT points and $\varepsilon$-KKT points for \eqref{problem:1}.
		
		(i)When $\unlhd$ is $=$, if there exists $(x^*,y^*,\lambda^*)\in\mathcal{X}\times\mathcal{Y}\times\mathbb{R}^p$ such that the following three equations hold: \begin{align}
			x^*-\mathcal{P}_{\mathcal{X}}\left( x^*-\frac{1}{\alpha}\nabla _x\mathcal{L}\left( x^*,y^*,\lambda^* \right) \right)&=0,\label{sec1:x}\\
			y^*-\mathcal{P}_{\mathcal{Y}}\left( y^*+ \frac{1}{\beta} \nabla _y\mathcal{L}\left( x^*,y^*,\lambda^* \right) \right)&=0,\label{sec1:y}\\
			Ax^*+By^*-c^*&=0;\nonumber
		\end{align}
		When $\unlhd$ is $\le$, if there exists $(x^*,y^*,\lambda^*)\in\mathcal{X}\times\mathcal{Y}\times\mathbb{R}_{+}^p$, so that \eqref{sec1:x}, \eqref{sec1:y} and the following two inequalities hold:
		\begin{align*}
			\langle\lambda^*,Ax^*+By^*-c\rangle&=0,\\
			Ax^*+By^*-c&\le0,
		\end{align*}
		We call the point $(x^*,y^*,\lambda^*)$  is a KKT point of problem \eqref{problem:1}.
		
		(ii) When $\unlhd$ is $=$, if there exists $(x^*,y^*,\lambda^*)\in\mathcal{X}\times\mathcal{Y}\times\mathbb{R}^p$ such that the following three inequalities hold:
		\begin{align*}
			\left\|\alpha\left( x^*-\mathcal{P}_{\mathcal{X}}\left( x^*-\frac{1}{\alpha}\nabla _x\mathcal{L}\left( x^*,y^*,\lambda^*  \right) \right) \right)\right\|&\le \varepsilon,\\
			\left\|\beta\left( y^*-\mathcal{P}_{\mathcal{Y}}\left( y^*+ \frac{1}{\beta} \nabla _y\mathcal{L}\left( x^*,y^*,\lambda^* \right) \right) \right)\right\|&\le \varepsilon,\\
			{\color{black}\|Ax^*+By^*-c^*\|}&\le\varepsilon,
		\end{align*}
		where $[\cdot]_+=\mathcal{P}_{\mathbb{R}_{+}^p}(\cdot)$. When $\unlhd$ is $\le$, if there exists $(x^*,y^*,\lambda^*)\in\mathcal{X}\times\mathcal{Y}\times\mathbb{R}_{+}^p$, such that the first two conditions above and the following two inequalities hold:
			\begin{align*}
				\|[Ax^*+By^*-c^*]_{+}\|&\le\varepsilon,\\
				|\langle\lambda^*,Ax^*+By^*-c\rangle|&\le \varepsilon,
		\end{align*}
		we call the point $(x^*,y^*,\lambda^*)$ is an $\varepsilon$-KKT of problem \eqref{problem:1}.
	
	Note that a stationary point of problem \eqref{problem:1} is also a KKT point of problem \eqref{problem:1}. Moreover, if $\unlhd$ is $=$, $\|\nabla {\color{black}\hat{\mathcal{G}}}^{\alpha,\beta,\gamma}(x,y,\lambda)\|\leq\varepsilon$ implies that  $\|Ax+By-c\|\leq \varepsilon$, which means that an $\varepsilon$-stationary point of problem \eqref{problem:1} is also an $\varepsilon$-KKT point of problem \eqref{problem:1}. 
	{\color{black}We give the following lemma to prove that if $\|\nabla \hat{\mathcal{G}}^{\alpha,\beta,\gamma}(x,y,\lambda)\|\leq\varepsilon$, then the constraint violation at $(x,y,\lambda)$ also satisfies the given accuracy.}
	\begin{lemma}\label{lem2.1}
		Suppose $\Lambda=\mathbb{R}_{+}^p$. If $\|\nabla {\color{black}\hat{\mathcal{G}}}^{\alpha,\beta,\gamma}(x,y,\lambda)\|\leq\varepsilon$ and $\lambda \geq 0$, for any  $i\in\mathcal{K}$, we have 
		\begin{align*}
			\max\{0,[Ax+By-c]_i\}\leq \varepsilon.
		\end{align*}
		{\color{black}Suppose $\Lambda=\mathbb{R}^p$. If $\|\nabla {\color{black}\hat{\mathcal{G}}}^{\alpha,\beta,\gamma}(x,y,\lambda)\| \leq \varepsilon$, then we have $\|A x+B y-c\| \leq \varepsilon$.}
	\end{lemma}
	{\color{black}
		\begin{proof}
			If $\|\nabla \hat{\mathcal{G}}^{\alpha,\beta,\gamma}( x,y,\lambda)\|\le\varepsilon$, by the definiton of $\nabla \hat{\mathcal{G}}^{\alpha,\beta,\gamma}( x,y,\lambda)$, we immediately have
			\begin{align}
				\frac{1}{\gamma}\|\lambda
				-\mathcal{P}_{\Lambda}\left(\lambda-\gamma\nabla_{\lambda}\mathcal{L}(x,y,\lambda)\right)\|\le\varepsilon.\label{lem2.1:2}
			\end{align}
			If $[Ax+By-c]_i\le 0$, then $\max\{0,[Ax+By-c]_i\}=0\leq \varepsilon$. Otherwise, if $[Ax+By-c]_i> 0$, by $\Lambda=\mathbb{R}_{+}^p$, then we have 
			\begin{align}
				\frac{1}{\gamma}\left|[\lambda -\mathcal{P}_{\Lambda}\left(\lambda-\gamma\nabla_{\lambda}\mathcal{L}(x,y,\lambda)\right)]_i\right|
				=&\frac{1}{\gamma}|[\lambda]_i-\max\{0,[\lambda+\gamma(Ax+By-c)]_i\}|\nonumber\\
				=&\frac{1}{\gamma}|[\lambda]_i-[\lambda]_i-\gamma[Ax+By-c]_i|
				=[Ax+By-c]_i.\label{lem2.1:4}
			\end{align}
			Combining \eqref{lem2.1:2}  and \eqref{lem2.1:4}, we obviously obtain $\max\{0,[Ax+By-c]_i\}\le\varepsilon$. Suppose $\Lambda=\mathbb{R}^p$, if $\|\nabla {\color{black}\hat{\mathcal{G}}}^{\alpha,\beta,\gamma}(x,y,\lambda)\| \leq \varepsilon$,  by the definiton of $\nabla \hat{\mathcal{G}}^{\alpha,\beta,\gamma}( x,y,\lambda)$, we immediately have $\|A x+B y-c\| \leq \varepsilon$, which completes the proof.
	\end{proof}}

	We provide an upper bound on the variance of the {\color{black}zeroth-order} gradient estimators as follows.
	\begin{lemma}{\color{black}(Lemma 2.3 in \cite{xu21zeroth})} \label{varboundlem}
		Suppose that Assumption \ref{zoass:Lip} holds. We have  
		\begin{align*}
			\left\|\widehat{\nabla}_{x}f_k\left(x,y\right) - \nabla_{x}f\left(x,y\right)\right\|^{2} \leq \frac{d_xL^2\theta_{1,k}^2}{4},\\
			\left\|\widehat{\nabla}_{y}f_k\left(x,y\right) - \nabla_{y}f\left(x,y\right)\right\|^{2} \leq \frac{d_yL^2\theta_{2,k}^2}{4}.
		\end{align*}
	\end{lemma}
	

	\subsection{Complextiy Analysis: Nonconvex-Strongly Concave Setting}
	In this subsection, we prove the iteration complexity of Algorithm \ref{zoalg:1} under the nonconvex-strongly concave setting, i.e., $f(x,y)$ is $\mu$-strongly concave with respect to $y$ for any given $x\in\mathcal{X}$. Under this setting, $\forall k \ge 1$, we set
	\begin{equation}\label{2.2par}
		\alpha_k=\alpha,\quad \gamma_k=\gamma, \quad \rho_k=0, \quad \theta_{1,k}=\theta_1, \quad\theta_{2,k}=\theta_2.
	\end{equation}
	{\color{black}Here, $1/\alpha$ and $\gamma$ are step size parameters, $\alpha$ is expected to take a larger value, $\gamma$ is expected to take a smaller value to ensure the convergence of the algorithm, $\theta_{1}, \theta_{2}$ are smoothing parameters. All these parameters are constants in the nonconvex-strongly concave setting, and are defined in detail in Theorem \ref{zothm1}.
		Since $\rho_k=0$, $\mathcal{L}(x,y,\lambda)=\tilde{\mathcal{L}}_k(x,y,\lambda)$. We will use the following simplified notations:} $\widehat{\nabla}_{x}f(x,y)=\widehat{\nabla}_{x}f_k(x,y)$, $\widehat{\nabla}_{y}f(x,y)=\widehat{\nabla}_{y}f_k(x,y)$, $\widehat{\nabla}_{x}\mathcal{L}(x,y,\lambda)=\widehat{\nabla}_{x}\tilde{\mathcal{L}}_k(x,y,\lambda)$, $\widehat{\nabla}_{y}\mathcal{L}(x,y,\lambda)=\widehat{\nabla}_{y}\tilde{\mathcal{L}}_k(x,y,\lambda) $. {\color{black} Here we drop the subscript $k$, since these are constants in the nonconvex-strongly concave setting, independent of $k$.}
	Let
	\begin{align*}
		\Phi(x,\lambda):=\max_{y\in\mathcal{Y}}\mathcal{L}(x,y,\lambda),\quad
		y^*(x,\lambda):=\arg\max_{y\in\mathcal{Y}}\mathcal{L}(x,y,\lambda).
	\end{align*}
	By Lemma {\color{black}B.1} in \cite{nouiehed2019solving} and the $\mu$-strong concavity of $\mathcal{L}(x,y,\lambda)$ with respect to $y$, $\Phi(x,\lambda)$ is $L_\Phi$-Lipschitz smooth with $L_\Phi=L+\frac{L^2}{\mu}$, and by Lemma {\color{black}23} in \cite{lin2020near},  for any given $x$ and $\lambda$ we have
	\begin{align}
		\nabla_x\Phi(x,\lambda)&=\nabla_x \mathcal{L}(x, y^*(x,\lambda),\lambda),\label{zogradx:nsc}\\
		\nabla_\lambda\Phi(x,\lambda)&=\nabla_\lambda \mathcal{L}(x, y^*(x,\lambda),\lambda)\label{zogradla:nsc}.
	\end{align}
	
	\begin{lemma}({\color{black}Lemma 1} in \cite{zhang2022primal})\label{zolem:eb}
		Suppose that Assumption \ref{zoass:Lip} holds and $f(x,y)$ is $\mu$-strongly concave with respect to $y$. Let $\eta=\frac{(2\beta+\mu)(\beta+L)}{\mu\beta}$, then
		\begin{align}
			\|y-y^*(x,\lambda)\|\le\eta\left\| y-{\color{black}\mathcal{P}_{\mathcal{Y}}}\left( y+ \frac{1}{\beta} \nabla _y\mathcal{L}\left( x,y,\lambda \right) \right)\right\|.\label{zolem:eb:1}
		\end{align}
	\end{lemma}
	{\color{black}\begin{proof}
			From the strong convexity of $\varphi(x, y, \lambda):=-\mathcal{L}(x,y,\lambda)$ and Theorem 4 in \cite{Necoara}, we can see that there exists a non-empty set $S$ consisting of the minimum points of $y$ in $\mathcal{Y}$, which satisfies the quadratic growth condition: $\varphi(x, y, \lambda)\ge\varphi^*+\frac{\mu}{2} \mbox{dist}^2(y,S)$, where $\mbox{dist}(y,S):=\inf_{z\in S}\|z-y\|$, and $\varphi^*=\min_{y\in \mathcal{Y}} \varphi(x, y, \lambda)$. Therefore, by the fact that the gradient of $-\mathcal{L}(x,y,\lambda)$ is $L$-Lipschitz continuous, and according to Corollary 3.6 in \cite{Drusvyatskiy}, the proof is complete.
	\end{proof}}

	Next, we provide an upper bound estimate of the difference between $\Phi(x_{k+1},\lambda_{k+1})$ and $\Phi(x_k,\lambda_k)$.
	\begin{lemma}\label{zolem1}
		Suppose that Assumption \ref{zoass:Lip} holds. Let $\{\left(x_k,y_k,\lambda_k\right)\}$ be a sequence generated by Algorithm \ref{zoalg:1} with parameter settings in \eqref{2.2par}, {\color{black} $\eta=\frac{(2\beta+\mu)(\beta+L)}{\mu\beta}$,}
		then $\forall k \ge 1$, for any {\color{black} $C_1>0$, $C_2>0$ and $C_3>0$},
		\begin{align}
			&\Phi(x_{k+1},\lambda_{k+1}) -\Phi(x_k,\lambda_k)\nonumber \\
			\leq& \langle  \widehat{\nabla} _x\mathcal{L}( x_k,y_{k+1},\lambda_k),x_{k+1}-x_k \rangle+\frac{(L+\beta)^2\eta^2}{2\beta^2}\left(\frac{2L^2}{{\color{black} C_1}}+\frac{3\|B\|^2}{{\color{black} C_3}} \right)\|y_{k+1}-y_k\|^2 \nonumber \\
			&+\left(\frac{3\|B\|^2L^2\eta^2}{2{\color{black} C_3}\beta^2}+\frac{{\color{black} C_1+C_2}+L_\Phi}{2}\right)\|x_{k+1}-x_k\|^2+\frac{ {\color{black} C_3}+L_\Phi}{2}\|\lambda_{k+1}-\lambda_k\|^2\nonumber \\
			&+\langle \nabla _\lambda \mathcal{L}(x_{k+1},y_{k+1},\lambda_k) ,\lambda_{k+1}-\lambda_k \rangle+\frac{3\|B\|^2d_yL^2\theta_2^2\eta^2}{8{\color{black} C_3}\beta^2}+\frac{d_yL^4\theta_2^2\eta^2}{4\beta^2{\color{black} C_1}}+\frac{d_xL^2\theta_1^2}{8{\color{black} C_2}}.\label{zolem1:iq1}
		\end{align}
	\end{lemma}
	
	\begin{proof}
		Since that $\Phi(x, \lambda)$ is $L_\Phi$-smooth with respect to $x$ and by \eqref{zogradx:nsc}, we have that
		\begin{align} 
			\Phi(x_{k+1},\lambda_{k})-\Phi(x_{k},\lambda_k)
			\le &\langle \nabla _x\Phi(x_{k},\lambda_k) ,x_{k+1}-x_k \rangle  +\frac{L_\Phi}{2}\|x_{k+1}-x_k\|^2\nonumber\\
			=&\langle \nabla _x\mathcal{L}(x_{k},y^*(x_k,\lambda_k),\lambda_k) ,x_{k+1}-x_k \rangle  +\frac{L_\Phi}{2}\|x_{k+1}-x_k\|^2\nonumber\\
			=&\langle \nabla _x\mathcal{L}(x_{k},y^*(x_k,\lambda_k),\lambda_k)-\nabla _x\mathcal{L}(x_{k},y_{k+1},\lambda_k) ,x_{k+1}-x_k \rangle\nonumber\\
			&+\langle \nabla _x\mathcal{L}(x_{k},y_{k+1},\lambda_{k})-\widehat{\nabla} _x\mathcal{L}(x_{k},y_{k+1},\lambda_k) ,x_{k+1}-x_k \rangle\nonumber\\
			&+\langle \widehat{\nabla} _x\mathcal{L}(x_{k},y_{k+1},\lambda_k) ,x_{k+1}-x_k \rangle+\frac{L_\Phi}{2}\|x_{k+1}-x_k\|^2.\label{zolem1:2}
		\end{align}
		Firstly, we estimate the first term in the r.h.s of \eqref{zolem1:2}. By the Cauchy-Schwarz inequality and Assumption \ref{zoass:Lip}, for any ${\color{black} C_1}>0$, we have
		\begin{align}
			&\langle \nabla _x\mathcal{L}(x_{k},y^*(x_k,\lambda_k),\lambda_k)-\nabla _x\mathcal{L}(x_{k},y_{k+1},\lambda_k) ,x_{k+1}-x_k \rangle\nonumber\\  
			=&\langle \nabla _xf(x_{k},y^*(x_k,\lambda_k))-\nabla _xf(x_{k},y_{k+1}) ,x_{k+1}-x_k \rangle\nonumber\\
			\le&\frac{L^2}{2{\color{black} C_1}}\|y_{k+1}-y^*(x_k,\lambda_k)\|^2+\frac{{\color{black} C_1}}{2}\|x_{k+1}-x_k\|^2.\label{zolem1:3}
		\end{align}
		By \eqref{zolem:eb:1} in Lemma \ref{zolem:eb}, we further have
		\begin{align}
			&\|y_{k+1}-y^*(x_k,\lambda_k)\|
			\le \eta\left\| y_{k+1}-{\color{black}\mathcal{P}_{\mathcal{Y}}}\left( y_{k+1}+ \frac{1}{\beta} \nabla _y\mathcal{L}\left( x_k,y_{k+1},\lambda_k \right) \right)\right\|.\label{zolem1:4}
		\end{align}
		Moreover, by \eqref{zoupdate-y}, the nonexpansive property of the projection operator ${\color{black}\mathcal{P}_{\mathcal{Y}}}(\cdot)$, Assumption \ref{zoass:Lip} and Lemma \ref{varboundlem}, we obtain
		\begin{align}
			&\left\| y_{k+1}-{\color{black}\mathcal{P}_{\mathcal{Y}}}\left( y_{k+1}+ \frac{1}{\beta} \nabla _y\mathcal{L}\left( x_k,y_{k+1},\lambda_k \right) \right)\right\|\nonumber\\
			= &\left\| {\color{black}\mathcal{P}_{\mathcal{Y}}}\left( y_{k}+ \frac{1}{\beta} \widehat{\nabla} _y\mathcal{L}\left( x_k,y_{k},\lambda_k \right) \right)-{\color{black}\mathcal{P}_{\mathcal{Y}}}\left( y_{k+1}+ \frac{1}{\beta} \nabla _y\mathcal{L}\left( x_k,y_{k+1},\lambda_k \right) \right)\right\|\nonumber\\
			\le&\|y_{k+1}-y_k\|+\frac{1}{\beta}\|\widehat{\nabla} _y\mathcal{L}\left( x_k,y_{k},\lambda_k \right)-\nabla _y\mathcal{L}\left( x_k,y_{k},\lambda_k \right)\|+\frac{1}{\beta}\|\nabla_y\mathcal{L}\left( x_k,y_{k},\lambda_k \right)-\nabla _y\mathcal{L}\left( x_k,y_{k+1},\lambda_k \right)\|\nonumber\\
			=&\|y_{k+1}-y_k\|+\frac{1}{\beta}\|\widehat{\nabla} _yf\left( x_k,y_{k}\right)-\nabla _yf\left( x_k,y_{k} \right)\|+\frac{1}{\beta}\|\nabla _yf\left( x_k,y_{k}\right)-\nabla _yf\left( x_k,y_{k+1} \right)\|\nonumber\\
			\le&\frac{\beta+L}{\beta}\|y_{k+1}-y_k\|+\frac{\sqrt{d_y}L\theta_2}{2\beta}.\label{zolem1:5}
		\end{align}
		Combing \eqref{zolem1:3}, \eqref{zolem1:4}, \eqref{zolem1:5} and using the fact $(a+b)^2\le 2a^2+2b^2$, we get
		\begin{align} 
			&\langle \nabla _x\mathcal{L}(x_{k},y^*(x_k,\lambda_k),\lambda_k)-\nabla _x\mathcal{L}(x_{k},y_{k+1},\lambda_k) ,x_{k+1}-x_k \rangle\nonumber\\
			\le&\frac{L^2\eta^2(L+\beta)^2}{{\color{black} C_1}\beta^2}\|y_{k+1}-y_k\|^2+\frac{{\color{black} C_1}}{2}\|x_{k+1}-x_k\|^2+\frac{d_yL^4\theta_2^2\eta^2}{4\beta^2{\color{black} C_1}}.\label{zolem1:3:2}
		\end{align}
		Next, we estimate the second term in the r.h.s. of \eqref{zolem1:2} as follows. By the Cauchy-Schwarz inequality and Lemma \ref{varboundlem}, for any ${\color{black} C_2}>0$, we have
		\begin{align}
			\langle \nabla _x\mathcal{L}(x_{k},y_{k+1},\lambda_{k})-\widehat{\nabla} _x\mathcal{L}(x_{k},y_{k+1},\lambda_k) ,x_{k+1}-x_k \rangle
			=&\langle \nabla _xf(x_{k},y_{k+1})-\widehat{\nabla} _xf(x_{k},y_{k+1}) ,x_{k+1}-x_k \rangle\nonumber\\
			\le&\frac{{\color{black} C_2}}{2}\|x_{k+1}-x_k\|^2+\frac{d_xL^2\theta_1^2}{8{\color{black} C_2}}.\label{zolem1:3:2.2}
		\end{align}
		Plugging \eqref{zolem1:3:2} and \eqref{zolem1:3:2.2} into \eqref{zolem1:2}, we have
		\begin{align}
			\Phi(x_{k+1},\lambda_{k})-\Phi(x_{k},\lambda_k)
			\le&\frac{L^2\eta^2(L+\beta)^2}{{\color{black} C_1}\beta^2}\|y_{k+1}-y_k\|^2+\frac{{\color{black} C_1+C_2}+L_\Phi}{2}\|x_{k+1}-x_k\|^2\nonumber\\
			&+\langle \widehat{\nabla} _x\mathcal{L}(x_{k},y_{k+1},\lambda_k) ,x_{k+1}-x_k \rangle+\frac{d_yL^4\theta_2^2\eta^2}{4\beta^2{\color{black} C_1}}+\frac{d_xL^2\theta_1^2}{8{\color{black} C_2}}.\label{zolem1:6}
		\end{align}
		On the other hand, $\Phi(x, \lambda)$ is $L_\Phi$-smooth with respect to $\lambda$ and \eqref{zogradla:nsc}, we have
		\begin{align} 
			\Phi(x_{k+1},\lambda_{k+1})-\Phi(x_{k+1},\lambda_k)
			\le &\langle \nabla _\lambda\Phi(x_{k+1},\lambda_k) ,\lambda_{k+1}-\lambda_k \rangle  +\frac{L_\Phi}{2}\|\lambda_{k+1}-\lambda_k\|^2\nonumber\\
			=&\langle \nabla _\lambda \mathcal{L}(x_{k+1},y^*(x_{k+1},\lambda_k),\lambda_k) ,\lambda_{k+1}-\lambda_k \rangle  +\frac{L_\Phi}{2}\|\lambda_{k+1}-\lambda_k\|^2\nonumber\\
			=&\langle \nabla _\lambda \mathcal{L}(x_{k+1},y^*(x_{k+1},\lambda_k),\lambda_k)-\nabla _\lambda \mathcal{L}(x_{k+1},y_{k+1},\lambda_k) ,\lambda_{k+1}-\lambda_k \rangle\nonumber\\
			&+\langle \nabla _\lambda \mathcal{L}(x_{k+1},y_{k+1},\lambda_k) ,\lambda_{k+1}-\lambda_k \rangle+\frac{L_\Phi}{2}\|\lambda_{k+1}-\lambda_k\|^2.\label{zolem1:7}
		\end{align}
		Next, we estimate the first term in the right hand side of \eqref{zolem1:7}. By the Cauchy-Schwarz inequality and the definition of $L(x,y,\lambda)$, for any ${\color{black} C_3}>0$, we get
		\begin{align}
			&\langle \nabla _\lambda \mathcal{L}(x_{k+1},y^*(x_{k+1},\lambda_k),\lambda_k)-\nabla _\lambda \mathcal{L}(x_{k+1},y_{k+1},\lambda_k) ,\lambda_{k+1}-\lambda_k \rangle\nonumber\\
			=&-\langle B(y^*(x_{k+1},\lambda_k)-y_{k+1}) ,\lambda_{k+1}-\lambda_k \rangle\nonumber\\
			\le&\frac{\|B\|^2}{2{\color{black} C_3}}\|y_{k+1}-y^*(x_{k+1},\lambda_k)\|^2+\frac{{\color{black} C_3}}{2}\|\lambda_{k+1}-\lambda_k\|^2.\label{zolem1:8}
		\end{align}
		By \eqref{zolem:eb:1} in Lemma \ref{zolem:eb}, we further have
		\begin{align}
			&\|y_{k+1}-y^*(x_{k+1},\lambda_k)\|
			\le \eta\left\| y_{k+1}-{\color{black}\mathcal{P}_{\mathcal{Y}}}\left( y_{k+1}+ \frac{1}{\beta} \nabla _y\mathcal{L}\left( x_{k+1},y_{k+1},\lambda_k \right) \right)\right\|.\label{zolem1:9}
		\end{align}
		By \eqref{zoupdate-y}, the nonexpansive property of the {\color{black}projection operator $\mathcal{P}_{\mathcal{Y}}(\cdot)$}, Assumption \ref{zoass:Lip} and Lemma \ref{varboundlem}, we obtain
		\begin{align}
			&\left\| y_{k+1}-{\color{black}\mathcal{P}_{\mathcal{Y}}}\left( y_{k+1}+ \frac{1}{\beta} \nabla _y\mathcal{L}\left( x_{k+1},y_{k+1},\lambda_k \right) \right)\right\|\nonumber\\
			= &\left\| {\color{black}\mathcal{P}_{\mathcal{Y}}}\left( y_{k}+ \frac{1}{\beta} \widehat{\nabla} _y\mathcal{L}\left( x_k,y_{k},\lambda_k \right) \right)-{\color{black}\mathcal{P}_{\mathcal{Y}}}\left( y_{k+1}+ \frac{1}{\beta} \nabla _y\mathcal{L}\left( x_{k+1},y_{k+1},\lambda_k \right) \right)\right\|\nonumber\\
			\le&\|y_{k+1}-y_k\|+\frac{1}{\beta}\|\nabla _y\mathcal{L}\left( x_{k+1},y_{k+1},\lambda_k \right)-\nabla _y\mathcal{L}\left( x_k,y_{k},\lambda_k \right)\|+\frac{1}{\beta}\|\nabla _y\mathcal{L}\left( x_{k},y_{k},\lambda_k \right)-\widehat{\nabla} _y\mathcal{L}\left( x_k,y_{k},\lambda_k \right)\|\nonumber\\
			=&\|y_{k+1}-y_k\|+\frac{1}{\beta}\|\nabla _yf\left( x_{k+1},y_{k+1} \right)-\nabla _yf\left( x_k,y_{k} \right)\|+\frac{1}{\beta}\|\nabla _yf\left( x_{k},y_{k} \right)-\widehat{\nabla} _yf\left( x_k,y_{k} \right)\|\nonumber\\
			\le&\frac{L+\beta}{\beta}\|y_{k+1}-y_k\|+\frac{L}{\beta}\|x_{k+1}-x_k\|+\frac{\sqrt{d_y}L\theta_2}{2\beta}.\label{zolem1:10}
		\end{align}
		By combing \eqref{zolem1:9} and \eqref{zolem1:10} and using the fact $(a+b+c)^2\le 3a^2+3b^2+3c^2$, we get
		\begin{align}
			&\|y_{k+1}-y^*(x_{k+1},\lambda_k)\|^2
			\le \frac{3(L+\beta)^2\eta^2}{\beta^2}\|y_{k+1}-y_k\|^2+\frac{3L^2\eta^2}{\beta^2}\|x_{k+1}-x_k\|^2+\frac{3d_yL^2\theta_2^2\eta^2}{4\beta^2}.\label{zolem1:11}
		\end{align}
		By combining \eqref{zolem1:7}, \eqref{zolem1:8} and \eqref{zolem1:11}, we have
		\begin{align} 
			&\Phi(x_{k+1},\lambda_{k+1})-\Phi(x_{k+1},\lambda_k)\nonumber\\
			\le&\frac{3\|B\|^2(L+\beta)^2\eta^2}{2{\color{black} C_3}\beta^2}\|y_{k+1}-y_k\|^2
			+\frac{3\|B\|^2L^2\eta^2}{2{\color{black} C_3}\beta^2}\|x_{k+1}-x_k\|^2+\frac{3\|B\|^2d_yL^2\theta_2^2\eta^2}{8{\color{black} C_3}\beta^2}\nonumber\\
			&+\frac{{\color{black} C_3}+L_\Phi}{2}\|\lambda_{k+1}-\lambda_k\|^2+\langle \nabla _\lambda \mathcal{L}(x_{k+1},y_{k+1},\lambda_k) ,\lambda_{k+1}-\lambda_k \rangle.\label{zolem1:iq13}
		\end{align}
		The proof is then completed by adding \eqref{zolem1:6} and \eqref{zolem1:iq13}.
	\end{proof}

	Next, we provide an lower bound for the difference between $\mathcal{L}(x_{k+1},y_{k+1},\lambda_{k+1})$ and $\mathcal{L}(x_k,y_k,\lambda_k)$.
	\begin{lemma}\label{zolem2}
		Suppose that Assumption \ref{zoass:Lip} holds. Let $\{\left(x_k,y_k,\lambda_k\right)\}$ be a sequence generated by Algorithm \ref{zoalg:1} with parameter settings in \eqref{2.2par},
		then $\forall k \ge 1$,
		\begin{align}\label{zolem2:iq1}
			&\mathcal{L}(x_{k+1},y_{k+1},\lambda_{k+1})-\mathcal{L}( x_{k},y_{k},\lambda_k)\nonumber \\
			\geq&\left(\beta-L\right)\|y_{k+1}-y_k\|^2-L\| x_{k+1}-x_{k} \|^2+\langle \nabla _\lambda \mathcal{L}(x_{k+1},y_{k+1},\lambda_k),\lambda_{k+1}-\lambda_{k} \rangle\nonumber\\
			&+\langle \widehat{\nabla} _{x}\mathcal{L}(x_{k},y_{k+1},\lambda_k),x_{k+1}-x_{k} \rangle-\frac{d_xL\theta_1^2}{8}-\frac{d_yL\theta_2^2}{8}.
		\end{align}
	\end{lemma}
	
	\begin{proof}
		The optimality condition for {\color{black}$y_{k+1}$} in \eqref{zoupdate-y} implies that $\forall y\in \mathcal{Y}$ and $\forall k\geq 1$,
		\begin{equation}
			\langle \widehat{\nabla} _y\mathcal{L}(x_{k},y_k,\lambda_k)-\beta(y_{k+1}-y_k),y-y_{k+1} \rangle \le 0.\label{zolem2:2}
		\end{equation}
		By choosing $y=y_k$ in \eqref{zolem2:2}, we get
		\begin{equation}
			\langle \widehat{\nabla} _y\mathcal{L}(x_{k},y_k,\lambda_k),y_{k+1}-y_k \rangle \ge \beta\|y_{k+1}-y_k\|^2.\label{zolem2:3}
		\end{equation}
		By  Assumption \ref{zoass:Lip}, $\mathcal{L}(x,y,\lambda)$ has Lipschitz continuous gradient with respect to $y$, which implies that
		\begin{align}\label{zolem2:4}
			&\mathcal{L}(x_{k},y_{k+1},\lambda_k)-\mathcal{L}( x_{k},y_{k},\lambda_k)\nonumber\\
			\ge& \langle \nabla _{y}\mathcal{L}(x_{k},y_{k},\lambda_k),y_{k+1}-y_{k} \rangle -\frac{L}{2}\| y_{k+1}-y_{k} \|^2\nonumber\\
			=& \langle \nabla _{y}\mathcal{L}(x_{k},y_{k},\lambda_k)-\widehat{\nabla} _y\mathcal{L}(x_{k},y_k,\lambda_k),y_{k+1}-y_{k} \rangle -\frac{L}{2}\| y_{k+1}-y_{k} \|^2
			+\langle \widehat{\nabla} _y\mathcal{L}(x_{k},y_k,\lambda_k),y_{k+1}-y_k \rangle.
		\end{align}
		Next, we estimate the first term in the r.h.s. of \eqref{zolem2:4} as follows. By the Cauchy-Schwarz inequality and Lemma \ref{varboundlem}, we have
		\begin{align}
			&\langle \nabla _y\mathcal{L}(x_{k},y_{k},\lambda_{k})-\widehat{\nabla} _y\mathcal{L}(x_{k},y_{k},\lambda_k) ,y_{k+1}-y_k \rangle\nonumber\\
			=&\langle \nabla _yf(x_{k},y_{k})-\widehat{\nabla} _yf(x_{k},y_{k}) ,y_{k+1}-y_k \rangle
			\ge-\frac{L}{2}\|y_{k+1}-y_k\|^2-\frac{d_yL\theta_2^2}{8}.\label{zolem2:4.1}
		\end{align}
		By plugging \eqref{zolem2:3} and \eqref{zolem2:4.1} into \eqref{zolem2:4}, we get
		\begin{align}
			\mathcal{L}(x_{k},y_{k+1},\lambda_k)-\mathcal{L}( x_{k},y_{k},\lambda_k)
			\ge\left(\beta-L\right)\|y_{k+1}-y_k\|^2-\frac{d_yL\theta_2^2}{8}.\label{zolem2:4.2}
		\end{align}
		Similarly, the gradient of $\mathcal{L}(x,y,\lambda)$ is Lipschitz continuous with respect to $x$, which implies that
		\begin{align}
			&\mathcal{L}(x_{k+1},y_{k+1},\lambda_k)-\mathcal{L}( x_{k},y_{k+1},\lambda_k)\nonumber\\
			\ge& \langle \nabla _{x}\mathcal{L}(x_{k},y_{k+1},\lambda_k),x_{k+1}-x_{k} \rangle -\frac{L}{2}\| x_{k+1}-x_{k} \|^2\nonumber\\
			=&\langle \nabla _{x}\mathcal{L}(x_{k},y_{k+1},\lambda_k)-\widehat{\nabla} _{x}\mathcal{L}(x_{k},y_{k+1},\lambda_k),x_{k+1}-x_{k} \rangle-\frac{L}{2}\| x_{k+1}-x_{k} \|^2+ \langle \widehat{\nabla} _{x}\mathcal{L}(x_{k},y_{k+1},\lambda_k),x_{k+1}-x_{k} \rangle\nonumber\\
			\ge&-L\|x_{k+1}-x_k\|^2-\frac{d_xL\theta_1^2}{8}+ \langle \widehat{\nabla} _{x}\mathcal{L}(x_{k},y_{k+1},\lambda_k),x_{k+1}-x_{k} \rangle,\label{zolem2:5}
		\end{align}
		where the last inequality is by the Cauchy-Schwarz inequality and Lemma \ref{varboundlem}.
		On the other hand, it can be easily checked that
		\begin{align}
			&\mathcal{L}(x_{k+1},y_{k+1},\lambda_{k+1})-\mathcal{L}( x_{k+1},y_{k+1},\lambda_k)
			=\langle \nabla _\lambda \mathcal{L}(x_{k+1},y_{k+1},\lambda_k),\lambda_{k+1}-\lambda_{k} \rangle.\label{zolem2:iq6}
		\end{align}
		The proof is completed by adding \eqref{zolem2:4.2}, \eqref{zolem2:5} and \eqref{zolem2:iq6}.
	\end{proof}
	
	We now establish an important recursion for {\color{black}Algorithm} \ref{zoalg:1}.
	\begin{lemma}\label{zolem3}
		Suppose that Assumption \ref{zoass:Lip} holds. Denote
		\begin{align*}
			S(x,y,\lambda)&=2\Phi(x,\lambda)-\mathcal{L}(x,y,\lambda).
		\end{align*}
		Let $\{\left(x_k,y_k,\lambda_k\right)\}$ be a sequence generated by Algorithm \ref{zoalg:1} with parameter settings in \eqref{2.2par}, {\color{black} $\eta=\frac{(2\beta+\mu)(\beta+L)}{\mu\beta}$,} then $\forall k \geq 1$,
		\begin{align}
			&S(x_{k+1},y_{k+1},\lambda_{k+1})-S( x_{k},y_{k},\lambda_k)\nonumber \\
			\le&-\left(\alpha-\frac{L^3}{(L+\beta)^2}-\frac{2L(L+\beta)^2\eta^2}{\beta^2}
			-\frac{L^2}{\mu}-4L\right)\|x_{k+1}-x_k\|^2 \nonumber \\
			&-\left(\frac{1}{\gamma}-\frac{3\|B\|^2(L+\beta)^2\eta^2}{L\beta^2}-L-\frac{L^2}{\mu} \right)\|\lambda_{k+1}-\lambda_{k}\|^2-\left(\beta-3L\right)\|y_{k+1}-y_k\|^2\nonumber\\
			&+\frac{d_xL\theta_1^2}{4}+\frac{[(L+\beta)^2+4L^2]d_yL\theta_2^2}{8(L+\beta)^2}.\label{zolem3:1}
		\end{align}
	\end{lemma}
	
	\begin{proof}
		By \eqref{zolem1:iq1} and \eqref{zolem2:iq1}, we obtain
		\begin{align}
			&S(x_{k+1},y_{k+1},\lambda_{k+1})-S( x_{k},y_{k},\lambda_k)\nonumber \\
			\le&-\left(\beta-L-\frac{(L+\beta)^2\eta^2}{\beta^2}\left(\frac{2L^2}{{\color{black} C_1}}+\frac{3\|B\|^2}{{\color{black} C_3}} \right)\right)\|y_{k+1}-y_k\|^2 +\left(\frac{3\|B\|^2L^2\eta^2}{{\color{black} C_3}\beta^2}+{\color{black} C_1+C_2}+L_\Phi+L\right)\|x_{k+1}-x_k\|^2\nonumber \\
			&+\left({\color{black} C_3}+L_\Phi \right)\|\lambda_{k+1}-\lambda_k\|^2+\langle \widehat{\nabla} _x\mathcal{L}( x_k,y_{k+1},\lambda_k),x_{k+1}-x_k \rangle+\langle \nabla _\lambda \mathcal{L}(x_{k+1},y_{k+1},\lambda_k),\lambda_{k+1}-\lambda_{k} \rangle\nonumber\\
			&+\frac{3\|B\|^2d_yL^2\theta_2^2\eta^2}{4{\color{black} C_3}\beta^2}+\frac{d_yL^4\theta_2^2\eta^2}{2\beta^2{\color{black} C_1}}+\frac{d_xL^2\theta_1^2}{4a_2}+\frac{d_xL\theta_1^2}{8}+\frac{d_yL\theta_2^2}{8}.\label{zolem3:iq2}
		\end{align}
		The optimality condition for {\color{black}$\lambda_{k+1}$} in \eqref{zoupdate-lambda} implies that
		\begin{equation}\label{zolem3:iq3}
			\langle \nabla _\lambda \mathcal{L}(x_{k+1},y_{k+1},\lambda_k)+\frac{1}{\gamma}(\lambda_{k+1}-\lambda_{k}),\lambda-\lambda_{k+1} \rangle \ge 0.
		\end{equation}
		By choosing $\lambda=\lambda_k$ in \eqref{zolem3:iq3}, we have
		\begin{equation}\label{zolem3:iq4}
			\langle \nabla _\lambda \mathcal{L}(x_{k+1},y_{k+1},\lambda_k),\lambda_{k+1}-\lambda_k \rangle \le -\frac{1}{\gamma}\|\lambda_{k+1}-\lambda_k\|^2.
		\end{equation}
		The optimality condition for {\color{black}$x_{k+1}$} in \eqref{zoupdate-x} implies that $\forall x\in \mathcal{X}$ and $\forall k\geq 1$,
		\begin{equation}\label{zolem3:3}
			\langle \widehat{\nabla} _x\mathcal{L}(x_{k},y_{k+1},\lambda_k)+\alpha(x_{k+1}-x_k),x-x_{k+1} \rangle \ge 0.
		\end{equation}
		By choosing $x=x_k$ in \eqref{zolem3:3}, we have
		\begin{equation}\label{zolem3:4}
			\langle \widehat{\nabla} _x\mathcal{L}(x_{k},y_{k+1},\lambda_k),x_{k+1}-x_k \rangle \le -\alpha\|x_{k+1}-x_k\|^2.
		\end{equation}
		Plugging \eqref{zolem3:iq4} and \eqref{zolem3:4} into \eqref{zolem3:iq2}, we obtain
		\begin{align}
			&S(x_{k+1},y_{k+1},\lambda_{k+1})-S( x_{k},y_{k},\lambda_k)\nonumber \\
			\le&-\left(\beta-L-\frac{(L+\beta)^2\eta^2}{\beta^2}\left(\frac{2L^2}{{\color{black} C_1}}+\frac{3\|B\|^2}{{\color{black} C_3}} \right)\right)\|y_{k+1}-y_k\|^2 \nonumber \\
			&-\left(\alpha-\frac{3\|B\|^2L^2\eta^2}{{\color{black} C_3}\beta^2}-{\color{black} C_1-C_2}-L_\Phi-L\right)\|x_{k+1}-x_k\|^2\nonumber \\
			&-\left(\frac{1}{\gamma}- {\color{black} C_3}-L_\Phi \right)\|\lambda_{k+1}-\lambda_k\|^2+\frac{3\|B\|^2d_yL^2\theta_2^2\eta^2}{4{\color{black} C_3}\beta^2}+\frac{d_yL^4\theta_2^2\eta^2}{2\beta^2{\color{black} C_1}}+\frac{d_xL^2\theta_1^2}{4{\color{black} C_2}}+\frac{d_xL\theta_1^2}{8}+\frac{d_yL\theta_2^2}{8}.\label{zolem3:5}
		\end{align}
		The proof is then completed by choosing ${\color{black} C_1}=\frac{2L(L+\beta)^2\eta^2}{\beta^2}$, ${\color{black} C_2}=2L$ and ${\color{black} C_3}=\frac{3\|B\|^2(L+\beta)^2\eta^2}{L\beta^2}$ in \eqref{zolem3:5} and the definition of $L_\Phi=L+\frac{L^2}{\mu}$.
	\end{proof}
	
	Let $\nabla {\color{black}\hat{\mathcal{G}}}^{\alpha,\beta,\gamma}\left( x_k,y_k,\lambda_k \right)$ be defined as  in Definition \ref{zogap-f}, we provide an upper bound on $\| \nabla {\color{black}\hat{\mathcal{G}}}^{\alpha,\beta,\gamma}(x_k,y_k,\lambda_k)\|$ in the following lemma.

	\begin{lemma}\label{zolem4}
		Suppose that Assumption \ref{zoass:Lip} holds. Let $\{\left(x_k,y_k,\lambda_k\right)\}$ be a sequence generated by Algorithm \ref{zoalg:1} with parameter settings in \eqref{2.2par}. Then $\forall k \geq 1$,
		\begin{align}\label{zolem4:1}
			\|\nabla {\color{black}\hat{\mathcal{G}}}^{\alpha,\beta,\gamma}( x_k,y_{k},\lambda_k)\|^2
			\le& (2\beta^2+3L^{2}+3\|B\|^2)\|y_{k+1}-y_k\|^2+(3\alpha^2+3\|A\|^2)\|x_{k+1}-x_k\|^2\nonumber\\
			&+\frac{3}{\gamma^2}\|\lambda_{k+1}-\lambda_k\|^2+\frac{3d_xL^2\theta_1^2}{4}+\frac{d_yL^2\theta_2^2}{2}.
		\end{align}
	\end{lemma}
	
	\begin{proof}
		By \eqref{zoupdate-y}, the nonexpansive property of the projection operator and Lemma \ref{varboundlem}, we immediately get
		\begin{align}
			\left\| \beta\left( y_k-{\color{black}\mathcal{P}_{\mathcal{Y}}}\left( y_k+ \frac{1}{\beta} \nabla _y\mathcal{L}\left( x_k,y_k,\lambda_k\right) \right) \right) \right\|
			\le&\beta\|y_{k+1}-y_k\|+\|\widehat{\nabla} _y\mathcal{L}\left( x_k,y_{k},\lambda_k \right)-\nabla _y\mathcal{L}\left( x_k,y_{k},\lambda_k \right)\|\nonumber\\
			=&\beta\|y_{k+1}-y_k\|+\|\widehat{\nabla} _yf\left( x_k,y_{k}\right)-\nabla _yf\left( x_k,y_{k} \right)\|\nonumber\\
			\le&\beta\|y_{k+1}-y_k\|+\frac{\sqrt{d_y}L\theta_2}{2}.\label{zolem4:2}
		\end{align}
		On the other hand, by \eqref{zoupdate-x}, the nonexpansive property of the projection operator and the Cauchy-Schwartz inequality, we have
		\begin{align}
			&\left\|	\alpha ( x_k-{\color{black}\mathcal{P}_{\mathcal{X}}}( x_k-\frac{1}{\alpha}\nabla _x\mathcal{L}\left( x_k,y_k,\lambda_k \right) ) )\right\|\nonumber\\
			\le &\alpha \left\| {\color{black}\mathcal{P}_{\mathcal{X}}}( x_k-\frac{1}{\alpha} \widehat{\nabla }_x\mathcal{L}(x_{k},y_{k+1},\lambda_k) )-{\color{black}\mathcal{P}_{\mathcal{X}}}( x_k-\frac{1}{\alpha} \nabla _x\mathcal{L}(x_{k},y_{k},\lambda_k) )\right\|+\alpha\| x_{k+1}-x_k\|\nonumber\\
			\leq & \alpha\|x_{k+1}-x_k\| +\|\widehat{\nabla }_xf(x_{k},y_{k+1})-\nabla _xf(x_{k},y_{k+1})\|+\|\nabla_xf(x_{k},y_{k+1})-\nabla _xf(x_{k},y_{k})\|\nonumber\\
			\le&\alpha\| x_{k+1}-x_k\|+L\|y_{k+1}-y_k\|+\frac{\sqrt{d_x}L\theta_1}{2},\label{zolem4:3}
		\end{align}
		where the last inequaliy is by Assumption \ref{zoass:Lip} and Lemma \ref{varboundlem}. By \eqref{zoupdate-lambda}, the Cauchy-Schwartz inequality and the nonexpansive property of the projection operator,  we obtain
		\begin{align}
			&\left\|\frac{1}{\gamma}\left( \lambda_k -\mathcal{P}_{\Lambda}\left(\lambda_k-\gamma\nabla_{\lambda}\mathcal{L}(x_k,y_k,\lambda_k)\right)\right)\right\|\nonumber\\
			\le &\frac{1}{\gamma}\left\| \mathcal{P}_{\Lambda}\left(\lambda_k-\gamma\nabla_{\lambda}\mathcal{L}(x_{k+1},y_{k+1},\lambda_k)\right)-
			\mathcal{P}_{\Lambda}\left(\lambda_k-\gamma\nabla_{\lambda}\mathcal{L}(x_k,y_k,\lambda_k)\right)\right\|+\frac{1}{\gamma}\| \lambda_{k+1}-\lambda_k\|\nonumber\\
			\leq & \frac{1}{\gamma}\|\lambda_{k+1}-\lambda_k\|+\|A\|\|x_{k+1}-x_k\| +\|B\|\|y_{k+1}-y_k\|.\label{zolem4:4}
		\end{align}
		Combining \eqref{zolem4:2}, \eqref{zolem4:3} and \eqref{zolem4:4}, and using Cauchy-Schwarz inequality, we complete the proof.
	\end{proof}
	Define $T(\varepsilon):=\min\{k \mid \|\nabla {\color{black}\hat{\mathcal{G}}}^{\alpha,\beta,\gamma}(x_k,y_k,\lambda_k)\|\leq \varepsilon \}$ with $\varepsilon>0$ being a given target accuracy. We provide a bound on $T(\varepsilon)$ in the following theorem. 
	
	\begin{theorem}\label{zothm1}
		Suppose that Assumptions \ref{fea} and \ref{zoass:Lip} hold. Let $\{\left(x_k,y_k,\lambda_k\right)\}$ be a sequence generated by Algorithm \ref{zoalg:1} with parameter settings in \eqref{2.2par}. Let $\eta=\frac{(2\beta+\mu)(\beta+L)}{\mu\beta}$, $\theta_1=\frac{\varepsilon\sqrt{{\color{black}D_1}}}{\sqrt{d_xL(3{\color{black}D_1}L+1)}}$ and $\theta_2=\frac{\varepsilon\sqrt{{\color{black}D_1}}}{\sqrt{2{\color{black}D_1}d_yL^2+\frac{[(L+\beta)^2+4L^2]d_yL}{2(L+\beta)^2}}}$ {\color{black}with $D_1=\frac{\min\left\{\alpha-\frac{L^3}{(L+\beta)^2}-\frac{2L(L+\beta)^2\eta^2}{\beta^2}
				-\frac{L^2}{\mu}-4L,\ \beta-3L,\ \frac{1}{\gamma}-\frac{3\|B\|^2(L+\beta)^2\eta^2}{L\beta^2}-L-\frac{L^2}{\mu}\right\} }{\max \{ 2\beta^2+3L^{2}+3\|B\|^2,\ 3\alpha^2+3\|A\|^2,\ 3/\gamma^2\} }$}. If
		\begin{align}
			&\beta>3L,\quad \alpha>\frac{L^3}{(L+\beta)^2}+\frac{2L(L+\beta)^2\eta^2}{\beta^2}
			+\frac{L^2}{\mu}+4L,
			\frac{1}{\gamma}>\left(\frac{3\|B\|^2(L+\beta)^2\eta^2}{L\beta^2}+L+\frac{L^2}{\mu} \right),\label{zothm1:1}
		\end{align}
		then $\forall \varepsilon>0$, we have {\color{black}$T\left( \varepsilon \right)$ is well-defined and}
		$$ T\left( \varepsilon \right) \le \frac{{\color{black}D_2}}{\varepsilon ^2 {\color{black}D_1}},$$
		where ${\color{black}D_2}:=S(x_{1},y_{1},\lambda_{1})-\underbar{S}$ with $\underbar{S}:=\min_{{\color{black}\lambda\in\Lambda}}\min_{x\in\mathcal{X}}\min_{y\in\mathcal{Y}}S(x,y,\lambda)$.
		Moreover, for any $i\in\mathcal{K}$, {\color{black}when $k= T(\varepsilon)$}, we have $$\max\{0,[Ax_{k}+By_{k}-c]_i\}\leq \varepsilon.$$
	\end{theorem}
	
	\begin{proof}
		By \eqref{zothm1:1}, it can be easily checked that ${\color{black}D_1}>0$. By multiplying ${\color{black}D_1}$ on the both sides of \eqref{zolem4:1}, using {\color{black}Lemma \ref{zolem3}} and the settings of $\theta_1$, $\theta_2$, we get
		\begin{align}
			&{\color{black}D_1}\|\nabla {\color{black}\hat{\mathcal{G}}}^{\alpha,\beta,\gamma}( x_k,y_{k},\lambda_k)\|^2\nonumber\\
			\le& S(x_{k},y_{k},\lambda_{k})-S(x_{k+1},y_{k+1},\lambda_{k+1})+\frac{3{\color{black}D_1}d_xL^2\theta_1^2}{4}+\frac{{\color{black}D_1}d_yL^2\theta_2^2}{2}+\frac{d_xL\theta_1^2}{4}+\frac{[(L+\beta)^2+4L^2]d_yL\theta_2^2}{8(L+\beta)^2}\nonumber\\
			=&S(x_{k},y_{k},\lambda_{k})-S(x_{k+1},y_{k+1},\lambda_{k+1})+\frac{{\color{black}D_1}\varepsilon^2}{2}.\label{zothm1:5}
		\end{align}
		{\color{black}Next we {\color{black}prove} $\min\limits_{\lambda\in\Lambda}\min\limits_{x\in\mathcal{X}}\min\limits_{y\in\mathcal{Y}}S(x,y,\lambda)>-\infty$. 
			By the definition of $\Phi(x,\lambda)$ and $S(x,y,\lambda)$ in Lemma \ref{zolem3}, we have
			$
			S(x,y,\lambda)
			=2\max\limits_{y\in\mathcal{Y}}\mathcal{L}(x,y,\lambda)-\mathcal{L}(x,y,\lambda)
			\ge\max\limits_{y\in\mathcal{Y}}\mathcal{L}(x,y,\lambda).
			$
			Then, we immediately get
			\begin{align}
				\min_{\lambda\in\Lambda}\min_{x\in\mathcal{X}}\min_{y\in\mathcal{Y}}S(x,y,\lambda)
				\ge&\min_{\lambda\in\Lambda}\min_{x\in\mathcal{X}}\min_{y\in\mathcal{Y}}
				\left\{\max_{y\in\mathcal{Y}}\mathcal{L}(x,y,\lambda)\right\}\nonumber\\
				=&\min_{\lambda\in\Lambda}\min_{x\in\mathcal{X}}\max_{y\in\mathcal{Y}}\left\{\mathcal{L}(x,y,\lambda)\right\}\nonumber\\
				=&\min_{x\in\mathcal{X}}\max_{\substack{y\in\mathcal{Y}\\Ax+By\unlhd c}}f(x,y) >-\infty,\label{thm1:6.5}
			\end{align}
			where the second equality is by Theorem \ref{dual}, and the last inequality follows from the assumption that $\mathcal{X}$ and $\mathcal{Y}$ are convex and compact sets.} Denote $\underbar{S}=\min\limits_{\lambda\in\Lambda}\min\limits_{x\in\mathcal{X}}\min\limits_{y\in\mathcal{Y}}S(x,y,\lambda)$.
		Then,
		{\color{black}we prove that $T(\varepsilon)$ is well-defined by contradiction.
			Suppose that there is no $k$ such that $\|\nabla \hat{\mathcal{G}}^{\alpha,\beta,\gamma}( x_k,y_{k},\lambda_k)\| \leq \varepsilon$, that is, $\|\nabla \hat{\mathcal{G}}^{\alpha,\beta,\gamma}( x_k,y_{k},\lambda_k)\| > \varepsilon$ for all $k\ge 1$. Then, for any $T\ge 1$, summing \eqref{zothm1:5} from $k=1$ to $T$, we have
			\begin{align}
				&\sum_{k=1}^{T}{D_1\|\nabla \hat{\mathcal{G}}^{\alpha,\beta,\gamma}( x_k,y_{k},\lambda_k) \| ^2}\nonumber\\
				\le& S(x_{1},y_{1},\lambda_{1})-S(x_{T+1},y_{T+1},\lambda_{T+1})+\frac{D_1\varepsilon^2}{2}T
				\le S(x_{1},y_{1},\lambda_{1})- \underbar{S}+\frac{D_1\varepsilon^2}{2}T.\label{t:1}
			\end{align}
			Thus, we get
			\begin{align}
				D_1\varepsilon^2<&\frac{1}{T}\sum_{k=1}^{T}{D_1\|\nabla \hat{\mathcal{G}}^{\alpha,\beta,\gamma}( x_k,y_{k},\lambda_k) \| ^2}
				\le \frac{S(x_{1},y_{1},\lambda_{1})- \underbar{S}}{T}+\frac{D_1\varepsilon^2}{2}.\label{t:2}
			\end{align}
			Letting $T\rightarrow \infty$, the above inequality implies that $\varepsilon^2<\frac{\varepsilon^2}{2}$, which is a contradiction. Therefore, there must exist a finite $k$ such that $\|\nabla \hat{\mathcal{G}}^{\alpha,\beta,\gamma}( x_k,y_{k},\lambda_k)\| \leq \varepsilon$, i.e., $T(\varepsilon)$ is well-defined.}
		Summing \eqref{zothm1:5} from $k=1$ to $T(\varepsilon)$, we then obtain
		\begin{align}
			&\sum_{k=1}^{T\left( \varepsilon \right)}{{\color{black}D_1}\|\nabla {\color{black}\hat{\mathcal{G}}}^{\alpha,\beta,\gamma}( x_k,y_{k},\lambda_k) \| ^2}
			\le S(x_{1},y_{1},\lambda_{1})-S(x_{T(\varepsilon)+1},y_{T(\varepsilon)+1},\lambda_{T(\varepsilon)+1})+\frac{{\color{black}D_1}\varepsilon^2}{2}T(\varepsilon).\label{zothm1:6}
		\end{align}
		By \eqref{zothm1:6}, we obtain
		\begin{align*}
			\sum_{k=1}^{T\left( \varepsilon \right)}{{\color{black}D_1}\|\nabla {\color{black}\hat{\mathcal{G}}}^{\alpha,\beta,\gamma}( x_k,y_{k},\lambda_k) \| ^2}\le S(x_{1},y_{1},\lambda_{1})-\underbar{S}+\frac{{\color{black}D_1}\varepsilon^2}{2}T(\varepsilon)={\color{black}D_2}+\frac{{\color{black}D_1}\varepsilon^2}{2}T(\varepsilon).
		\end{align*}
		In view of the definition of $T(\varepsilon)$, {\color{black}we have $	\sum_{k=1}^{T\left( \varepsilon \right)} D_{1}\|\nabla \hat{\mathcal{G}}^{\alpha,\beta,\gamma}( x_k,y_{k},\lambda_k) \|^{2} \geq T(\varepsilon) D_{1} \varepsilon^{2}$}, the above inequality implies that
		$\frac{\varepsilon ^2}{2}\le \frac{{\color{black}D_2}}{T( \varepsilon ){\color{black}D_1}}$ or equivalently, $T\left( \varepsilon \right) \le \frac{2{\color{black}D_2}}{\varepsilon ^2 {\color{black}D_1}}$. 
		Next, by Lemma \ref{lem2.1}, if $\|\nabla {\color{black}\hat{\mathcal{G}}}^{\alpha,\beta,\gamma}( x_k,y_{k},\lambda_k)\|\le\varepsilon$, we have $\max\{0,[Ax_{k}+By_{k}-c]_i\}\le\varepsilon$ which completes the proof. 
	\end{proof}
	
	\begin{remark}
		Denote $\kappa=L/\mu$. If we set $\beta=4L$, by Theorem \ref{zothm1} we have that $\eta=\mathcal{O}(\kappa)$, $\alpha=\mathcal{O}(\kappa^2)$ and $1/\gamma=\mathcal{O}(\kappa^2)$, which implies that the number of iterations for Algorithm \ref{zoalg:1} to obtain an $\varepsilon$-stationary point of \eqref{problem:1} is bounded by $\mathcal{O}\left(\kappa^2\varepsilon ^{-2} \right)$ and the total number of function value queries are bounded by $\mathcal{O}\left((d_x+d_y)\kappa^2\varepsilon ^{-2} \right)$ for solving the deterministic nonconvex-strongly concave minimax problem with coupled linear constraints. 
	\end{remark}
	
	\subsection{Complexity Analysis: Nonconvex-Concave Setting}
	In this subsection, we prove the iteration complexity of Algorithm \ref{zoalg:1} under the nonconvex-concave setting. $\forall k\ge 1$, we first denote
	\begin{align}
		\Psi_k(x,\lambda)&=\max_{y\in\mathcal{Y}}\tilde{\mathcal{L}}_k(x,y,\lambda),\label{zosc:2}\\
		\tilde{y}_k^*(x,\lambda)&=\arg\max_{y\in\mathcal{Y}}\tilde{\mathcal{L}}_k(x,y,\lambda).\label{zosc:3}
	\end{align}
	{\color{black}
		By Lemma {\color{black}B.1} in \cite{nouiehed2019solving} and the $\rho_k$-strong concavity of $\tilde{\mathcal{L}}_k(x,y,\lambda)$ with respect to $y$, $\Psi_k(x,\lambda)$ is $L_{\Psi_k}$-Lipschitz smooth with $L_{\Psi_k}=L+\frac{L^2}{\rho_k}$, and by Lemma {\color{black}23} in \cite{lin2020near},  for any given $x$ and $\lambda$ we have
		\begin{align}
			\nabla_x\Psi_k(x,\lambda)&=\nabla_x \tilde{\mathcal{L}}_k(x, \tilde{y}_k^*(x,\lambda),\lambda),\label{zogradx:nc}\\
			\nabla_\lambda\Psi_k(x,\lambda)&=\nabla_\lambda \tilde{\mathcal{L}}_k(x, \tilde{y}_k^*(x,\lambda),\lambda)\label{zogradla:nc}.
	\end{align}}
	We also need to make the following assumption on the parameter $\rho_k$.
	\begin{assumption}\label{zomuk}
		$\{\rho_k\}$ is a nonnegative monotonically decreasing sequence.
	\end{assumption}
	
	
	\begin{lemma}\label{zosc:lem1}
		Suppose that {\color{black}Assumptions \ref{zoass:Lip} and \ref{zomuk} hold}. Denote
		\begin{align*}
			M_k(x,y,\lambda)&=2\Psi_k(x,\lambda)-\tilde{\mathcal{L}}_k(x,y,\lambda).
		\end{align*}
		Let $\{\left(x_k,y_k,\lambda_k\right)\}$ be a sequence generated by Algorithm \ref{zoalg:1}. Set $\eta_k=\frac{(2\beta+\rho_k)(\beta+L)}{\rho_k\beta}$. Then $\forall k \geq 1$,
		\begin{align}
			&M_{k+1}(x_{k+1},y_{k+1},\lambda_{k+1})-M_k( x_{k},y_{k},\lambda_k)\nonumber \\
			\le&-\left(\alpha_k-\frac{L^3}{(L+\beta)^2}-\frac{2L(L+\beta)^2\eta_k^2}{\beta^2}
			-\frac{L^2}{\rho_k}-4L\right)\|x_{k+1}-x_k\|^2 \nonumber \\
			&-\left(\frac{1}{\gamma_k}-\frac{3\|B\|^2(L+\beta)^2\eta_k^2}{L\beta^2}-L-\frac{L^2}{\rho_k} \right)\|\lambda_{k+1}-\lambda_{k}\|^2+\frac{d_xL\theta_{1,k}^2}{4}\nonumber\\
			&-\left(\beta-3L\right)\|y_{k+1}-y_k\|^2+(\rho_k-\rho_{k+1})\sigma_y^2+\frac{[(L+\beta)^2+4L^2]d_yL\theta_{2,k}^2}{8(L+\beta)^2},\label{zosc:lem1:1}
		\end{align}
		where $\sigma_y=\max\{\|y\| \mid y\in\mathcal{Y}\}$.
	\end{lemma}
	
	\begin{proof}
		Similar to the proof of Lemma \ref{zolem3}, by replacing {\color{black} $S$ with $M_k$, $\Phi$ with $\Psi_k$, $\mathcal{L}$ with $\tilde{\mathcal{L}}_k$, } $\alpha$ with $\alpha_k$, $\gamma$ with $\gamma_k$, $\mu$ with $\rho_k$, $\eta$ with $\eta_k$, $\theta_1$ with $\theta_{1,k}$, $\theta_2$ with $\theta_{2,k}$, respectively, we have
		\begin{align}
			&M_k(x_{k+1},y_{k+1},\lambda_{k+1})-M_k( x_{k},y_{k},\lambda_k)\nonumber \\
			\le&-\left(\alpha_k-\frac{L^3}{(L+\beta)^2}-\frac{2L(L+\beta)^2\eta_k^2}{\beta^2}
			-\frac{L^2}{\rho_k}-4L\right)\|x_{k+1}-x_k\|^2 \nonumber \\
			&-\left(\frac{1}{\gamma_k}-\frac{3\|B\|^2(L+\beta)^2\eta_k^2}{L\beta^2}-L-\frac{L^2}{\rho_k} \right)\|\lambda_{k+1}-\lambda_{k}\|^2\nonumber\\
			&-\left(\beta-3L\right)\|y_{k+1}-y_k\|^2+\frac{d_xL\theta_{1,k}^2}{4}+\frac{[(L+\beta)^2+4L^2]d_yL\theta_{2,k}^2}{8(L+\beta)^2}.\label{zosc:lem1:2}
		\end{align}
		On the other hand, by \eqref{zosc:1}, \eqref{zosc:2} and \eqref{zosc:3}, we obtain
		\begin{align}
			&M_{k+1}(x_{k+1},y_{k+1},\lambda_{k+1})-M_k( x_{k+1},y_{k+1},\lambda_{k+1})\nonumber \\
			=&2(\Psi_{k+1}(x_{k+1},\lambda_{k+1})-\Psi_{k}(x_{k+1},\lambda_{k+1}))-\tilde{\mathcal{L}}_{k+1}(x_{k+1},y_{k+1},\lambda_{k+1})+\tilde{\mathcal{L}}_k(x_{k+1},y_{k+1},\lambda_{k+1}))\nonumber\\
			=&2(\tilde{\mathcal{L}}_{k+1}(x_{k+1},\tilde{y}_{k+1}^*(x_{k+1},\lambda_{k+1}),\lambda_{k+1})-\tilde{\mathcal{L}}_{k}(x_{k+1},\tilde{y}_k^*(x_{k+1},\lambda_{k+1}),\lambda_{k+1}))+\frac{\rho_{k+1}-\rho_k}{2}\|y_{k+1}\|^2\nonumber\\
			\le&2(\tilde{\mathcal{L}}_{k+1}(x_{k+1},\tilde{y}_{k+1}^*(x_{k+1},\lambda_{k+1}),\lambda_{k+1})-\tilde{\mathcal{L}}_{k}(x_{k+1},\tilde{y}_{k+1}^*(x_{k+1},\lambda_{k+1}),\lambda_{k+1}))+\frac{\rho_{k+1}-\rho_k}{2}\|y_{k+1}\|^2\nonumber\\
			=&(\rho_k-\rho_{k+1})\|\tilde{y}_{k+1}^*(x_{k+1},\lambda_{k+1})\|^2+\frac{\rho_{k+1}-\rho_k}{2}\|y_{k+1}\|^2\nonumber\\
			\le&(\rho_k-\rho_{k+1})\sigma_y^2,\label{zosc:lem1:3}
		\end{align}
		where the last inequality holds since $\{\rho_k\}$ is a nonnegative monotonically decreasing sequence. The proof is completed by adding \eqref{zosc:lem1:2} and \eqref{zosc:lem1:3}.
	\end{proof}
	
	Similar to the proof of Theorem \ref{zothm1}, we then obtain the following theorem which provides a bound on $T(\varepsilon)$, where $T(\varepsilon):=\min\{k \mid \|\nabla {\color{black}\hat{\mathcal{G}}}^{\alpha_k,\beta,\gamma_k}(x_k,y_k,\lambda_k)\|\leq \varepsilon \}$ and $\varepsilon>0$ is a given target accuracy.

	\begin{theorem}\label{zosc:thm1}
		Suppose that Assumptions \ref{fea}, \ref{zoass:Lip} and \ref{zomuk} hold. Let $\{\left(x_k,y_k,\lambda_k\right)\}$ be a sequence generated by Algorithm \ref{zoalg:1}. $\forall k\geq 1$, set  {\color{black}$\alpha_k=\frac{L^3}{(L+\beta)^2}+\frac{4L(L+\beta)^4(2\beta+\rho_k)^2}{\beta^4\rho_k^2}
			+\frac{2L^2}{\rho_k}+4L$, $\frac{1}{\gamma_k}=\frac{[3\|B\|^2+2L^2](L+\beta)^4(2\beta+\rho_k)^2}{L\beta^4\rho_k^2}+L+\frac{2L^2}{\rho_k}$} with $\rho_k=\frac{2(L+\beta)}{k^{1/4}}$, and {\color{black}$\beta=4L$}. Then for any given $\varepsilon>0$, if we set
		\begin{align}
			\theta_{1,k}&=\frac{\varepsilon\beta^2}{Lk^{1/4}}\sqrt{\frac{C_2}{2d_x(3C_1L+1)}},\nonumber\\
			\theta_{2,k}&=\frac{\varepsilon\beta^2}{Lk^{1/4}}\sqrt{\frac{C_2}{d_y(1+4C_1L+\frac{4L^2}{(L+\beta)^2})}},\label{theta}
		\end{align}
		then
		{\color{black}$$T( \varepsilon)\le \frac{{\color{black}16118}\tilde{D}_3^2\tilde{D}_4^2L^2}{\varepsilon^4},$$}
		where  {\color{black}$C_1=\max\{\frac{\beta-3L}{2\beta^2+3L^{2}+3\|B\|^2}, \frac{2(L+\beta)}{L^2\max\{\tilde{D}_1, \tilde{D}_2\}}\}$, $C_2=\frac{L+\beta}{4(L+\beta)^3(L+2\beta)^2+L\beta^4}$, $\tilde{D}_1=\beta^4\frac{\frac{12L^6}{(L+\beta)^4}+18L^2+3\|A\|^2}{4L^2(L+\beta)^4}+\frac{3L^2\beta^4}{(L+\beta)^6}+12$, $\tilde{D}_2=\frac{9(3\|B\|^2+L^2)^2}{4L^4} + \frac{9\beta^4((L+\beta)^2+L^2)}{4(L+\beta)^6}$, $\tilde{D}_4=\max\left\{\frac{(2\beta^2+3L^{2}+3\|B\|^2)\rho_1}{(\beta-3L)L^2}, \tilde{D}_1, \tilde{D}_2\right\}$,} {\color{black}$\sigma_y=\max\{\|y\| \mid y\in\mathcal{Y}\}$} and $\tilde{{\color{black}D}}_3=M_{1}(x_{1},y_{1},\lambda_1)-\underline{M}+\rho_1\sigma_y^2$ with $\underline{M}:=\min\limits_{\lambda\in\Lambda}\min\limits_{x\in\mathcal{X}}\min\limits_{y\in\mathcal{Y}}M_k(x,y,\lambda)$.
	\end{theorem}
	
	\begin{proof}
		Firstly, we denote 	
		\begin{equation*}
			\nabla {\color{black}\bar{\mathcal{G}}}^{\alpha_k,\beta,\gamma_k}_k\left( x_k,y_k,\lambda_k \right) =\left( \begin{array}{c}
				\alpha_k\left( x_k-{\color{black}\mathcal{P}_{\mathcal{X}}}\left( x_k-\frac{1}{\alpha_k}\nabla _x\mathcal{L}\left( x_k,y_k,\lambda_k \right) \right) \right)\\
				\beta\left( y_k-{\color{black}\mathcal{P}_{\mathcal{Y}}}\left( y_k+ \frac{1}{\beta} \nabla _y\tilde{\mathcal{L}}_k\left( x_k,y_k,\lambda_k\right) \right) \right)\\
				\frac{1}{\gamma_k}\left( \lambda_k-\mathcal{P}_{\Lambda}\left(\lambda_k-\gamma_k\nabla_{\lambda}\mathcal{L}(x_k,y_k,\lambda_k)\right)\right)
			\end{array} \right).
		\end{equation*}
		It then can be easily checked that
		\begin{equation}
			\|\nabla {\color{black}\hat{\mathcal{G}}}^{\alpha_k,\beta,\gamma_k}( x_{k},y_{k},\lambda_{k}) \|\le\|\nabla {\color{black}\bar{\mathcal{G}}}^{\alpha_k,\beta,\gamma_k}_k\left( x_{k},y_{k},\lambda_{k} \right) \|+\rho_{k}\|y_{k}\|.\label{zosc:thm1:2}
		\end{equation}
		Next, we estimate the upper bound of $\|\nabla {\color{black}\bar{\mathcal{G}}}^{\alpha_k,\beta,\gamma_k}_k\left( x_k,y_k,\lambda_k \right)\|$. Similar to the proof of Lemma \ref{zolem4}, by replacing $\alpha$ with $\alpha_k$, $\gamma$ with $\gamma_k$, $\theta_1$ with $\theta_{1,k}$, $\theta_2$ with $\theta_{2,k}$, respectively, we have
		\begin{align}
			\|\nabla  {\color{black}\bar{\mathcal{G}}}^{\alpha_k,\beta,\gamma_k}_k( x_k,y_{k},\lambda_k)\|^2
			\le& (2\beta^2+3L^{2}+3\|B\|^2)\|y_{k+1}-y_k\|^2+(3\alpha_k^2+3\|A\|^2)\|x_{k+1}-x_k\|^2\nonumber\\
			&+\frac{3}{\gamma_k^2}\|\lambda_{k+1}-\lambda_k\|^2+\frac{3d_xL^2\theta_{1,k}^2}{4}+\frac{d_yL^2\theta_{2,k}^2}{2}.\label{zosc:thm1:3}
		\end{align}
		Let {\color{black}$\vartheta_k=\frac{2L(L+\beta)^2\eta_k^2}{\beta^2}
			+\frac{L^2}{\rho_k}$}, then we get that $\alpha_k-\frac{L^3}{(L+\beta)^2}-\frac{2L(L+\beta)^2\eta_k^2}{\beta^2}
		-\frac{L^2}{\rho_k}-4L=\vartheta_k$ and $\frac{1}{\gamma_k}-\frac{3\|B\|^2(L+\beta)^2\eta_k^2}{L\beta^2}-L-\frac{L^2}{\rho_k}=\vartheta_k$ by the settings of $\alpha_k$ and $\gamma_k$. In view of Lemma \ref{zosc:lem1}, then we immediately obtain
		\begin{align}
			&\vartheta_k\|x_{k+1}-x_{k}\|^2+\vartheta_k\|\lambda_{k+1}-\lambda_{k}\|^2+\left(\beta-3L\right)\| y_{k+1}-y_{k} \|^2\nonumber\\
			\le&M_{k}(x_{k},y_{k},\lambda_k)-M_{k+1}(x_{k+1},y_{k+1},\lambda_{k+1})+(\rho_k-\rho_{k+1})\sigma_y^2+\frac{d_xL\theta_{1,k}^2}{4}+\frac{[(L+\beta)^2+4L^2]d_yL\theta_{2,k}^2}{8(L+\beta)^2}.\label{zosc:thm1:4}
		\end{align}
		It follows from the definition of $\tilde{{\color{black}D}}_1$ that $\forall k\ge 1$,
		\begin{align*}
			\frac{3\alpha_k^2+3\|A\|^{2}}{\vartheta _k^2} 
			=&  \frac{3\left(\frac{L^3}{(L+\beta)^2}+\frac{2L(L+\beta)^2\eta_k^2}{\beta^2}
				+\frac{2 L^2}{\rho_k}+\frac{3L}{2}\right)^2+3\|A\|^{2}}{\left(\frac{2L(L+\beta)^2\eta_k^2}{\beta^2}
				+\frac{L^2}{\rho_k}\right)^2}\nonumber\\
			\leq& \frac{\frac{12L^6}{(L+\beta)^4}+\frac{48L^2(L+\beta)^4\eta_k^4}{\beta^4}+\frac{48L^4}{\rho_k^2}+18L^2+3\|A\|^2}{\frac{4L^2(L+\beta)^4\eta_k^4}{\beta^4}}\nonumber\\
			=& 12+ \frac{\frac{12L^6\beta^4}{(L+\beta)^4}+18L^2\beta^4+3\|A\|^2\beta^4}{4L^2(L+\beta)^4\eta_k^4}+\frac{12L^4\beta^4}{L^2(L+\beta)^4\eta_k^4\rho_k^2}\le\tilde{{\color{black}D}}_1,
		\end{align*}
		where the last inequality holds since $\eta_k>\frac{2(L+\beta)}{\rho_k}=k^{1/4}\ge1$. It also follows from the definition of $\tilde{{\color{black}D}}_2$ that $\forall k\ge 1$, we have $\frac{3}{\gamma_k^2}\le\tilde{{\color{black}D}}_2\vartheta_k^2$. Then \eqref{zosc:thm1:3} can be rewritten as
		\begin{align}
			\|\nabla  {\color{black}\bar{\mathcal{G}}}^{\alpha_k,\beta,\gamma_k}_k( x_k,y_{k},\lambda_k)\|^2
			\le& (2\beta^2+3L^{2}+3\|B\|^2)\|y_{k+1}-y_k\|^2+\tilde{{\color{black}D}}_1\vartheta_k^2\|x_{k+1}-x_k\|^2\nonumber\\
			&+\tilde{{\color{black}D}}_2\vartheta_k^2\|\lambda_{k+1}-\lambda_k\|^2+\frac{3d_xL^2\theta_{1,k}^2}{4}+\frac{d_yL^2\theta_{2,k}^2}{2}.\label{zosc:thm1:5}
		\end{align}
		Setting ${\color{black}D}_k^{(2)}=\frac{1}{\max\{\frac{2\beta^2+3L^{2}+3\|B\|^2}{\beta-3L}, \max\{\tilde{{\color{black}D}}_1, \tilde{{\color{black}D}}_2\}\vartheta_k\}}$, then by multiplying ${\color{black}D}_k^{(2)}$ on the both sides of \eqref{zosc:thm1:5} and combining \eqref{zosc:thm1:4}, we then conclude that
		\begin{align}
			&{\color{black}D}_k^{(2)}\|\nabla {\color{black}\bar{\mathcal{G}}}^{\alpha_k,\beta,\gamma_k}_k( x_k,y_{k},\lambda_k)\|^2\nonumber\\
			\le& M_{k}(x_{k},y_{k},\lambda_k)-M_{k+1}(x_{k+1},y_{k+1},\lambda_{k+1})+(\rho_k-\rho_{k+1})\sigma_y^2\nonumber\\
			&+{\color{black}D}_k^{(2)}\left(\frac{3d_xL^2\theta_{1,k}^2}{4}+\frac{d_yL^2\theta_{2,k}^2}{2}\right)+\frac{d_xL\theta_{1,k}^2}{4}+\frac{[(L+\beta)^2+4L^2]d_yL\theta_{2,k}^2}{8(L+\beta)^2}.\label{zosc:thm1:6}
		\end{align}
		Denoting $\tilde{T}_1(\varepsilon)=\min\{k \mid \| \nabla {\color{black}\bar{\mathcal{G}}}^{\alpha_k,\beta,\gamma_k}_k(x_k,y_{k},\lambda_k) \| \leq \frac{\varepsilon}{2}, k\geq 1\}$, {\color{black}The well-definedness of $\tilde{T}_1(\varepsilon)$ follows from arguments similar to those used in proving \eqref{t:1} and \eqref{t:2}.} By summing both sides of \eqref{zosc:thm1:6} from $k=1$ to $\tilde{T}_1(\varepsilon)$, we obtain
		\begin{align}
			&\sum_{k=1}^{\tilde{T}_1(\varepsilon)}{\color{black}D}_k^{(2)}\|\nabla  {\color{black}\bar{\mathcal{G}}}^{\alpha_k,\beta,\gamma_k}_k( x_k,y_{k},\lambda_k)\|^2\nonumber\\
			\le& M_{1}(x_{1},y_{1},\lambda_1)-M_{\tilde{T}_1(\varepsilon)+1}(x_{\tilde{T}_1(\varepsilon)+1},y_{\tilde{T}_1(\varepsilon)+1},\lambda_{\tilde{T}_1(\varepsilon)+1})+\rho_1\sigma_y^2\nonumber\\
			&+\sum_{k=1}^{\tilde{T}_1(\varepsilon)}{\color{black}D}_k^{(2)}\left(\frac{3d_xL^2\theta_{1,k}^2}{4}+\frac{d_yL^2\theta_{2,k}^2}{2}\right)+\sum_{k=1}^{\tilde{T}_1(\varepsilon)}\theta_{1,k}^2\cdot\frac{d_xL}{4}
			+\sum_{k=1}^{\tilde{T}_1(\varepsilon)}\theta_{2,k}^2\cdot\frac{[(L+\beta)^2+4L^2]d_yL}{8(L+\beta)^2}.\label{zosc:thm1:7}
		\end{align}
		Denote $\underbar{M}=\min\limits_{\lambda\in\Lambda}\min\limits_{x\in\mathcal{X}}\min\limits_{y\in\mathcal{Y}}M_k(x,y,\lambda)$. {\color{black}Next, we prove that $\underline{M}$ is a finite value. According to the definition of $M_k(x,y,\lambda)$ in Lemma \ref{zosc:lem1} and the definition of $\Psi_k(x,\lambda)$, we have
			\begin{align}
				M_k(x,y,\lambda)+\max_{y\in \mathcal{Y}}\frac{\rho_k}{2}\|y\|^2
				\ge&\max_{y\in \mathcal{Y}}\mathcal{L}_k(x,y,\lambda)+\max_{y\in \mathcal{Y}}\frac{\rho_k}{2}\|y\|^2
				\ge\max_{y\in \mathcal{Y}}\mathcal{L}(x,y,\lambda),\label{sc:thm1:8}
			\end{align}
			where the last inequality is by the fact that $\max_y f_1(y)+\max_y f_2(y)\ge\max_y\{f_1(y)+f_2(y)\}$. Simlilar to the proof of \eqref{thm1:6.5}, we get $\min\limits_{\lambda\in\Lambda}\min\limits_{x\in\mathcal{X}}\min\limits_{y\in\mathcal{Y}}M_k(x,y,\lambda)>-\infty$.} By the definition of $\tilde{{\color{black}D}}_3$, we then conclude from \eqref{zosc:thm1:7} that
		\begin{align}
			&\sum_{k=1}^{\tilde{T}_1(\varepsilon)}{\color{black}D}_k^{(2)}\|\nabla  {\color{black}\bar{\mathcal{G}}}^{\alpha_k,\beta,\gamma_k}_k( x_k,y_{k},\lambda_k)\|^2\nonumber\\
			&\le \tilde{{\color{black}D}}_3+\sum_{k=1}^{\tilde{T}_1(\varepsilon)}{\color{black}D}_k^{(2)}\left(\frac{3d_xL^2\theta_{1,k}^2}{4}+\frac{d_yL^2\theta_{2,k}^2}{2}\right)+\sum_{k=1}^{\tilde{T}_1(\varepsilon)}\theta_{1,k}^2\cdot\frac{d_xL}{4}+\sum_{k=1}^{\tilde{T}_1(\varepsilon)}\theta_{2,k}^2\cdot\frac{[(L+\beta)^2+4L^2]d_yL}{8(L+\beta)^2}.\label{zosc:thm1:9}
		\end{align}
		Note that {\color{black}$\vartheta_k\ge\frac{L^2}{\rho_1}$}, then we have $\tilde{{\color{black}D}}_4\ge\max\left\{\frac{\beta^2+2L^{2}+3\|B\|^2}{(\beta-3L)\vartheta_k}, \max\{\tilde{{\color{black}D}}_1, \tilde{{\color{black}D}}_2\}\right\}$ by the definition of $\tilde{{\color{black}D}}_4$ and $C_1$, which implies that $\frac{1}{\tilde{{\color{black}D}}_4\vartheta_k}\le d_k^{(2)}\le C_1$. Then by \eqref{zosc:thm1:9} and the definition of $\tilde{T}_1(\varepsilon)$, we obtain 
		\begin{align}
			\frac{\varepsilon^2}{4}\le&\frac{\tilde{{\color{black}D}}_3\tilde{{\color{black}D}}_4}{\sum_{k=1}^{\tilde{T}_1(\varepsilon)}\frac{1}{\vartheta_k}}+\frac{d_xL(3C_1L+1)}{4}\cdot\frac{\sum_{k=1}^{\tilde{T}_1(\varepsilon)}\theta_{1,k}^2}{\sum_{k=1}^{\tilde{T}_1(\varepsilon)}\frac{1}{\vartheta_k}}\nonumber\\
			&+(\frac{[(L+\beta)^2+4L^2]d_yL}{8(L+\beta)^2}+\frac{C_1d_yL^2}{2})\cdot\frac{\sum_{k=1}^{\tilde{T}_1(\varepsilon)}\theta_{2,k}^2}{\sum_{k=1}^{\tilde{T}_1(\varepsilon)}\frac{1}{\vartheta_k}}.\label{zosc:thm1:10}
		\end{align}
		Note that $\rho_k=\frac{2(L+\beta)}{k^{1/4}}\le2(L+\beta)$, and thus by the definition of $\eta_k$, we have $\eta_k\le\frac{2(L+2\beta)(L+\beta)}{\rho_k\beta}$. Then we obtain
		$\sum_{k=1}^{\tilde{T}_1(\varepsilon)}\frac{1}{\vartheta_k}\ge\sum_{k=1}^{\tilde{T}_1(\varepsilon)}\frac{2\beta^4(L+\beta)k^{-1/2}}{L[{\color{black}4}(L+\beta)^3(L+2\beta)^2+L\beta^4]}$. Using the fact that $\sum_{k=1}^{\tilde{T}_1( \varepsilon)}1/k^{1/2}\ge\tilde{T}_1( \varepsilon)^{1/2}$ and the definition of $\theta_{1,k}$, $\theta_{2,k}$ and $C_2$, we conclude that
		\begin{align}
			\tilde{T}_1( \varepsilon)\le \left(\frac{4\tilde{{\color{black}D}}_3\tilde{{\color{black}D}}_4L}{\beta^4C_2\varepsilon^2}\right)^{2}.
		\end{align}
		On the other hand, if $k\ge\frac{256(L+\beta)^4\sigma_y^4}{\varepsilon^4}$, then $\rho_k\le\frac{\varepsilon}{2\sigma_y}$. This inequality together with the definition of $\sigma_y$ then imply that $\rho_k\|y_k\|\le\frac{\varepsilon}{2}$. {\color{black}The definition of $\tilde{D}_3$ implies $\tilde{D}_3 \geq \rho_1 \sigma_y^2$. With $\rho_1 = 2(L+\beta)$ and $\beta = 4L$, a direct substitution yields $\tilde{D}_3 \geq 10 L \sigma_y^2$.} {\color{black}Moreover, under} $\beta=4L$, it can be easily checked that $\tilde{{\color{black}D}}_1>12$, $\tilde{{\color{black}D}}_2>2$, $\tilde{{\color{black}D}}_4\ge350$ and $C_2={\color{black}\frac{1}{8,125.2L^4}}$, thus we get $\left(\frac{4\tilde{{\color{black}D}}_3\tilde{{\color{black}D}}_4L}{\beta^4C_2\varepsilon^2}\right)^{2}>\frac{256(L+\beta)^4\sigma_y^4}{\varepsilon^4}$.
		{\color{black}Therefore,
			\begin{align*}
				T( \varepsilon)\le \frac{16118\tilde{D}_3^2\tilde{D}_4^2L^2}{\varepsilon^4}, 
			\end{align*}
			which gives an upper bound on the minimum iteration 
			$k$ required to achieve $\|\nabla \hat{\mathcal{G}}^{\alpha,\beta,\gamma}(x_k,y_k,\lambda_k)\|\leq \varepsilon$.}
		By Lemma \ref{lem2.1}, if $\|\nabla {\color{black}\hat{\mathcal{G}}}^{\alpha,\beta,\gamma}( x_k,y_{k},\lambda_k)\|\le\varepsilon$, we have $\max\{0,[Ax_{k}+By_{k}-c]_i\}\le\varepsilon$ which completes the proof. 
	\end{proof}
	
	\begin{remark}
		Theorem \ref{zosc:thm1} implies that the number of iterations for Algorithm \ref{zoalg:1} to obtain an $\varepsilon$-stationary point of problem \eqref{problem:1} is bounded by $\mathcal{O}\left(\varepsilon ^{-4} \right)$ and the total number of function value queries {\color{black}is} bounded by $\mathcal{O}\left((d_x+d_y)\varepsilon ^{-4} \right)$  for solving the deterministic nonconvex-concave minimax problem with coupled linear constraints. 
	\end{remark}
	
	{\color{black}
		\begin{remark}
			The unified notation $\unlhd$ in Problem \eqref{problem:1} is designed to cover scenarios with either all inequality or all equality constraints. The algorithm readily accommodates a mixture of both by partitioning the constraints and redefining $\Lambda = \mathbb{R}^p \times \mathbb{R}_{+}^p$. With this modification, the core algorithm and its convergence analysis remain valid, necessitating only a slight adaptation in the projection of $\lambda$, the details of which are omitted for simplicity.
	\end{remark}}
	
	\section{Stochastic Nonconvex-(Strongly) Concave Minimax Problems with Coupled Linear Constraints.}\label{sec3}
	
	{\color{black}In this section we consider problem \eqref{problem:s}, which has many important applications, such as the following example.
		
		\textbf{Data poisoning against logistic regression.} Data poisoning \cite{Jagielski} is an adversarial attack in which the attacker attempts to manipulate the training dataset. The attacker's goal is to corrupt the training dataset in order to modify the predictions on the dataset during the testing phase of a machine learning model. This problem can be formulated as
		\begin{align}\label{app:dp}
			\max_{x\in\mathcal{X}}\min_{y}& ~g(x, y):=\mathbb{E}_{(a, b) \sim \mathcal{D}_{p}} \ell(x, y ; a, b)+\mathbb{E}_{(a, b) \sim \mathcal{D}_{ c}} \ell(0, y ; a, b)+p\|y\|^2,\nonumber\\
			\mbox{s.t.} & ~ Ax+By\le c,
		\end{align}
		where $(a,b)$ is the training dataset, $\ell(x, y ; a, b)=-\left[b \log \left(h\left(x, y; a\right)\right)+(1-b) \log \left(h\left(x, y; a\right)\right)\right]$ where $h\left(x, y; a\right)=\frac{1}{1+e^{-(x+a)^{T}y}}$, $\mathcal{X}=\{\|x\|_{\infty} \leq D_x\}$, $\mathcal{D}_{p}$ and $\mathcal{D}_{c}$ represent the distribution of the poisoned set and clean set, respectively. Note that \eqref{app:dp} can be written in the form $\min\limits_{x\in\mathcal{X}}\max\limits_{\substack{y \\Ax+By\le c}} -g(x, y)$. If the regularization parameter $p=0$, then $-g(x,y)$ is concave with respect to $y$. If $p>0$, then $-g(x,y)$ is strongly concave with respect to $y$.}
	
	Similar to the analysis in Section \ref{sec2}, instead of the original problem \eqref{problem:s}, we solve the following dual problem of \eqref{problem:s}, i.e.,
	\begin{align}
		\min_{\lambda\in\Lambda} \min_{x\in\mathcal{X}}\max_{y\in\mathcal{Y}}\mathcal{L}_g(x,y,\lambda)\tag{D-S},\label{3:dual}
	\end{align}
	where $\mathcal{L}_g(x,y,\lambda)=g(x,y)-\lambda^\top(Ax+By-c)$, $g(x,y)=\mathbb{E}_{\zeta\sim D}[G(x,y;\zeta)]$. 
	Denote $\mathcal{L}_{G}(x,y,\lambda;\zeta)=G(x,y;\zeta)-\lambda^\top(Ax+By-c)$.

	In this section, we propose a zeroth-order regularized momentum primal-dual projected gradient algorithm (ZO-RMPDPG) for solving \eqref{3:dual}. At the $k$th iteration of ZO-RMPDPG, we consider a regularized function of $\mathcal{L}_g(x,y,\lambda)$, i.e., $\tilde{\mathcal{L}}_{g,k}(x,y,\lambda)=\mathcal{L}_g(x,y,\lambda)-\frac{\rho_k}{2}\|y\|^2$ where $\rho_k\ge0$ is a regularization parameter, and correspondingly we denote $ \tilde{\mathcal{L}}_{G,k}(x,y,\lambda;\zeta)=\mathcal{L}_{G}(x,y,\lambda;\zeta)-\frac{\rho_{k}}{2}\|y\|^2$. More detailedly, at the $k$th iteration, for some given $I=\{\zeta_1,\cdots,\zeta_b\}$ drawn {\color{black} independent and identically distributed (i.i.d.)} from an unknown distribution, we compute the {\color{black}zeroth-order} gradient estimators of the stochastic function $\tilde{\mathcal{L}}_{G,k}(x,y,\lambda;I)$ as follows,
	\begin{align}
		\widehat{\nabla} _x\tilde{\mathcal{L}}_{G,k}(x,y,\lambda;I)&=\frac{1}{b}\sum_{j=1}^{b}\widehat{\nabla} _x\tilde{\mathcal{L}}_{G,k}(x,y,\lambda;\zeta_j),\label{sec3:Gkx}\\
		\widehat{\nabla} _y\tilde{\mathcal{L}}_{G,k}(x,y,\lambda;I)&=\frac{1}{b}\sum_{j=1}^{b}\widehat{\nabla} _y\tilde{\mathcal{L}}_{G,k}(x,y,\lambda;\zeta_j),\label{sec3:Gky}
	\end{align}
	where
	\begin{align}
		\widehat{\nabla} _x\tilde{\mathcal{L}}_{G,k}(x,y,\lambda;\zeta_j) &= \widehat{\nabla} _xG_k(x,y;\zeta_j)-A^\top\lambda, \label{lx_esti}\\
		\widehat{\nabla}_{y} \tilde{\mathcal{L}}_{G,k}\left(x,y,\lambda;\zeta_j\right) &= \widehat{\nabla} _yG_k(x,y;\zeta_j)-B^\top\lambda-\rho_ky,\label{ly_esti}\\
		\widehat{\nabla} _xG_k(x,y;\zeta_j) &= \sum_{i=1}^{d_x} \frac{\left[G\left(x+\theta_{1,k} u_{i},y;\zeta_j\right) - G(x,y;\zeta_j)\right]}{\theta_{1,k}} u_{i},\nonumber \\
		\widehat{\nabla}_{y} G_k\left(x,y;\zeta_j\right) &= \sum_{i=1}^{d_y} \frac{\left[G(x,y+\theta_{2,k} v_{i};\zeta_j)-G(x,y;\zeta_j)\right]}{\theta_{2,k}} v_{i}.\nonumber
	\end{align}
	Then, based on $\widehat{\nabla} _x\tilde{\mathcal{L}}_{G,k}(x,y,\lambda;I)$ and $\widehat{\nabla} _y\tilde{\mathcal{L}}_{G,k}(x,y,\lambda;I)$, we compute the variance-reduced stochastic gradient $v_k$ and $w_k$ as shown in \eqref{vk} and \eqref{wk} respectively with $0< \iota_k \leq 1$ and $0<\varrho_{k}\leq 1$ that will be defined later. We update $x_k$ and $y_k$ through alternating stochastic ``gradient" projection with the momentum technique shown in \eqref{zo2:update-tx}-\eqref{zo2:update-y} which is based on the idea of the Acc-ZOMDA algorithm\cite{Huang}. The proposed ZO-RMPDPG algorithm is formally stated in Algorithm \ref{zoalg:2}.
	
	Note that Algorithm \ref{zoalg:2} differs from the Acc-ZOMDA algorithm proposed in \cite{Huang} in three ways. On one hand, compared to the Acc-ZOMDA algorithm, the main difference in the ZO-RMPDPG algorithm is that instead of {\color{black}$\mathcal{L}_G(x,y,\lambda;\zeta)$}, the zeroth-order gradient estimators of {\color{black}$\tilde{\mathcal{L}}_{G,k}(x,y,\lambda;\zeta)$} are computed and used at each iteration. On the other hand, the Acc-ZOMDA algorithm uses uniformly distributed random vectors on the unit sphere to calculate the zeroth-order gradient estimators of stochastic gradient, while our algorithm uses a finite-difference {\color{black}zeroth-order gradient estimators of stochastic gradient} which can {\color{black}yield} to a better complexity bound in our settings. 
	Thirdly, the Acc-ZOMDA algorithm is used to solve nonconvex-strongly concave minimax problems without coupled linear constraints, whereas Algorithm \ref{zoalg:2} is designed to solve more general nonconvex-concave minimax problems with coupled linear constraints.

	\begin{algorithm}[t]
		\caption{A zeroth-order regularized momentum primal-dual projected gradient algorithm(ZO-RMPDPG)}
		\label{zoalg:2}
		\begin{algorithmic}
			\STATE{\textbf{Step 1}: Input $x_1,y_1,\lambda_1,\tilde{\alpha}_1,\tilde{\beta},\tilde{\gamma}_1,0<\eta_1\leq 1, b$; $\varrho_0=1$,$\iota_0=1$. Set $k=1$.}
			\STATE{\textbf{Step 2}: Draw a {\color{black}mini-batch sample} $I_{k+1}=\{\zeta_i^{k+1}\}_{i=1}^b$. Compute 
				\begin{align}
					v_{k}&=\widehat{\nabla} _x\tilde{\mathcal{L}}_{G,k}( x_{k},y_{k},\lambda_{k};I_{k})+(1-\varrho_{k-1})[v_{k-1}-\widehat{\nabla} _x\tilde{\mathcal{L}}_{G,k-1}( x_{k-1},y_{k-1},\lambda_{k-1};I_{k})],\label{vk}\\
					w_{k}&=\widehat{\nabla} _y\tilde{\mathcal{L}}_{G,k}( x_{k},y_{k},\lambda_{k};I_{k})+(1-\iota_{k-1})[w_{k-1}-\widehat{\nabla} _y\tilde{\mathcal{L}}_{G,k-1}( x_{k-1},y_{k-1},\lambda_{k-1};I_{k})],\label{wk}
				\end{align}
				where $\widehat{\nabla} _x\tilde{\mathcal{L}}_{G,k}( x_{k},y_{k},\lambda_{k};I)$, $\widehat{\nabla} _x\tilde{\mathcal{L}}_{G,k-1}( x_{k},y_{k},\lambda_{k};I)$ and $\widehat{\nabla} _y\tilde{\mathcal{L}}_{G,k}( x_{k},y_{k},\lambda_{k};I)$, $\widehat{\nabla} _x\tilde{\mathcal{L}}_{G,k-1}( x_{k},y_{k},\lambda_{k};I)$ are defined as in \eqref{sec3:Gkx} and \eqref{sec3:Gky}.}
			\STATE{\textbf{Step 3}: Perform the following update for $x_k$, $y_k$ and $\lambda_k$:  	
				\qquad \begin{align}
					\tilde{x}_{k+1}&={\color{black}\mathcal{P}_{\mathcal{X}}} \left( x_k - \tilde{\alpha}_kv_k\right),\label{zo2:update-tx}\\ x_{k+1}&=x_k+\eta_k(\tilde{x}_{k+1}-x_k),\label{zo2:update-x}\\
					\tilde{y}_{k+1}&={\color{black}\mathcal{P}_{\mathcal{Y}}} \left( y_k + \tilde{\beta}w_k\right),\label{zo2:update-ty}\\ y_{k+1}&=y_k+\eta_k(\tilde{y}_{k+1}-y_k).\label{zo2:update-y}\\
					\tilde{\lambda}_{k+1}&={\color{black}\mathcal{P}_{\Lambda}}(\lambda_k+\tilde{\gamma}_k(Ax_{k+1}+By_{k+1}-c)),\label{zo2:update-tlambda}\\ \lambda_{k+1}&=\lambda_k+\eta_k(\tilde{\lambda}_{k+1}-\lambda_k).\label{zo2:update-lambda}
			\end{align}}
			\STATE{\textbf{Step 4}: If some stationary condition is satisfied, stop; otherwise, set $k=k+1, $ go to Step 2.}
		\end{algorithmic}
	\end{algorithm}

	Before we prove the iteration complexity of Algorithm \ref{zoalg:2}, we first give some mild assumptions. 
	\begin{assumption}\label{azoass:Lip}
		For any given $\zeta$, the function $G(x,y,\zeta)$ has Lipschitz continuous gradients, i.e., there exists a constant $l>0$ such that for any $x, x_1, x_2\in \mathcal{X}$, and $y, y_1, y_2\in \mathcal{Y}$, we have
		\begin{align*}
			\| \nabla_{x} G(x_1,y,\zeta)-\nabla_{x} G(x_2,y,\zeta)\| &\leq l\|x_{1}-x_{2}\|,\\
			\|\nabla_{x} G(x,y_1,\zeta)-\nabla_{x} G(x,y_2,\zeta)\| &\leq l\|y_{1}-y_{2}\|,\\
			\|\nabla_{y}G(x,y_1,\zeta)-\nabla_{y} G(x,y_2,\zeta)\| &\leq l\|y_{1}-y_{2}\|,\\
			\|\nabla_{y} G(x_1,y,\zeta)-\nabla_{y} G(x_2,y,\zeta)\| &\leq l\|x_{1}-x_{2}\|.
		\end{align*}
	\end{assumption}
	
	If Assumption \ref{azoass:Lip} holds, $g(x,y)$ has Lipschitz continuous gradients with constant $l$ by Lemma 7 in \cite{xu22zeroth}. Then, by the definition of $\mathcal{L}_g(x,y,\lambda)$, we know that $\mathcal{L}_g(x,y,\lambda)$ has Lipschitz continuous gradients with constant $L$ and $L=\max\{l,\|A\|,\|B\|\}$. 
	
	{\color{black}In the following analysis, we use $\mathbb{E}$ to denote $\mathbb{E}_{\zeta\sim D}$ for simplicity.}
	
	\begin{assumption}\label{azoass:var}
		{\color{black}The variance of  zeroth-order gradient estimators is bounded, i.e.,} for any given $\zeta$, there exists a constant $\delta>0$ such that for all $x$ and $y$, it has
		\begin{align*}
			\mathbb{E}\|\widehat{\nabla}_x  G_k(x,y;\zeta)-\widehat{\nabla} _xg_k(x,y) \|^2&\le\delta^2,\\
			\mathbb{E}\|\widehat{\nabla}_y  G_k(x,y;\zeta)-\widehat{\nabla} _yg_k(x,y) \|^2&\le\delta^2,
		\end{align*}
		where
		\begin{align*}
			\widehat{\nabla} _xg_k(x,y) &= \sum_{i=1}^{d_x} \frac{\left[g\left(x+\theta_{1,k} u_{i},y\right) - g(x,y)\right]}{\theta_{1,k}} u_{i}, \\
			\widehat{\nabla}_{y} g_k\left(x,y\right) &= \sum_{i=1}^{d_y} \frac{\left[g(x,y+\theta_{2,k} v_{i})-g(x,y)\right]}{\theta_{2,k}} v_{i}.
		\end{align*}
	\end{assumption}
	Denote $\widehat{\nabla} _xG_k(x,y;I)=\frac{1}{b}\sum_{j=1}^{b}\widehat{\nabla} _xG_k(x,y;\zeta_j)$, $\widehat{\nabla} _yG_k(x,y;I)=\frac{1}{b}\sum_{j=1}^{b}\widehat{\nabla} _yG_k(x,y;\zeta_j)$. By Assumption \ref{azoass:var}, 
	we can immediately get that
	$\mathbb{E}\|\widehat{\nabla}_x  G_k(x,y;I)-\widehat{\nabla} _xg_k(x,y) \|^2\le\frac{\delta^2}{b}$ and  $\mathbb{E}\|\widehat{\nabla}_y  G_k(x,y;I)-\widehat{\nabla} _yg_k(x,y) \|^2\le\frac{\delta^2}{b}$.
	Similar to Lemma \ref{varboundlem}, we provide an upper bound on the {\color{black}zeroth-order} gradient estimators as follows.
	\begin{lemma}\label{azovb}
		Suppose that Assumption \ref{azoass:Lip} holds. Then for any given $\zeta$, {\color{black}let $L=\max\{l,\|A\|,\|B\|\}$}, we have  
		\begin{align*}
			\left\|\widehat{\nabla}_{x}\tilde{\mathcal{L}}_{G,k}(x,y,\lambda;\zeta) - \nabla_{x}\tilde{\mathcal{L}}_{G,k}(x,y,\lambda;\zeta)\right\|^{2} \leq \frac{d_xL^2\theta_{1,k}^2}{4},\\
			\left\|\widehat{\nabla}_{y}\tilde{\mathcal{L}}_{G,k}(x,y,\lambda;\zeta)- \nabla_{y}\tilde{\mathcal{L}}_{G,k}(x,y,\lambda;\zeta)\right\|^{2} \leq \frac{d_yL^2\theta_{2,k}^2}{4},
		\end{align*}
		and
		\begin{align*}
			\left\|\widehat{\nabla}_{x}\tilde{\mathcal{L}}_{g,k}(x,y,\lambda) - \nabla_{x}\tilde{\mathcal{L}}_{g,k}(x,y,\lambda)\right\|^{2} \leq \frac{d_xL^2\theta_{1,k}^2}{4},\\
			\left\|\widehat{\nabla}_{y}\tilde{\mathcal{L}}_{g,k}(x,y,\lambda)- \nabla_{y}\tilde{\mathcal{L}}_{g,k}(x,y,\lambda)\right\|^{2} \leq \frac{d_yL^2\theta_{2,k}^2}{4},
		\end{align*}
	\end{lemma}

	
	We define the stationarity gap as the termination criterion as follows.
	\begin{definition}\label{azogap-f}
		For some given $\alpha>0$, $\beta>0$ and $\gamma>0$, denote 	\begin{equation*}
			\nabla {\color{black}\mathcal{G}^{\alpha,\beta ,\gamma}}\left( x,y,\lambda \right) :=\left( \begin{array}{c}
				{\color{black}\frac{1}{\alpha}}\left( x-{\color{black}\mathcal{P}_{\mathcal{X}}}\left( x-{\color{black}\alpha}\nabla _x\mathcal{L}_g\left( x,y,\lambda  \right) \right) \right)\\
				{\color{black}\frac{1}{\beta}} \left( y-{\color{black}\mathcal{P}_{\mathcal{Y}}}\left( y+ {\color{black}\beta}\nabla _y\mathcal{L}_g\left( x,y,\lambda \right) \right) \right)\\
				{\color{black}\frac{1}{\gamma}}\left( \lambda-{\color{black}\mathcal{P}_{\Lambda}}\left(\lambda-{\color{black}\gamma}\nabla_{\lambda}\mathcal{L}_g(x,y,\lambda)\right)\right)\\
			\end{array} \right) .
		\end{equation*}
	\end{definition}
	
	\subsection{Complextiy Analysis: Nonconvex-Strongly Concave Setting}
	In this subsection, we prove the iteration complexity of Algorithm \ref{zoalg:2} under the nonconvex-strongly concave setting, i.e., $G(x,y,\zeta)$ is $\mu$-strongly concave with respect to $y$ for any given $x\in\mathcal{X}$. Under this setting, $\forall k \ge 1$, we set
	\begin{equation}\label{3.1par}
		\rho_k=0, \tilde{\alpha}_k=\tilde{\alpha}, \tilde{\gamma}_k=\tilde{\gamma}.
	\end{equation}
	Let 
	\begin{align*}
		\Phi(x,\lambda):=\max_{y\in\mathcal{Y}}\mathcal{L}_g(x,y,\lambda),\quad
		y^*(x,\lambda):=\arg\max_{y\in\mathcal{Y}}\mathcal{L}_g(x,y,\lambda).
	\end{align*}
	By Lemma {\color{black}B.1} in \cite{nouiehed2019solving} and $\mu$-strong concavity of $\mathcal{L}_g(x,y,\lambda)$, $\Phi(x,\lambda)$ is $L_\Phi$-Lipschitz smooth with $L_\Phi=L+\frac{L^2}{\mu}$, and by Lemma {\color{black}23} in \cite{lin2020near}, for any given $x$ and $\lambda$, we have 
	\begin{align}
		&\nabla_x\Phi(x,\lambda)=\nabla_x \mathcal{L}_g(x, y^*(x,\lambda),\lambda),\quad \nabla_\lambda\Phi(x,\lambda)=\nabla_\lambda\mathcal{L}_g(x, y^*(x,\lambda),\lambda),\label{azogradx:nsc}\\
		&\|y^*(x,\lambda_{1})-y^*(x,\lambda_{2})\|\le\frac{L}{\mu}\|\lambda_{1}-\lambda_{2}\|,\quad\|y^*(x_{1},\lambda)-y^*(x_{2},\lambda)\|\le\frac{L}{\mu}\|x_{1}-x_{2}\|\label{sc:1x}.
	\end{align}

	We also need to make the following assumption on the parameters $\theta_{1,k}$ and $\theta_{2,k}$.
	
	\begin{assumption}\label{azoscthetak}
		$\{\theta_{1,k}\}$ and $\{\theta_{2,k}\}$ are nonnegative monotonically decreasing sequences.
	\end{assumption}

	\begin{lemma}\label{azosclem1}
		Suppose that Assumption \ref{azoass:Lip} holds. Let $\{\left(x_k,y_k,\lambda_k\right)\}$ be a sequence generated by Algorithm \ref{zoalg:2} with parameter settings in \eqref{3.1par},
		then $\forall k \ge 1$, 
		\begin{align}
			&\Phi(x_{k+1},\lambda_{k+1})-\Phi(x_{k},\lambda_k)\nonumber\\
			\le&4\eta_k\tilde{\alpha} L^2\|y_k-y^*(x_k,\lambda_k)\|^2+2\eta_k\tilde{\alpha}\|\widehat{\nabla} _x\tilde{\mathcal{L}}_{g,k}(x_{k},y_{k},\lambda_k)-v_k\|^2-(\frac{3\eta_k}{4\tilde{\alpha}}-\frac{L^2\eta_k^2}{\mu})\|\tilde{x}_{k+1}-x_k\|^2\nonumber\\
			&+4\eta_k\tilde{\alpha} L^2\|y_{k+1}-y^*(x_{k+1},\lambda_{k+1})\|^2-(\frac{\eta_k}{\tilde{\gamma}}-\frac{\|B\|^2\eta_k}{16\tilde{\alpha}L^2}-\frac{\eta_k^2L^2}{\mu})\|\tilde{\lambda}_{k+1}-\lambda_k\|^2+\eta_k\tilde{\alpha} d_xL^2\theta_{1,k}^2.\label{azosclem1:1}
		\end{align}
	\end{lemma}
	
	\begin{proof}
		Since that $\Phi(x, \lambda)$ is $L_\Phi$-smooth with respect to $x$ and by \eqref{azogradx:nsc} and \eqref{zo2:update-x}, we have that
		\begin{align} 
			&\Phi(x_{k+1},\lambda_{k})-\Phi(x_{k},\lambda_k)\nonumber\\
			\le &\langle \nabla _x\Phi(x_{k},\lambda_k) ,x_{k+1}-x_k \rangle  +\frac{L_\Phi}{2}\|x_{k+1}-x_k\|^2\nonumber\\
			=&\langle \nabla _x\mathcal{L}_g(x_{k},y^*(x_k,\lambda_k),\lambda_k) ,x_{k+1}-x_k \rangle  +\frac{L_\Phi}{2}\|x_{k+1}-x_k\|^2\nonumber\\
			=&\eta_k\langle \nabla _x\mathcal{L}_g(x_{k},y^*(x_k,\lambda_k),\lambda_k)-\widehat{\nabla} _x\tilde{\mathcal{L}}_{g,k}(x_{k},y_{k},\lambda_k) ,\tilde{x}_{k+1}-x_k \rangle+\eta_k\langle v_k ,\tilde{x}_{k+1}-x_k \rangle  \nonumber\\
			&+\eta_k\langle \widehat{\nabla} _x\tilde{\mathcal{L}}_{g,k}(x_{k},y_{k},\lambda_k)-v_k,\tilde{x}_{k+1}-x_k \rangle+\frac{L_\Phi\eta_k^2}{2}\| \tilde{x}_{k+1}-x_k\|^2.\label{azosclem1:2}
		\end{align}
		Next, we estimate the first three terms in the right hand side of \eqref{azosclem1:2}. By the Cauchy-Schwarz inequality and Lemma \ref{azovb}, we get
		\begin{align}
			&\langle \nabla _x\mathcal{L}_g(x_{k},y^*(x_k,\lambda_k),\lambda_k)-\widehat{\nabla} _x\tilde{\mathcal{L}}_{g,k}(x_{k},y_{k},\lambda_k) ,\tilde{x}_{k+1}-x_k \rangle\nonumber\\
			\le&2\tilde{\alpha}\|\nabla _x\mathcal{L}_g(x_{k},y^*(x_k,\lambda_k),\lambda_k)-\widehat{\nabla} _x\tilde{\mathcal{L}}_{g,k}(x_{k},y_{k},\lambda_k)\|^2+\frac{1}{8\tilde{\alpha}}\|\tilde{x}_{k+1}-x_k\|^2\nonumber\\
			\le&4\tilde{\alpha}\|\nabla _x\mathcal{L}_g(x_{k},y^*(x_k,\lambda_k),\lambda_k)-\nabla _x\mathcal{L}_g(x_{k},y_{k},\lambda_k)\|^2+\frac{1}{8\tilde{\alpha}}\|\tilde{x}_{k+1}-x_k\|^2\nonumber\\
			&+4\tilde{\alpha}\|\nabla _x\mathcal{L}_g(x_{k},y_{k},\lambda_k)-\widehat{\nabla} _x\tilde{\mathcal{L}}_{g,k}(x_{k},y_{k},\lambda_k)\|^2\nonumber\\
			\le&4\tilde{\alpha} L^2\|y_k-y^*(x_k,\lambda_k)\|^2+\frac{1}{8\tilde{\alpha}}\|\tilde{x}_{k+1}-x_k\|^2+\tilde{\alpha} d_xL^2\theta_{1,k}^2.\label{azosclem1:3}
		\end{align}
		By the Cauchy-Schwarz inequality, we have
		\begin{align}
			\langle \widehat{\nabla} _x\tilde{\mathcal{L}}_{g,k}(x_{k},y_{k},\lambda_k)-v_k,\tilde{x}_{k+1}-x_k \rangle
			\le2\tilde{\alpha}\|\widehat{\nabla} _x\tilde{\mathcal{L}}_{g,k}(x_{k},y_{k},\lambda_k)-v_k\|^2+\frac{1}{8\tilde{\alpha}}\|\tilde{x}_{k+1}-x_k\|^2.\label{azosclem1:4}
		\end{align}
		The optimality condition for {\color{black}$\tilde{x}_{k+1}$} in \eqref{zo2:update-tx} implies that $\forall x\in \mathcal{X}$ and $\forall k\geq 1$,
		\begin{equation}\label{azosclem1:5}
			\langle v_k,\tilde{x}_{k+1}-x_k \rangle \le -\frac{1}{\tilde{\alpha}}\|\tilde{x}_{k+1}-x_k\|^2.
		\end{equation}
		Plugging \eqref{azosclem1:3}, \eqref{azosclem1:4} and \eqref{azosclem1:5} into \eqref{azosclem1:2} and by $L_\Phi\le\frac{2L^2}{\mu}$, we get
		\begin{align}
			\Phi(x_{k+1},\lambda_{k})-\Phi(x_{k},\lambda_k)
			\le&4\eta_k\tilde{\alpha} L^2\|y_k-y^*(x_k,\lambda_k)\|^2+2\eta_k\tilde{\alpha}\|\widehat{\nabla} _x\tilde{\mathcal{L}}_{g,k}(x_{k},y_{k},\lambda_k)-v_k\|^2\nonumber\\
			&-(\frac{3\eta_k}{4\tilde{\alpha}}-\frac{L^2\eta_k^2}{\mu})\|\tilde{x}_{k+1}-x_k\|^2+\eta_k\tilde{\alpha} d_xL^2\theta_{1,k}^2.\label{azosclem1:6}
		\end{align}
		On the other hand, $\Phi(x, \lambda)$ is $L_\Phi$-smooth with respect to $\lambda$ and by \eqref{zo2:update-lambda}, we have
		\begin{align} 
			&\Phi(x_{k+1},\lambda_k)-\Phi(x_{k+1},\lambda_{k+1})\nonumber\\
			\ge &\langle \nabla _\lambda\Phi(x_{k+1},\lambda_{k+1}) ,\lambda_k-\lambda_{k+1}\rangle  -\frac{L_\Phi}{2}\|\lambda_{k+1}-\lambda_k\|^2\nonumber\\
			=&\eta_k\langle \nabla _\lambda \mathcal{L}_g(x_{k+1},y^*(x_{k+1},\lambda_{k+1}),\lambda_k) ,\lambda_k-\tilde{\lambda}_{k+1} \rangle  -\frac{\eta_k^2L_\Phi}{2}\|\tilde{\lambda}_{k+1}-\lambda_k\|^2\nonumber\\
			=&\eta_k\langle \nabla _\lambda \mathcal{L}_g(x_{k+1},y^*(x_{k+1},\lambda_{k+1}),\lambda_k)-\nabla _\lambda \mathcal{L}_g(x_{k+1},y_{k+1},\lambda_k) ,\lambda_k-\tilde{\lambda}_{k+1}\rangle\nonumber\\
			&+\eta_k\langle \nabla _\lambda \mathcal{L}_g(x_{k+1},y_{k+1},\lambda_k) ,\lambda_k-\tilde{\lambda}_{k+1} \rangle-\frac{\eta_k^2L_\Phi}{2}\|\tilde{\lambda}_{k+1}-\lambda_k\|^2.\label{azosclem1:7}
		\end{align}
		Next, we estimate the first two terms in the right hand side of \eqref{azosclem1:7}. By the Cauchy-Schwarz inequality, we get
		\begin{align}
			&\langle \nabla _\lambda \mathcal{L}_g(x_{k+1},y^*(x_{k+1},\lambda_{k+1}),\lambda_k)-\nabla _\lambda \mathcal{L}_g(x_{k+1},y_{k+1},\lambda_k) ,\lambda_k-\tilde{\lambda}_{k+1}\rangle\nonumber\\
			\ge&-4\tilde{\alpha} L^2\|y_{k+1}-y^*(x_{k+1},\lambda_{k+1})\|^2-\frac{\|B\|^2}{16\tilde{\alpha}L^2}\|\tilde{\lambda}_{k+1}-\lambda_k\|^2.\label{azosclem1:8}
		\end{align}
		The optimality condition for {\color{black}$\tilde{\lambda}_{k+1}$} in \eqref{zo2:update-tlambda} implies that,
		\begin{equation}\label{azosclem1:9}
			\langle \nabla _\lambda \mathcal{L}_g(x_{k+1},y_{k+1},\lambda_k) ,\lambda_k-\tilde{\lambda}_{k+1} \rangle\ge\frac{1}{\tilde{\gamma}}\|\tilde{\lambda}_{k+1}-\lambda_k\|^2.
		\end{equation}
		Plugging \eqref{azosclem1:8} and \eqref{azosclem1:9} into \eqref{azosclem1:7}, and combining \eqref{azosclem1:6}, we have
		\begin{align*}
			&\Phi(x_{k+1},\lambda_{k+1})-\Phi(x_{k},\lambda_k)\nonumber\\
			\le&4\eta_k\tilde{\alpha} L^2\|y_k-y^*(x_k,\lambda_k)\|^2+2\eta_k\tilde{\alpha}\|\widehat{\nabla} _x\tilde{\mathcal{L}}_{g,k}(x_{k},y_{k},\lambda_k)-v_k\|^2-(\frac{3\eta_k}{4\tilde{\alpha}}-\frac{L^2\eta_k^2}{\mu})\|\tilde{x}_{k+1}-x_k\|^2\nonumber\\
			&+4\eta_k\tilde{\alpha} L^2\|y_{k+1}-y^*(x_{k+1},\lambda_{k+1})\|^2-(\frac{\eta_k}{\tilde{\gamma}}-\frac{\|B\|^2\eta_k}{16\tilde{\alpha}L^2}-\frac{\eta_k^2L_\Phi}{2})\|\tilde{\lambda}_{k+1}-\lambda_k\|^2+\eta_k\tilde{\alpha} d_xL^2\theta_{1,k}^2.
		\end{align*}
		Then, by $L_\Phi\le\frac{2L^2}{\mu}$, we complete the proof.
	\end{proof}

	\begin{lemma}\label{azosclem2}
		Suppose that Assumption \ref{azoass:Lip} holds. Let $\{\left(x_k,y_k,\lambda_k\right)\}$ be a sequence generated by Algorithm \ref{zoalg:2} with parameter settings in \eqref{3.1par}, if $0<\eta_k\le1$ and  $0<\tilde{\beta}\le\frac{1}{6L}$,
		then $\forall k \ge 1$, 
		\begin{align}
			&\|y_{k+1}-y^*(x_{k+1},\lambda_{k+1})\|^2\nonumber\\
			\le&(1-\frac{\eta_k\tilde{\beta}\mu}{4})\|y_{k}-y^*(x_k,\lambda_k)\|^2-\frac{3\eta_k}{4}\|\tilde{y}_{k+1} -y_k\|^2+\frac{10L^2\eta_k}{\mu^3\tilde{\beta}}\|\tilde{x}_{k+1}-x_k\|^2\nonumber\\
			&+\frac{10\eta_k\tilde{\beta}}{\mu}\|\widehat{\nabla} _{y}\tilde{\mathcal{L}}_{g,k}(x_{k},y_{k},\lambda_k)-w_k\|^2+\frac{5\eta_k\tilde{\beta}d_yL^2\theta_{2,k}^2}{2\mu}+\frac{10L^2\eta_k}{\mu^3\tilde{\beta}}\|\tilde{\lambda}_{k+1}-\lambda_k\|^2.\label{azosclem2:1}
		\end{align}
	\end{lemma}
	
	\begin{proof}
		$\mathcal{L}_g(x,y,\lambda)$ is $\mu$-strongly concave with respect to $y$, which implies that
		\begin{align}\label{azosclem2:2}
			&\mathcal{L}_g(x_{k},y,\lambda_k)-\mathcal{L}_g( x_{k},y_{k},\lambda_k)\nonumber\\
			\le& \langle \nabla _{y}\mathcal{L}_g(x_{k},y_{k},\lambda_k),y-y_{k} \rangle -\frac{\mu}{2}\| y-y_{k} \|^2\nonumber\\
			=& \langle w_k,y-\tilde{y}_{k+1} \rangle +\langle \nabla _{y}\mathcal{L}_g(x_{k},y_{k},\lambda_k)-w_k,y-\tilde{y}_{k+1} \rangle+\langle \nabla _{y}\mathcal{L}_g(x_{k},y_{k},\lambda_k),\tilde{y}_{k+1}-y_{k} \rangle-\frac{\mu}{2}\| y-y_{k} \|^2.
		\end{align}
		By  Assumption \ref{azoass:Lip}, $\mathcal{L}_g(x,y,\lambda)$ has Lipschitz continuous gradient with respect to $y$, which implies that
		\begin{align}\label{azosclem2:3}
			\mathcal{L}_g(x_{k},\tilde{y}_{k+1},\lambda_k)-\mathcal{L}_g( x_{k},y_{k},\lambda_k)
			\ge \langle \nabla _{y}\mathcal{L}_g(x_{k},y_{k},\lambda_k),\tilde{y}_{k+1}-y_{k} \rangle -\frac{L}{2}\| \tilde{y}_{k+1}-y_{k} \|^2.
		\end{align}
		The optimality condition for {\color{black}$\tilde{y}_{k+1}$} in \eqref{zo2:update-y} implies that $\forall y\in \mathcal{Y}$ and $\forall k\geq 1$,
		\begin{align}\label{azosclem2:4}
			\langle w_k,y-\tilde{y}_{k+1} \rangle\le \frac{1}{\tilde{\beta}}\langle\tilde{y}_{k+1}-y_k,y-\tilde{y}_{k+1} \rangle
			=-\frac{1}{\tilde{\beta}}\|\tilde{y}_{k+1}-y_k\|^2+\frac{1}{\tilde{\beta}}\langle\tilde{y}_{k+1}-y_k,y-y_k \rangle.
		\end{align}
		Plugging \eqref{azosclem2:4} into \eqref{azosclem2:2} and combining \eqref{azosclem2:3}, and setting $y=y^*(x_k,\lambda_k)$, we have
		\begin{align}\label{azosclem2:5}
			&\mathcal{L}_g(x_{k},y^*(x_k,\lambda_k),\lambda_k)-\mathcal{L}_g( x_{k},\tilde{y}_{k+1},\lambda_k)\nonumber\\
			\le& \frac{1}{\tilde{\beta}}\langle\tilde{y}_{k+1}-y_k,y^*(x_k,\lambda_k)-y_k \rangle+\langle \nabla _{y}\mathcal{L}_g(x_{k},y_{k},\lambda_k)-w_k,y^*(x_k,\lambda_k)-\tilde{y}_{k+1} \rangle \nonumber\\
			&-\frac{\mu}{2}\|y_{k}-y^*(x_k,\lambda_k) \|^2-(\frac{1}{\tilde{\beta}}-\frac{L}{2})\|\tilde{y}_{k+1}-y_k\|^2.
		\end{align}
		Next, we estimate the first two terms in the right hand side of \eqref{azosclem2:5}. By \eqref{zo2:update-y}, we get
		\begin{align}
			\|y_{k+1}-y^*(x_k,\lambda_k)\|^2
			=&\|y_{k}+\eta_k(\tilde{y}_{k+1}-y_k)-y^*(x_k,\lambda_k)\|^2\nonumber\\
			=&\|y_k-y^*(x_k,\lambda_k)\|^2+2\eta_k\langle\tilde{y}_{k+1}-y_k,y_k-y^*(x_k,\lambda_k) \rangle+\eta_k^2\|\tilde{y}_{k+1}-y_k\|^2.\label{azosclem2:6}
		\end{align}
		\eqref{azosclem2:6} can be rewritten as 
		\begin{align}
			\langle\tilde{y}_{k+1}-y_k,y^*(x_k,\lambda_k)-y_k \rangle
			=\frac{1}{2\eta_k}\|y_k-y^*(x_k,\lambda_k)\|^2+\frac{\eta_k}{2}\|\tilde{y}_{k+1}-y_k\|^2-\frac{1}{2\eta_k}	\|y_{k+1}-y^*(x_k,\lambda_k)\|^2.\label{azosclem2:7}
		\end{align}
		By the Cauchy-Schwarz inequality and Lemma \ref{azovb}, we have
		\begin{align}
			&\langle \nabla _{y}\mathcal{L}_g(x_{k},y_{k},\lambda_k)-w_k,y^*(x_k,\lambda_k)-\tilde{y}_{k+1} \rangle \nonumber\\
			\le&\frac{2}{\mu}\|\nabla _{y}\mathcal{L}_g(x_{k},y_{k},\lambda_k)-w_k\|^2+\frac{\mu}{8}\|y^*(x_k,\lambda_k)-\tilde{y}_{k+1} \|^2\nonumber\\
			\le&\frac{4}{\mu}\|\widehat{\nabla} _{y}\tilde{\mathcal{L}}_{g,k}(x_{k},y_{k},\lambda_k)-w_k\|^2+\frac{4}{\mu}\|\widehat{\nabla} _{y}\tilde{\mathcal{L}}_{g,k}(x_{k},y_{k},\lambda_k)-\nabla _{y}\mathcal{L}_g(x_{k},y_{k},\lambda_k)\|^2\nonumber\\
			&+\frac{\mu}{4}\|y^*(x_k,\lambda_k)-y_k\|^2+\frac{\mu}{4}\|\tilde{y}_{k+1} -y_k\|^2\nonumber\\
			\le&\frac{4}{\mu}\|\widehat{\nabla} _{y}\tilde{\mathcal{L}}_{g,k}(x_{k},y_{k},\lambda_k)-w_k\|^2+\frac{\mu}{4}\|y^*(x_k,\lambda_k)-y_k\|^2+\frac{\mu}{4}\|\tilde{y}_{k+1} -y_k\|^2+\frac{d_yL^2\theta_{2,k}^2}{\mu}.\label{azosclem2:8}
		\end{align}
		Plugging \eqref{azosclem2:7}, \eqref{azosclem2:8} into \eqref{azosclem2:5}, and using the fact that $\mathcal{L}_g(x_{k},y^*(x_k,\lambda_k),\lambda_k)\ge \mathcal{L}_g( x_{k},\tilde{y}_{k+1},\lambda_k)$, we get
		\begin{align}
			\frac{1}{2\eta_k\tilde{\beta}}\|y_{k+1}-y^*(x_k,\lambda_k)\|^2
			\le&(\frac{1}{2\eta_k\tilde{\beta}}-\frac{\mu}{4})\|y_{k}-y^*(x_k,\lambda_k)\|^2+(\frac{\eta_k}{2\tilde{\beta}}+\frac{\mu}{4}+\frac{L}{2}-\frac{1}{\tilde{\beta}})\|\tilde{y}_{k+1} -y_k\|^2\nonumber\\
			&+\frac{4}{\mu}\|\widehat{\nabla} _{y}\tilde{\mathcal{L}}_{g,k}(x_{k},y_{k},\lambda_k)-w_k\|^2+\frac{d_yL^2\theta_{2,k}^2}{\mu}.\label{azosclem2:9}
		\end{align}
		By the assumption $0<\eta_k\le1$, $0<\tilde{\beta}\le\frac{1}{6L}$ and the fact $\mu\le L$, \eqref{azosclem2:9} implies that
		\begin{align}
			\|y_{k+1}-y^*(x_k,\lambda_k)\|^2
			\le&(1-\frac{\eta_k\tilde{\beta}\mu}{2})\|y_{k}-y^*(x_k,\lambda_k)\|^2-\frac{3\eta_k}{4}\|\tilde{y}_{k+1} -y_k\|^2\nonumber\\
			&+\frac{8\eta_k\tilde{\beta}}{\mu}\|\widehat{\nabla} _{y}\tilde{\mathcal{L}}_{g,k}(x_{k},y_{k},\lambda_k)-w_k\|^2+\frac{2\eta_k\tilde{\beta}d_yL^2\theta_{2,k}^2}{\mu}.\label{azosclem2:10}
		\end{align}
		By the Cauchy-Schwarz inequality, \eqref{zo2:update-x}, \eqref{zo2:update-lambda} and \eqref{sc:1x}, we have
		\begin{align}
			&\|y_{k+1}-y^*(x_{k+1},\lambda_{k+1})\|^2\nonumber\\
			=&\|y_{k+1}-y^*(x_{k},\lambda_{k})\|^2+2\langle y_{k+1}-y^*(x_{k},\lambda_{k}),y^*(x_{k},\lambda_{k})-y^*(x_{k+1},\lambda_{k+1}) \rangle+\|y^*(x_{k},\lambda_{k})-y^*(x_{k+1},\lambda_{k+1})\|^2\nonumber\\
			\le&(1+\frac{\eta_k\tilde{\beta}\mu}{4})\|y_{k+1}-y^*(x_{k},\lambda_{k})\|^2+(1+\frac{4}{\eta_k\tilde{\beta}\mu})\|y^*(x_{k},\lambda_{k})-y^*(x_{k+1},\lambda_{k+1})\|^2\nonumber\\
			\le&(1+\frac{\eta_k\tilde{\beta}\mu}{4})\|y_{k+1}-y^*(x_{k},\lambda_{k})\|^2+(2+\frac{8}{\eta_k\tilde{\beta}\mu})\frac{L^2\eta_k^2}{\mu^2}\|\tilde{x}_{k+1}-x_k\|^2+(2+\frac{8}{\eta_k\tilde{\beta}\mu})\frac{L^2\eta_k^2}{\mu^2}\|\tilde{\lambda}_{k+1}-\lambda_k\|^2.\label{azosclem2:11}
		\end{align}
		Combining \eqref{azosclem2:10} and \eqref{azosclem2:11}, we obtain
		\begin{align}
			\|y_{k+1}-y^*(x_{k+1},\lambda_{k+1})\|^2
			\le&(1-\frac{\eta_k\tilde{\beta}\mu}{2})(1+\frac{\eta_k\tilde{\beta}\mu}{4})\|y_{k}-y^*(x_k,\lambda_k)\|^2-\frac{3\eta_k}{4}(1+\frac{\eta_k\tilde{\beta}\mu}{4})\|\tilde{y}_{k+1} -y_k\|^2\nonumber\\
			&+\frac{8\eta_k\tilde{\beta}}{\mu}(1+\frac{\eta_k\tilde{\beta}\mu}{4})\|\widehat{\nabla} _{y}\tilde{\mathcal{L}}_{g,k}(x_{k},y_{k},\lambda_k)-w_k\|^2+\frac{2\eta_k\tilde{\beta}d_yL^2\theta_{2,k}^2}{\mu}(1+\frac{\eta_k\tilde{\beta}\mu}{4})\nonumber\\
			&+(2+\frac{8}{\eta_k\tilde{\beta}\mu})\frac{L^2\eta_k^2}{\mu^2}\|\tilde{x}_{k+1}-x_k\|^2+(2+\frac{8}{\eta_k\tilde{\beta}\mu})\frac{L^2\eta_k^2}{\mu^2}\|\tilde{\lambda}_{k+1}-\lambda_k\|^2.\label{azosclem2:12}
		\end{align}
		Since $0<\eta_k\le1$, $0<\tilde{\beta}\le\frac{1}{6L}$ and $\mu\le L$, we have $\eta_k\tilde{\beta}\mu<1$.Then, we get $(1-\frac{\eta_k\tilde{\beta}\mu}{2})(1+\frac{\eta_k\tilde{\beta}\mu}{4})\le1-\frac{\eta_k\tilde{\beta}\mu}{4}$, $-\frac{3\eta_k}{4}(1+\frac{\eta_k\tilde{\beta}\mu}{4})\le-\frac{3\eta_k}{4}$, $\frac{8\eta_k\tilde{\beta}}{\mu}(1+\frac{\eta_k\tilde{\beta}\mu}{4})\le\frac{10\eta_k\tilde{\beta}}{\mu}$, $\frac{2\eta_k\tilde{\beta}d_yL^2\theta_{2,k}^2}{\mu}(1+\frac{\eta_k\tilde{\beta}\mu}{4})\le\frac{5\eta_k\tilde{\beta}d_yL^2\theta_{2,k}^2}{2\mu}$, $(2+\frac{8}{\eta_k\tilde{\beta}\mu})\frac{L^2\eta_k^2}{\mu^2}\le \frac{10L^2\eta_k}{\mu^3\tilde{\beta}}$.
		Thus, we obtain
		\begin{align}
			&\|y_{k+1}-y^*(x_{k+1},\lambda_{k+1})\|^2\nonumber\\
			\le&(1-\frac{\eta_k\tilde{\beta}\mu}{4})\|y_{k}-y^*(x_k,\lambda_k)\|^2-\frac{3\eta_k}{4}\|\tilde{y}_{k+1} -y_k\|^2+\frac{10L^2\eta_k}{\mu^3\tilde{\beta}}\|\tilde{x}_{k+1}-x_k\|^2\nonumber\\
			&+\frac{10\eta_k\tilde{\beta}}{\mu}\|\widehat{\nabla} _{y}\tilde{\mathcal{L}}_{g,k}(x_{k},y_{k},\lambda_k)-w_k\|^2+\frac{5\eta_k\tilde{\beta}d_yL^2\theta_{2,k}^2}{2\mu}+\frac{10L^2\eta_k}{\mu^3\tilde{\beta}}\|\tilde{\lambda}_{k+1}-\lambda_k\|^2.\label{azosclem2:13}
		\end{align}
		The proof is then completed.
	\end{proof}

	\begin{lemma}\label{azosclem3}
		Suppose that Assumptions \ref{azoass:Lip}, \ref{azoass:var} and  \ref{azoscthetak} hold. Let $\{\left(x_k,y_k,\lambda_k\right)\}$ be a sequence generated by Algorithm \ref{zoalg:2} with parameter settings in \eqref{3.1par}, 
		then $\forall k \ge 1$, 
		\begin{align}
			&\mathbb{E}\|\widehat{\nabla} _{x}\tilde{\mathcal{L}}_{g,k+1}(x_{k+1},y_{k+1},\lambda_{k+1})-v_{k+1}\|^2\nonumber\\
			\le&(1-\varrho_{k})\mathbb{E}\|\widehat{\nabla} _{x}\tilde{\mathcal{L}}_{g,k}(x_{k},y_{k},\lambda_k)-v_k\|^2+\frac{9L^2\eta_k^2}{b}\mathbb{E}[\|\tilde{x}_{k+1}-x_k\|^2+\|\tilde{y}_{k+1}-y_k\|^2+\|\tilde{\lambda}_{k+1}-\lambda_k\|^2]\nonumber\\
			&+\frac{2\varrho_k^2\delta^2}{b}+\frac{3d_xL^2\theta_{1,k}^2}{b},\label{azosclem3:1x}\\
			&\mathbb{E}\|\widehat{\nabla} _{y}\tilde{\mathcal{L}}_{g,k+1}(x_{k+1},y_{k+1},\lambda_{k+1})-w_{k+1}\|^2\nonumber\\
			\le&(1-\iota_{k})\mathbb{E}\|\widehat{\nabla} _{y}\tilde{\mathcal{L}}_{g,k}(x_{k},y_{k},\lambda_k)-w_k\|^2+\frac{9L^2\eta_k^2}{b}\mathbb{E}[\|\tilde{x}_{k+1}-x_k\|^2+\|\tilde{y}_{k+1}-y_k\|^2+\|\tilde{\lambda}_{k+1}-\lambda_k\|^2]\nonumber\\
			&+\frac{2\iota_k^2\delta^2}{b}+\frac{3d_yL^2\theta_{2,k}^2}{b}.\label{azosclem3:1y}
		\end{align}
	\end{lemma}
	
	\begin{proof}
		By $\mathbb{E}\widehat{\nabla} _{x}\tilde{\mathcal{L}}_{G,k}(x,y,\lambda;I_{k+1})=\widehat{\nabla} _{x}\tilde{\mathcal{L}}_{g,k}(x,y,\lambda)$, $\mathbb{E}\widehat{\nabla} _{x}\tilde{\mathcal{L}}_{G,k+1}(x,y,\lambda;I_{k+1})=\widehat{\nabla} _{x}\tilde{\mathcal{L}}_{g,k+1}(x,y,\lambda)$, and the definition of $v_{k+1}$, we have
		\begin{align}
			&\mathbb{E}\|\widehat{\nabla} _{x}\tilde{\mathcal{L}}_{g,k+1}(x_{k+1},y_{k+1},\lambda_{k+1})-v_{k+1}\|^2\nonumber\\
			=&\mathbb{E}\|\widehat{\nabla} _{x}\tilde{\mathcal{L}}_{g,k+1}(x_{k+1},y_{k+1},\lambda_{k+1})-\widehat{\nabla} _x\tilde{\mathcal{L}}_{G,k+1}( x_{k+1},y_{k+1},\lambda_{k+1};I_{k+1})-(1-\varrho_{k})[v_k-\widehat{\nabla} _x\tilde{\mathcal{L}}_{G,k}( x_{k},y_{k},\lambda_k;I_{k+1})]\|^2\nonumber\\
			=&\mathbb{E}\|(1-\varrho_{k})(\widehat{\nabla} _{x}\tilde{\mathcal{L}}_{g,k}(x_{k},y_{k},\lambda_k)-v_k)+\varrho_k(\widehat{\nabla} _{x}\tilde{\mathcal{L}}_{g,k+1}(x_{k+1},y_{k+1},\lambda_{k+1})\nonumber\\
			&-\widehat{\nabla} _x\tilde{\mathcal{L}}_{G,k+1}( x_{k+1},y_{k+1},\lambda_{k+1};I_{k+1}))
			+(1-\varrho_{k})[\widehat{\nabla} _x\tilde{\mathcal{L}}_{g,k+1}( x_{k+1},y_{k+1},\lambda_{k+1})\nonumber\\
			&-\widehat{\nabla} _x\tilde{\mathcal{L}}_{g,k}( x_{k},y_{k},\lambda_k)-\widehat{\nabla} _x\tilde{\mathcal{L}}_{G,k+1}( x_{k+1},y_{k+1},\lambda_{k+1};I_{k+1})+\widehat{\nabla} _x\tilde{\mathcal{L}}_{G,k}( x_{k},y_{k},\lambda_k;I_{k+1})]\|^2\nonumber\\
			=&(1-\varrho_{k})^2\mathbb{E}\|\widehat{\nabla} _{x}\tilde{\mathcal{L}}_{g,k}(x_{k},y_{k},\lambda_k)-v_k\|^2+\mathbb{E}\|\varrho_k(\widehat{\nabla} _{x}\tilde{\mathcal{L}}_{g,k+1}(x_{k+1},y_{k+1},\lambda_{k+1})\nonumber\\
			&-\widehat{\nabla} _x\tilde{\mathcal{L}}_{G,k+1}( x_{k+1},y_{k+1},\lambda_{k+1};I_{k+1}))
			+(1-\varrho_{k})[\widehat{\nabla} _x\tilde{\mathcal{L}}_{g,k+1}( x_{k+1},y_{k+1},\lambda_{k+1})\nonumber\\
			&-\widehat{\nabla} _x\tilde{\mathcal{L}}_{g,k}( x_{k},y_{k},\lambda_k)-\widehat{\nabla} _x\tilde{\mathcal{L}}_{G,k+1}( x_{k+1},y_{k+1},\lambda_{k+1};I_{k+1})+\widehat{\nabla} _x\tilde{\mathcal{L}}_{G,k}( x_{k},y_{k},\lambda_k;I_{k+1})]\|^2.\label{azosclem3:2}
		\end{align}
		By Assumption \ref{azoass:var}, the fact that $\mathbb{E}\|\zeta-\mathbb{E}\zeta\|^2=\mathbb{E}\|\zeta\|^2-\|\mathbb{E}\zeta\|^2\le\mathbb{E}\|\zeta\|^2$ and $\mathbb{E}\|\frac{1}{b}\sum_{j=1}^{b}\zeta_j\|^2=\frac{1}{b}\mathbb{E}\|\zeta_j\|^2$ for i.i.d. random variables $\{\zeta_j\}_{j=1}^b$ with zero mean, we have
		\begin{align}
			&\mathbb{E}\|\widehat{\nabla} _{x}\tilde{\mathcal{L}}_{g,k+1}(x_{k+1},y_{k+1},\lambda_{k+1})-v_{k+1}\|^2\nonumber\\
			\le&\frac{2(1-\varrho_k)^2}{b}\mathbb{E}\|\widehat{\nabla} _x\tilde{\mathcal{L}}_{G,k+1}( x_{k+1},y_{k+1},\lambda_{k+1};\zeta_1^{k+1})-\widehat{\nabla} _x\tilde{\mathcal{L}}_{G,k}( x_{k},y_{k},\lambda_k;\zeta_1^{k+1})\|^2.\nonumber\\
			&	+	(1-\varrho_{k})^2\mathbb{E}\|\widehat{\nabla} _{x}\tilde{\mathcal{L}}_{g,k}(x_{k},y_{k},\lambda_k)-v_k\|^2+\frac{2\varrho_k^2\delta^2}{b}
			\label{azosclem3:3}
		\end{align}
		By the definition of $\widehat{\nabla} _x\tilde{\mathcal{L}}_{G,k}( x,y,\lambda;\zeta)$, Assumptions \ref{azoass:Lip}, \ref{azoscthetak} ,\eqref{zo2:update-x} and lemma \ref{azovb}, we get
		\begin{align}
			&\mathbb{E}\left\|\widehat{\nabla} _x\tilde{\mathcal{L}}_{G,k+1}( x_{k+1},y_{k+1},\lambda_{k+1};\zeta_1^{k+1})-\widehat{\nabla} _x\tilde{\mathcal{L}}_{G,k}( x_{k},y_{k},\lambda_k;\zeta_1^{k+1})\right\|^2\nonumber\\
			=&\mathbb{E}\left\|\widehat{\nabla} _x\tilde{\mathcal{L}}_{G,k+1}( x_{k+1},y_{k+1},\lambda_{k+1};\zeta_1^{k+1})-\nabla_x\tilde{\mathcal{L}}_{G,k+1}( x_{k+1},y_{k+1},\lambda_{k+1};\zeta_1^{k+1})\right.\nonumber\\
			&\left.+\nabla _x\tilde{\mathcal{L}}_{G,k}( x_{k},y_{k},\lambda_k;\zeta_1^{k+1})-\widehat{\nabla} _x\tilde{\mathcal{L}}_{G,k}( x_{k},y_{k},\lambda_k;\zeta_1^{k+1})+\nabla_x\tilde{\mathcal{L}}_{G,k+1}( x_{k+1},y_{k+1},\lambda_{k+1};\zeta_1^{k+1})\right.\nonumber\\
			&\left.-\nabla _x\tilde{\mathcal{L}}_{G,k}( x_{k},y_{k},\lambda_k;\zeta_1^{k+1})\right\|^2\nonumber\\
			\le&\frac{3d_xL^2\theta_{1,k+1}^2}{4}+3\mathbb{E}\|\nabla_x\tilde{\mathcal{L}}_{G,k+1}( x_{k+1},y_{k+1},\lambda_{k+1};\zeta_1^{k+1})-\nabla _x\tilde{\mathcal{L}}_{G,k}( x_{k},y_{k},\lambda_k;\zeta_1^{k+1})\|^2+\frac{3d_xL^2\theta_{1,k}^2}{4}\nonumber\\
			\le&\frac{3d_xL^2\theta_{1,k}^2}{2}+9L^2\eta_k^2\mathbb{E}[\|\tilde{x}_{k+1}-x_k\|^2+\|\tilde{y}_{k+1}-y_k\|^2+\|\tilde{\lambda}_{k+1}-\lambda_k\|^2].\label{azosclem3:4}
		\end{align}
		Plugging \eqref{azosclem3:4} into \eqref{azosclem3:3} and by $1-\varrho_k<1$, we have
		\begin{align*}
			&\mathbb{E}\|\widehat{\nabla} _{x}\tilde{\mathcal{L}}_{g,k+1}(x_{k+1},y_{k+1},\lambda_{k+1})-v_{k+1}\|^2\nonumber\\
			\le&(1-\varrho_{k})\mathbb{E}\|\widehat{\nabla} _{x}\tilde{\mathcal{L}}_{g,k}(x_{k},y_{k},\lambda_k)-v_k\|^2+\frac{9L^2\eta_k^2}{b}\mathbb{E}[\|\tilde{x}_{k+1}-x_k\|^2+\|\tilde{y}_{k+1}-y_k\|^2+\|\tilde{\lambda}_{k+1}-\lambda_k\|^2]\nonumber\\
			&+\frac{2\varrho_k^2\delta^2}{b}+\frac{3d_xL^2\theta_{1,k}^2}{b}.
		\end{align*}
		Similarly, we can prove \eqref{azosclem3:1y} and then the proof is completed.
	\end{proof}
	
	\begin{lemma}\label{azosclem4}
		Suppose that Assumptions \ref{azoass:Lip}, \ref{azoass:var} and \ref{azoscthetak} hold. Let $\{\left(x_k,y_k,\lambda_k\right)\}$ be a sequence generated by Algorithm \ref{zoalg:2} with parameter settings in \eqref{3.1par}. Denote
		\begin{align*}
			F_{k+1}(x_{k+1},y_{k+1},\lambda_{k+1})
			=&\mathbb{E}[\Phi(x_{k+1},\lambda_{k+1})+\frac{32\tilde{\alpha}L^2}{\tilde{\beta}\mu}\|y_{k+1}-y^*(x_{k+1},\lambda_{k+1})\|^2\\
			&+\frac{{\color{black}A_1}}{\eta_k}\|\widehat{\nabla} _{x}\tilde{\mathcal{L}}_{g,k+1}(x_{k+1},y_{k+1},\lambda_{k+1})-v_{k+1}\|^2\\
			&+\frac{{\color{black}A_2}}{\eta_k}\|\widehat{\nabla} _{y}\tilde{\mathcal{L}}_{g,k+1}(x_{k+1},y_{k+1},\lambda_{k+1})-w_{k+1}\|^2],\\
			S_{k+1}(x_{k+1},y_{k+1},\lambda_{k+1})
			=&F_{k+1}(x_{k+1},y_{k+1},\lambda_{k+1})-4\eta_k\tilde{\alpha} L^2\mathbb{E}\|y_{k+1}-y^*(x_{k+1},\lambda_{k+1})\|^2,
		\end{align*}
		where ${\color{black}A_1}\le\min\{\frac{b}{72\tilde{\alpha}L^2},\frac{11\tilde{\alpha}b}{9\mu\tilde{\beta}},\frac{b}{144\tilde{\gamma}L^2}\}$ and ${\color{black}A_2}\le\min\{\frac{b}{72\tilde{\alpha}L^2},\frac{11\tilde{\alpha}b}{9\mu\tilde{\beta}},\frac{b}{144\tilde{\gamma}L^2}\}$. If $0<\eta_k\le\min\{1,\frac{\mu}{8\tilde{\alpha}L^2},\frac{\mu}{4\tilde{\gamma}L^2}\}$, $0<\tilde{\beta}\le\frac{1}{6L}$, $0<\tilde{\alpha}\le\frac{\mu^2\tilde{\beta}}{64L^2}$ and $0<\tilde{\gamma}\le\min\{\frac{2\tilde{\alpha}L^2}{\|B\|^2},\frac{\mu^4\tilde{\beta}^2}{1280L^4\tilde{\alpha}}\}$,
		then $\forall k \ge 1$, 
		\begin{align}\label{azosclem4:1}
			&S_{k+1}(x_{k+1},y_{k+1},\lambda_{k+1})-S_{k}(x_{k},y_{k},\lambda_{k})\nonumber\\
			\le&4\tilde{\alpha} L^2(\eta_{k-1}-\eta_k)\mathbb{E}\|y_k-y^*(x_k,\lambda_k)\|^2-\frac{\eta_k}{4\tilde{\alpha}}\mathbb{E}\|\tilde{x}_{k+1}-x_k\|^2-\frac{2\tilde{\alpha}L^2\eta_k}{\mu\tilde{\beta}}\mathbb{E}\|\tilde{y}_{k+1} -y_k\|^2-\frac{\eta_k}{4\tilde{\gamma}}\mathbb{E}\|\tilde{\lambda}_{k+1}-\lambda_k\|^2\nonumber\\
			&+(2\eta_k\tilde{\alpha}-\frac{{\color{black}A_1}\varrho_{k}}{\eta_k}+\frac{{\color{black}A_1}}{\eta_k}-\frac{{\color{black}A_1}}{\eta_{k-1}})\mathbb{E}\|\widehat{\nabla} _{x}\tilde{\mathcal{L}}_{g,k}(x_{k},y_{k},\lambda_k)-v_k\|^2\nonumber\\
			&+(\frac{320\eta_k\tilde{\alpha}L^2}{\mu^2}-\frac{{\color{black}A_2}\iota_{k}}{\eta_k}+\frac{{\color{black}A_2}}{\eta_k}-\frac{{\color{black}A_2}}{\eta_{k-1}})\mathbb{E}\|\widehat{\nabla} _{y}\tilde{\mathcal{L}}_{g,k}(x_{k},y_{k},\lambda_k)-w_k\|^2\nonumber\\
			&+\eta_k\tilde{\alpha} d_xL^2\theta_{1,k}^2+\frac{2\delta^2(\varrho_k^2{\color{black}A_1}+\iota_k^2{\color{black}A_2})}{b\eta_k}+\frac{3L^2(d_x\theta_{1,k}^2{\color{black}A_1}+d_y\theta_{2,k}^2{\color{black}A_2})}{b\eta_k}+\frac{80\eta_k\tilde{\alpha}d_yL^4\theta_{2,k}^2}{\mu^2}.
		\end{align}
	\end{lemma}
	
	\begin{proof}
		By using Lemmas \ref{azosclem1}, \ref{azosclem2}, \ref{azosclem3}, and the definition of $F_{k+1}(x_{k+1},y_{k+1},\lambda_{k+1})$, we get
		\begin{align}\label{azosclem4:2}
			&F_{k+1}(x_{k+1},y_{k+1},\lambda_{k+1})-F_{k}(x_{k},y_{k},\lambda_{k})\nonumber\\
			\le&4\eta_k\tilde{\alpha} L^2\mathbb{E}\|y_{k+1}-y^*(x_{k+1},\lambda_{k+1})\|^2-4\eta_k\tilde{\alpha} L^2\mathbb{E}\|y_k-y^*(x_k,\lambda_k)\|^2\nonumber\\
			&+(2\eta_k\tilde{\alpha}-\frac{{\color{black}A_1}\varrho_{k}}{\eta_k}+\frac{{\color{black}A_1}}{\eta_k}-\frac{{\color{black}A_1}}{\eta_{k-1}})\mathbb{E}\|\widehat{\nabla} _{x}\tilde{\mathcal{L}}_{g,k}(x_{k},y_{k},\lambda_k)-v_k\|^2\nonumber\\
			&+(\frac{320\eta_k\tilde{\alpha}L^2}{\mu^2}-\frac{{\color{black}A_2}\iota_{k}}{\eta_k}+\frac{{\color{black}A_2}}{\eta_k}-\frac{{\color{black}A_2}}{\eta_{k-1}})\mathbb{E}\|\widehat{\nabla} _{y}\tilde{\mathcal{L}}_{g,k}(x_{k},y_{k},\lambda_k)-w_k\|^2\nonumber\\
			&-(\frac{3\eta_k}{4\tilde{\alpha}}-\frac{L^2\eta_k^2}{\mu}-\frac{320L^4\eta_k\tilde{\alpha}}{\mu^4\tilde{\beta}^2}-\frac{9L^2\eta_k({\color{black}A_1+A_2})}{b})\mathbb{E}\|\tilde{x}_{k+1}-x_k\|^2\nonumber\\
			&-(\frac{24\tilde{\alpha}L^2\eta_k}{\mu\tilde{\beta}}-\frac{9L^2\eta_k({\color{black}A_1+A_2})}{b})\mathbb{E}\|\tilde{y}_{k+1} -y_k\|^2+\eta_k\tilde{\alpha} d_xL^2\theta_{1,k}^2\nonumber\\
			&-(\frac{\eta_k}{\tilde{\gamma}}-\frac{\|B\|^2\eta_k}{16\tilde{\alpha}L^2}-\frac{\eta_k^2L^2}{\mu}-\frac{320L^4\eta_k\tilde{\alpha}}{\mu^4\tilde{\beta}^2}-\frac{9L^2\eta_k({\color{black}A_1+A_2})}{b})\mathbb{E}\|\tilde{\lambda}_{k+1}-\lambda_k\|^2+\frac{80\eta_k\tilde{\alpha}d_yL^4\theta_{2,k}^2}{\mu^2}\nonumber\\
			&+\frac{2\delta^2(\varrho_k^2{\color{black}A_1}+\iota_k^2{\color{black}A_2})}{b\eta_k}+\frac{3L^2(d_x\theta_{1,k}^2{\color{black}A_1}+d_y\theta_{2,k}^2{\color{black}A_2})}{b\eta_k}.
		\end{align}
		By the definition of ${\color{black}A_1}$, ${\color{black}A_2}$, $\eta_k$, $\tilde{\alpha}$ and $\tilde{\gamma}$, we have $\frac{9L^2\eta_k({\color{black}A_1+A_2})}{b}\le\frac{\eta_k}{4\tilde{\alpha}}$, $\frac{320L^4\eta_k\tilde{\alpha}}{\mu^4\tilde{\beta}^2}\le\frac{\eta_k}{8\tilde{\alpha}}$, $\frac{\eta_k^2L^2}{\mu}\le\frac{\eta_k}{8\tilde{\alpha}}$, $\frac{\|B\|^2\eta_k}{16\tilde{\alpha}L^2}\le\frac{\eta_k}{8\tilde{\gamma}}$, $\frac{320L^4\eta_k\tilde{\alpha}}{\mu^4\tilde{\beta}^2}\le\frac{\eta_k}{4\tilde{\gamma}}$, $\frac{\eta_k^2L^2}{\mu}\le\frac{\eta_k}{4\tilde{\gamma}}$, $\frac{9L^2\eta_k({\color{black}A_1+A_2})}{b}\le\frac{\eta_k}{8\tilde{\gamma}}$, $\frac{9L^2\eta_k({\color{black}A_1+A_2})}{b}\le\frac{22\tilde{\alpha}L^2\eta_k}{\mu\tilde{\beta}}$. Then, by the definition of $S_{k+1}(x_{k+1},y_{k+1},\lambda_{k+1})$, we complete the proof.
	\end{proof}

	\begin{lemma}\label{azosclem5}
		Suppose that Assumption \ref{azoass:Lip} holds. Let $\{\left(x_k,y_k,\lambda_k\right)\}$ be a sequence generated by Algorithm \ref{zoalg:2} with parameter settings in \eqref{3.1par},
		then $\forall k \ge 1$, 
		\begin{align}\label{azosclem5:1}
			&\mathbb{E}\|\nabla  \mathcal{G}_k^{\tilde{\alpha},\tilde{\beta} ,\tilde{\gamma} }(x_k,y_k,\lambda_k)\|^2\nonumber\\
			\le&(\frac{3}{\tilde{\alpha}^2}+3\|A\|^2\eta_k^2) \mathbb{E}\|\tilde{x}_{k+1}-x_k\|^2+(\frac{3}{\tilde{\beta}^2}+3\|B\|^2\eta_k^2) \mathbb{E}\|\tilde{y}_{k+1}-y_k\|^2+\frac{3}{\tilde{\gamma}^2}\mathbb{E}\|\tilde{\lambda}_{k+1}-\lambda_k\|^2\nonumber\\
			&+3\mathbb{E}\|\widehat{\nabla}_x\tilde{\mathcal{L}}_{g,k}\left( x_k,y_{k},\lambda_k \right)-v_k\|^2+3\mathbb{E}\|\widehat{\nabla}_y\tilde{\mathcal{L}}_{g,k}\left( x_k,y_{k},\lambda_k \right)-w_k\|^2+\frac{3d_xL^2\theta_{1,k}^2}{4}+\frac{3d_yL^2\theta_{2,k}^2}{4}.
		\end{align}
	\end{lemma}
	
	\begin{proof}
		By \eqref{zo2:update-y} and using the nonexpansive property of the projection operator, we immediately get
		\begin{align}
			&\left\| \frac{1}{\tilde{\beta}} \left( y_k-{\color{black}\mathcal{P}_{\mathcal{Y}}}\left( y_k+ \tilde{\beta}\nabla _y\mathcal{L}_g\left( x_k,y_k,\lambda_k\right) \right) \right) \right\|\nonumber\\
			\le&\frac{1}{\tilde{\beta}} \|\tilde{y}_{k+1}-y_k\|+\|\nabla _y\mathcal{L}_g\left( x_k,y_{k},\lambda_k \right)-w_k\|\nonumber\\
			\le&\|\widehat{\nabla}_y\tilde{\mathcal{L}}_{g,k}\left( x_k,y_{k},\lambda_k \right)-w_k\|+\|\nabla _y\mathcal{L}_g\left( x_k,y_{k},\lambda_k \right)-\widehat{\nabla}_y\tilde{\mathcal{L}}_{g,k}\left( x_k,y_{k},\lambda_k \right)\|+\frac{1}{\tilde{\beta}} \|\tilde{y}_{k+1}-y_k\|.\label{azosclem5:2}
		\end{align}
		By \eqref{zo2:update-x} and the nonexpansive property of the projection operator, we have
		\begin{align}
			&\left\| \frac{1}{\tilde{\alpha}} \left( x_k-{\color{black}\mathcal{P}_{\mathcal{X}}}\left( x_k- \tilde{\alpha}\nabla _x\mathcal{L}_g\left( x_k,y_k,\lambda_k\right) \right) \right) \right\|\nonumber\\
			\le&\frac{1}{\tilde{\alpha}} \|\tilde{x}_{k+1}-x_k\|+\|\nabla _x\mathcal{L}_g\left( x_k,y_{k},\lambda_k \right)-v_k\|\nonumber\\
			\le&\|\widehat{\nabla}_x\tilde{\mathcal{L}}_{g,k}\left( x_k,y_{k},\lambda_k \right)-v_k\|+\|\nabla _x\mathcal{L}_g\left( x_k,y_{k},\lambda_k \right)-\widehat{\nabla}_x\tilde{\mathcal{L}}_{g,k}\left( x_k,y_{k},\lambda_k \right)\|+\frac{1}{\tilde{\alpha}} \|\tilde{x}_{k+1}-x_k\|.\label{azosclem5:3}
		\end{align}
		By \eqref{zo2:update-lambda} and the nonexpansive property of the projection operator,  we obtain
		\begin{align}
			&\left\|\frac{1}{\tilde{\gamma}}\left( \lambda_k -{\color{black}\mathcal{P}_{\Lambda}}\left(\lambda_k-\tilde{\gamma}\nabla_{\lambda}\mathcal{L}_g(x_k,y_k,\lambda_k)\right)\right)\right\|\nonumber\\
			\le &\frac{1}{\tilde{\gamma}}\left\| {\color{black}\mathcal{P}_{\Lambda}}\left(\lambda_k-\tilde{\gamma}\nabla_{\lambda}\mathcal{L}_g(x_{k+1},y_{k+1},\lambda_k)\right)-
			{\color{black}\mathcal{P}_{\Lambda}}\left(\lambda_k-\tilde{\gamma}\nabla_{\lambda}\mathcal{L}_g(x_k,y_k,\lambda_k)\right)\right\|+\frac{1}{\tilde{\gamma}}\| \tilde{\lambda}_{k+1}-\lambda_k\|\nonumber\\
			\leq & \frac{1}{\tilde{\gamma}}\|\tilde{\lambda}_{k+1}-\lambda_k\|+\|A\|\eta_k\|\tilde{x}_{k+1}-x_k\| +\|B\|\eta_k\|\tilde{y}_{k+1}-y_k\|.\label{azosclem5:4}
		\end{align}
		Combing \eqref{azosclem5:2}, \eqref{azosclem5:3}, \eqref{azosclem5:4} and Lemma \ref{azovb}, using Cauchy-Schwarz inequality and taking the expectation leads to the desired result.	
	\end{proof}
	
	Define $T_2(\varepsilon):=\min\{k \mid \mathbb{E}\|\nabla  \mathcal{G}_k^{\tilde{\alpha},\tilde{\beta} ,\tilde{\gamma} }(x_k,y_k,\lambda_k)\|^2 \leq \varepsilon^2 \}$ and $\varepsilon>0$ is a given target accuracy. We then obtain the following theorem which provides a bound on $T_2(\varepsilon)$.
	
	\begin{theorem}\label{azoscthm1}
		Suppose that Assumptions \ref{fea}, \ref{azoass:Lip}, \ref{azoass:var} and \ref{azoscthetak} hold. Let $\{\left(x_k,y_k,\lambda_k\right)\}$ be a sequence generated by Algorithm \ref{zoalg:2}. Set  $\eta_k=\frac{{\color{black}A_3}}{(k+2)^{1/3}}$, $\varrho_k=\frac{{\color{black}A_4}}{(k+2)^{2/3}}$,$\iota_k=\frac{{\color{black}A_5}}{(k+2)^{2/3}}$, $\theta_{1,k}=\frac{\varepsilon}{L(k+2)^{1/3}}\sqrt{\frac{{\color{black}D_1}}{d_x(3{\color{black}D_1}+4\tilde{\alpha}+\frac{12{\color{black}A_1}}{b{\color{black}A_3^2}})}}$, $\theta_{2,k}=\frac{\varepsilon}{L(k+2)^{1/3}}\sqrt{\frac{{\color{black}D_1}}{d_y(3{\color{black}D_1}+\frac{12{\color{black}A_2}}{b{\color{black}A_3^2}}+\frac{320\tilde{\alpha}L^2}{\mu^2})}}$ with ${\color{black}A_3}\le\min\{1,\frac{\mu}{8\tilde{\alpha}L^2},\frac{\mu}{4\tilde{\gamma}L^2}\}$, ${\color{black}A_4}\ge\frac{3{\color{black}A_3^2}\tilde{\alpha}+{\color{black}A_1}}{{\color{black}A_1}}$, ${\color{black}A_5}\ge\frac{480{\color{black}A_3^2}\tilde{\alpha}L^2+{\color{black}A_2}\mu^2}{{\color{black}A_2}\mu^2}$, ${\color{black}A_1}\le\min\{\frac{b}{72\tilde{\alpha}L^2},\frac{11\tilde{\alpha}b}{9\mu\tilde{\beta}},\frac{b}{144\tilde{\gamma}L^2}\}$ and ${\color{black}A_2}\le\min\{\frac{b}{72\tilde{\alpha}L^2},\frac{11\tilde{\alpha}b}{9\mu\tilde{\beta}},\frac{b}{144\tilde{\gamma}L^2}\}$. If $0<\tilde{\beta}\le\frac{1}{6L}$, $0<\tilde{\alpha}\le\frac{\mu^2\tilde{\beta}}{64L^2}$ and $0<\tilde{\gamma}\le\min\{\frac{2\tilde{\alpha}L^2}{\|B\|^2},\frac{\mu^4\tilde{\beta}^2}{1280L^4\tilde{\alpha}}\}$, then for any given $\varepsilon>0$, 
		$$\frac{\varepsilon^2}{4}\le\frac{C_1+C_2\ln(T_2(\varepsilon)+2)}{{\color{black}D_1A_3}(\frac{3}{2}(T_2(\varepsilon)+3)^{2/3}-\frac{3\cdot4^{2/3}}{2})},$$
		where $C_1=S_{2}(x_{2},y_{2},\lambda_2)-\underbar{S}+16\tilde{\alpha} L^2\eta_1\sigma_y^2$, $C_2=\frac{2\delta^2({\color{black}A_4^2A_1+A_5^2A_2})}{b{\color{black}A_3}}$ with $\underbar{S}=\min\limits_{\lambda\in\Lambda}\min\limits_{x\in\mathcal{X}}\min\limits_{y\in\mathcal{Y}}S_k(x,y,\lambda)$, ${\color{black}D_1}=\frac{\min\{\frac{1}{4\tilde{\alpha}},\frac{2\tilde{\alpha}L^2}{\mu\tilde{\beta}},\frac{1}{4\tilde{\gamma}},\frac{{\color{black}A_1A_4}}{3{\color{black}A_3^2}},\frac{{\color{black}A_2A_5}}{3{\color{black}A_3^2}}\}}{\max\{\frac{3}{\tilde{\alpha}^2}+3\|A\|^2,\frac{3}{\tilde{\beta}^2}+3\|B\|^2,\frac{3}{\tilde{\gamma}^2},3\}}$ and $\sigma_y=\max\{\|y\| \mid y\in\mathcal{Y}\}$.
		Moreover, for any  $i\in\mathcal{K}$, when {\color{black}$k=T_2(\varepsilon)$}, we have
		$$\max\{0,[Ax_{k}+By_{k}-c]_i\}\leq \varepsilon.$$
	\end{theorem}
	
	\begin{proof}
		Firstly, by the definition of $\eta_k$, we  have
		\begin{align}\label{azoscthm1:2}
			\frac{1}{\eta_k}-\frac{1}{\eta_{k-1}}&\le\frac{(k+1)^{-2/3}}{3{\color{black}A_3}}\le\frac{2^{2/3}(2+k)^{-2/3}}{3{\color{black}A_3}}\le\frac{2\eta_k}{3{\color{black}A_3^2}}.
		\end{align}
		By the definition of $\varrho_k$, $\iota_k$, ${\color{black}A_4}$, ${\color{black}A_5}$ and \eqref{azoscthm1:2}, we get
		\begin{align*}
			&2\eta_k\tilde{\alpha}-\frac{{\color{black}A_1}\varrho_{k}}{\eta_k}+\frac{{\color{black}A_1}}{\eta_k}-\frac{{\color{black}A_1}}{\eta_{k-1}}\le\frac{6{\color{black}A_3^2}\tilde{\alpha}+2{\color{black}A_1}-3{\color{black}A_1A_4}}{3{\color{black}A_3^2}}\eta_k\le-\frac{{\color{black}A_1A_4}\eta_k}{3{\color{black}A_3^2}},\nonumber\\
			&\frac{320\eta_k\tilde{\alpha}L^2}{\mu^2}-\frac{{\color{black}A_2}\iota_{k}}{\eta_k}+\frac{{\color{black}A_2}}{\eta_k}-\frac{{\color{black}A_2}}{\eta_{k-1}}\le-\frac{{\color{black}A_2A_5}\eta_k}{3{\color{black}A_3^2}}.
		\end{align*}
		In view of Lemma \ref{azosclem4}, then we immediately obtian
		\begin{align}\label{azoscthm1:3}
			&S_{k+1}(x_{k+1},y_{k+1},\lambda_{k+1})-S_{k}(x_{k},y_{k},\lambda_{k})\nonumber\\
			\le&-\frac{\eta_k}{4\tilde{\alpha}}\mathbb{E}\|\tilde{x}_{k+1}-x_k\|^2-\frac{2\tilde{\alpha}L^2\eta_k}{\mu\tilde{\beta}}\mathbb{E}\|\tilde{y}_{k+1} -y_k\|^2-\frac{\eta_k}{4\tilde{\gamma}}\mathbb{E}\|\tilde{\lambda}_{k+1}-\lambda_k\|^2-\frac{{\color{black}A_1A_4}\eta_k}{3{\color{black}A_3^2}}\mathbb{E}\|\widehat{\nabla} _{x}\tilde{\mathcal{L}}_{g,k}(x_{k},y_{k},\lambda_k)-v_k\|^2\nonumber\\
			&-\frac{{\color{black}A_2A_5}\eta_k}{3{\color{black}A_3^2}}\mathbb{E}\|\widehat{\nabla} _{y}\tilde{\mathcal{L}}_{g,k}(x_{k},y_{k},\lambda_k)-w_k\|^2+\eta_k\tilde{\alpha} d_xL^2\theta_{1,k}^2+\frac{80\eta_k\tilde{\alpha}d_yL^4\theta_{2,k}^2}{\mu^2}+\frac{2\delta^2(\varrho_k^2{\color{black}{\color{black}A_1}}+\iota_k^2{\color{black}A_2})}{b\eta_k}\nonumber\\
			&+\frac{3L^2(d_x\theta_{1,k}^2{\color{black}A_1}+d_y\theta_{2,k}^2{\color{black}A_2})}{b\eta_k}+16\tilde{\alpha} L^2(\eta_{k-1}-\eta_k)\sigma_y^2.
		\end{align}
		By Lemma \ref{azosclem5} and $\eta_k\le1$, we have
		\begin{align}\label{azoscthm1:4}
			&\mathbb{E}\|\nabla  \mathcal{G}_k^{\tilde{\alpha},\tilde{\beta} ,\tilde{\gamma} }(x_k,y_k,\lambda_k)\|^2\nonumber\\
			\le&(\frac{3}{\tilde{\alpha}^2}+3\|A\|^2) \mathbb{E}\|\tilde{x}_{k+1}-x_k\|^2+(\frac{3}{\tilde{\beta}^2}+3\|B\|^2) \mathbb{E}\|\tilde{y}_{k+1}-y_k\|^2+\frac{3}{\tilde{\gamma}^2}\mathbb{E}\|\tilde{\lambda}_{k+1}-\lambda_k\|^2\nonumber\\
			&+3\mathbb{E}\|\widehat{\nabla}_x\tilde{\mathcal{L}}_{g,k}\left( x_k,y_{k},\lambda_k \right)-v_k\|^2+3\mathbb{E}\|\widehat{\nabla}_y\tilde{\mathcal{L}}_{g,k}\left( x_k,y_{k},\lambda_k \right)-w_k\|^2+\frac{3d_xL^2\theta_{1,k}^2}{4}+\frac{3d_yL^2\theta_{2,k}^2}{4}.
		\end{align}
		It follows from the definition of ${\color{black}D_1}$ and \eqref{azoscthm1:3} that $\forall k\ge 1$,
		\begin{align}\label{azoscthm1:5}
			&{\color{black}D_1}\eta_k\mathbb{E}\|\nabla  \mathcal{G}_k^{\tilde{\alpha},\tilde{\beta} ,\tilde{\gamma} }(x_k,y_k,\lambda_k)\|^2\nonumber\\
			\le&S_{k}(x_{k},y_{k},\lambda_{k})-S_{k+1}(x_{k+1},y_{k+1},\lambda_{k+1})+\frac{3d_x{\color{black}D_1}L^2\theta_{1,k}^2\eta_k}{4}+\frac{3d_y{\color{black}D_1}L^2\theta_{2,k}^2\eta_k}{4}+\eta_k\tilde{\alpha} d_xL^2\theta_{1,k}^2\nonumber\\
			&+\frac{2\delta^2(\varrho_k^2{\color{black}A_1}+\iota_k^2{\color{black}A_2})}{b\eta_k}+\frac{3L^2(d_x\theta_{1,k}^2{\color{black}A_1}+d_y\theta_{2,k}^2{\color{black}A_2})}{b\eta_k}+16\tilde{\alpha} L^2(\eta_{k-1}-\eta_k)\sigma_y^2+\frac{80\eta_k\tilde{\alpha}d_yL^4\theta_{2,k}^2}{\mu^2}.
		\end{align}
		{\color{black}By arguments similar to those for \eqref{t:1} and \eqref{t:2}, $T_2(\varepsilon)$ is well-defined.} By summing both sides of \eqref{azoscthm1:5} from $k=2$ to $T_2(\varepsilon)$, we obtain
		\begin{align}
			&\sum_{k=2}^{T_2(\varepsilon)}{\color{black}D_1}\eta_k\mathbb{E}\|\nabla  \mathcal{G}_k^{\tilde{\alpha},\tilde{\beta} ,\tilde{\gamma} }(x_k,y_k,\lambda_k)\|^2\nonumber\\
			\le& S_{2}(x_{2},y_{2},\lambda_2)-S_{T_2(\varepsilon)+1}(x_{T_2(\varepsilon)+1},y_{T_2(\varepsilon)+1},\lambda_{T_2(\varepsilon)+1})+16\tilde{\alpha} L^2\eta_1\sigma_y^2\nonumber\\
			&+\sum_{k=2}^{T_2(\varepsilon)}\eta_k(\frac{3d_x{\color{black}D_1}L^2\theta_{1,k}^2}{4}+\tilde{\alpha} d_xL^2\theta_{1,k}^2+\frac{3L^2d_x{\color{black}A_1}\theta_{1,k}^2}{b\eta_k^2})\nonumber\\
			&+\sum_{k=2}^{T_2(\varepsilon)}\eta_k(\frac{3d_y{\color{black}D_1}L^2\theta_{2,k}^2}{4}+\frac{3L^2d_y{\color{black}A_2}\theta_{2,k}^2}{b\eta_k^2}+\frac{80\tilde{\alpha}d_yL^4\theta_{2,k}^2}{\mu^2})+\sum_{k=2}^{T_2(\varepsilon)}\frac{2\delta^2(\varrho_k^2{\color{black}A_1}+\iota_k^2{\color{black}A_2})}{b\eta_k}.\label{azoscthm1:6}
		\end{align}
		Denote $\underbar{S}=\min\limits_{\lambda\in\Lambda}\min\limits_{x\in\mathcal{X}}\min\limits_{y\in\mathcal{Y}}S_k(x,y,\lambda)$. By the definition of $\theta_{1,k}^2$ and $\theta_{2,k}^2$, we then conclude from \eqref{azoscthm1:6} that
		\begin{align}
			&\sum_{k=2}^{T_2(\varepsilon)}{\color{black}D_1}\eta_k\mathbb{E}\|\nabla  \mathcal{G}_k^{\tilde{\alpha},\tilde{\beta} ,\tilde{\gamma} }(x_k,y_k,\lambda_k)\|^2\nonumber\\
			\le& S_{2}(x_{2},y_{2},\lambda_2)-\underbar{S}+16\tilde{\alpha} L^2\eta_1\sigma_y^2+\sum_{k=2}^{T_2(\varepsilon)}\frac{{\color{black}D_1}\varepsilon^2\eta_k}{4}+\sum_{k=2}^{T_2(\varepsilon)}\frac{{\color{black}D_1}\varepsilon^2\eta_k}{4}+\sum_{k=2}^{T_2(\varepsilon)}\frac{2\delta^2({\color{black}A_4^2A_1}+{\color{black}A_5^2A_2})}{b{\color{black}A_3}}(k+2)^{-1}.\label{azoscthm1:7}
		\end{align}
		Since $\sum_{k=2}^{T_2(\varepsilon)}(k+2)^{-1}\le\ln(T_2(\varepsilon)+2)$ and $\sum_{k=2}^{T_2(\varepsilon)}(k+2)^{-1/3}\ge\frac{3}{2}(T_2(\varepsilon)+2)^{2/3}-\frac{3\cdot4^{2/3}}{2}$, by the definition of $C_1$ and $C_2$, we get
		\begin{align}
			\frac{\varepsilon^2}{4}\le\frac{C_1+C_2\ln(T_2(\varepsilon)+2)}{{\color{black}D_1A_3}(\frac{3}{2}(T_2(\varepsilon)+3)^{2/3}-\frac{3\cdot4^{2/3}}{2})}.\label{azoscthm1:8}
		\end{align}
		We complete the proof. 
	\end{proof}
	
	\begin{remark}
		Denote $\kappa=L/\mu$. If we set $\tilde{\beta}=\frac{1}{6L}$, by Theorem \ref{azoscthm1} we have that $\tilde{\alpha}=\mathcal{O}(\frac{1}{\kappa^2})$ and $\tilde{\gamma}=\mathcal{O}(\frac{1}{\kappa^2})$, then it can be easily verified that the number of iterations for Algorithm \ref{zoalg:2} to obtain an $\varepsilon$-stationary point of problem  \eqref{problem:s} is upper bounded by $\tilde{\mathcal{O}}\left(\kappa^{4.5}\varepsilon ^{-3} \right)$ for solving stochastic nonconvex-strongly concave minimax problem with coupled linear constraints. By Theorem \ref{azoscthm1} we have that $b=\mathcal{O}(1)$ which implies that the total number of function value queries are bounded by $\tilde{\mathcal{O}}\left((d_x+d_y)\kappa^{4.5}\varepsilon ^{-3}  \right)$ for solving the stochastic nonconvex-strongly concave minimax problem with coupled linear constraints.
	\end{remark}

	\begin{remark}
		If $A=B=c=0$, problem \eqref{problem:s} degenerates into problem \eqref{min-s}, Algorithm \ref{zoalg:2} can be simplified as Algorithm \ref{zoalg:3}. Theorem \ref{azoscthm1} implies that the number of iterations for Algorithm \ref{zoalg:3} to obtain an $\varepsilon$-stationary point of problem  \eqref{min-s} is bounded by $\tilde{\mathcal{O}}\left(\kappa^{4.5}\varepsilon ^{-3} \right)$ and the total number of function value queries are bounded by $\tilde{\mathcal{O}}\left((d_x+d_y)\kappa^{4.5}\varepsilon ^{-3}  \right)$ for solving stochastic nonconvex-strongly concave minimax problem.
	\end{remark}
	
	\begin{algorithm}[t]
		\caption{(AZOM-PG)}
		\label{zoalg:3}
		\begin{algorithmic}
			\STATE{\textbf{Step 1}: Input $x_1,y_1,\tilde{\alpha}_1,\tilde{\beta}$, $0<\eta_1\leq 1, b$; $\varrho_0=1$, $\iota_0=1$; Set $k=1$.}
			\STATE{\textbf{Step 2}: Draw a mini-batch sample $I_{k+1}=\{\zeta_i^{k+1}\}_{i=1}^b$. Compute
				\begin{align*}
					v_{k}&=\widehat{\nabla} _xG_{k}( x_{k},y_{k};I_{k})+(1-\varrho_{k-1})[v_{k-1}-\widehat{\nabla} _xG_{k-1}( x_{k-1},y_{k-1};I_{k})],\\
					w_{k}&=\widehat{\nabla} _yG_{k}( x_{k},y_{k};I_{k})+(1-\iota_{k-1})[w_{k-1}-\widehat{\nabla} _yG_{k-1}( x_{k-1},y_{k-1};I_{k})].
			\end{align*}}
			\STATE{\textbf{Step 3}: Perform the following update for $x_k$, $y_k$ and $\lambda_k$:  	
				\begin{align}
					\tilde{x}_{k+1}&={\color{black}\mathcal{P}_{\mathcal{X}}} \left( x_k - \tilde{\alpha}_kv_k\right),\quad x_{k+1}=x_k+\eta_k(\tilde{x}_{k+1}-x_k).\label{zo3:update-x}\\
					\tilde{y}_{k+1}&={\color{black}\mathcal{P}_{\mathcal{Y}}} \left( y_k + \tilde{\beta}w_k\right),\quad y_{k+1}=y_k+\eta_k(\tilde{y}_{k+1}-y_k).\label{zo3:update-y}
			\end{align}}
			\STATE{\textbf{Step 4}: If some stationary condition is satisfied, stop; otherwise, set $k=k+1, $ go to Step 2.}
		\end{algorithmic}
	\end{algorithm}

	\subsection{Complextiy Analysis: Nonconvex-Concave Setting}
	In this subsection, we prove the iteration complexity of Algorithm \ref{zoalg:2} under the nonconvex-concave setting. $\forall k\ge 1$, we first denote
	\begin{align}
		\Psi_k(x,\lambda)&=\max_{y\in\mathcal{Y}}\tilde{\mathcal{L}}_{g,k}(x,y,\lambda),\label{azoc:2}\\
		y_k^*(x,\lambda)&=\arg\max_{y\in\mathcal{Y}}\tilde{\mathcal{L}}_{g,k}(x,y,\lambda).\label{azoc:3}
	\end{align}
	We also need to make the following assumption on the parameters $\rho_k$, $\theta_{1,k}$ and $\theta_{2,k}$.
	\begin{assumption}\label{azocrho}
		$\{\rho_k\}$ is a nonnegative monotonically decreasing sequence.
	\end{assumption}
	
	\begin{assumption}\label{azocthetak}
		$\{\theta_{1,k}\}$ and $\{\theta_{2,k}\}$ are nonnegative monotonically decreasing sequences.
	\end{assumption}
	
	If Assumption \ref{azoass:Lip} holds, $g(x,y)$ has Lipschitz continuous gradients with constant $l$ by Lemma 7 in \cite{xu22zeroth}. Then, by the definition of $\tilde{\mathcal{L}}_{g,k}(x,y,\lambda)$ and Assumption \ref{azocrho}, we know that $\tilde{\mathcal{L}}_{g,k}(x,y,\lambda)$ has Lipschitz continuous gradients with constant $L$, where $L=\max\{l+\rho_1,\|A\|,\|B\|\}$.
	
	\begin{lemma}\label{azoclem1}
		Suppose that Assumptions \ref{azoass:Lip} and  \ref{azocrho} hold.
		Then for any $x, \bar{x}\in\mathcal{X}$, $\lambda, \tilde{\lambda}\in\Lambda$,
		\begin{align}\label{azoclem1:1}
			\|y_{k+1}^*(\bar{x},\bar{\lambda})-y_{k}^*(x,\lambda)\|^2
			\le\frac{2L^2}{\rho_{k+1}^2}\|\bar{x}-x\|^2+\frac{2L^2}{\rho_{k+1}^2}\|\bar{\lambda}-\lambda\|^2+\frac{\rho_k-\rho_{k+1}}{\rho_{k+1}}(\|y_{k+1}^*(\bar{x},\bar{\lambda})\|^2-\|y_{k}^*(x,\lambda)\|^2).
		\end{align}
	\end{lemma}
	
	\begin{proof}
		The optimality condition for $y_k^*(x,\lambda)$ in \eqref{azoc:3} implies that $\forall y\in \mathcal{Y}$ and $\forall k\geq 1$,
		\begin{align}
			\langle\nabla_y\tilde{\mathcal{L}}_{g,k+1}(\bar{x},y_{k+1}^*(\bar{x},\bar{\lambda}),\bar{\lambda}), y-y_{k+1}^*(\bar{x},\bar{\lambda})\rangle&\le0,\label{azoclem1:2}\\
			\langle\nabla_y\tilde{\mathcal{L}}_{g,k}(x,y_{k}^*(x,\lambda),\lambda), y-y_{k}^*(x,\lambda)\rangle&\le0.\label{azoclem1:3}
		\end{align}
		Setting $y=y_{k}^*(x,\lambda)$ in \eqref{azoclem1:2} and $y=y_{k+1}^*(\bar{x},\bar{\lambda})$ in \eqref{azoclem1:3}, adding these two inequalities and using the strong concavity of $\tilde{\mathcal{L}}_{g,k}(x,y,\lambda)$ with respect to $y$, we have
		\begin{align}
			&\langle\nabla_y\tilde{\mathcal{L}}_{g,k+1}(\bar{x},y_{k+1}^*(\bar{x},\bar{\lambda}),\bar{\lambda})-\nabla_y\tilde{\mathcal{L}}_{g,k}(x,y_{k+1}^*(\bar{x},\bar{\lambda}),\lambda), y_{k}^*(x,\lambda)-y_{k+1}^*(\bar{x},\bar{\lambda})\rangle\nonumber\\
			\le&\langle\nabla_y\tilde{\mathcal{L}}_{g,k}(x,y_{k+1}^*(\bar{x},\bar{\lambda}),\lambda)-\nabla_y\tilde{\mathcal{L}}_{g,k}(x,y_{k}^*(x,\lambda),\lambda), y_{k+1}^*(\bar{x},\bar{\lambda})-y_{k}^*(x,\lambda)\rangle\nonumber\\
			\le&-\rho_k\|y_{k+1}^*(\bar{x},\bar{\lambda})-y_{k}^*(x,\lambda)\|^2.\label{azoclem1:4}
		\end{align}
		By the definition of $\tilde{\mathcal{L}}_{g,k}(x,y,\lambda)$, the Cauchy-Schwarz inequality and Assumption \ref{azoass:Lip}, \eqref{azoclem1:4} implies that
		\begin{align}
			&(\rho_k-\rho_{k+1})\langle y_{k+1}^*(\bar{x},\bar{\lambda}),y_{k}^*(x,\lambda)-y_{k+1}^*(\bar{x},\bar{\lambda}) \rangle\nonumber\\
			\le&\langle\nabla_y\mathcal{L}_g(\bar{x},y_{k+1}^*(\bar{x},\bar{\lambda}),\bar{\lambda})-\nabla_y\mathcal{L}_g(x,y_{k+1}^*(\bar{x},\bar{\lambda}),\bar{\lambda}), y_{k+1}^*(\bar{x},\bar{\lambda})-y_{k}^*(x,\lambda)\rangle\nonumber\\
			&+\langle\nabla_y\mathcal{L}_g(x,y_{k+1}^*(\bar{x},\bar{\lambda}),\bar{\lambda})-\nabla_y\mathcal{L}_g(x,y_{k+1}^*(\bar{x},\bar{\lambda}),\lambda), y_{k+1}^*(\bar{x},\bar{\lambda})-y_{k}^*(x,\lambda)\rangle-\rho_k\|y_{k+1}^*(\bar{x},\bar{\lambda})-y_{k}^*(x,\lambda)\|^2\nonumber\\
			\le&\frac{L^2}{\rho_k}\|\bar{x}-x\|^2+\frac{L^2}{\rho_k}\|\bar{\lambda}-\lambda\|^2-\frac{\rho_k}{2}\|y_{k+1}^*(\bar{x},\bar{\lambda})-y_{k}^*(x,\lambda)\|^2.\label{azoclem1:5}
		\end{align}
		Since $\langle y_{k+1}^*(\bar{x},\bar{\lambda}),y_{k}^*(x,\lambda)-y_{k+1}^*(\bar{x},\bar{\lambda}) \rangle=\frac{1}{2}(\|y_{k}^*(x,\lambda)\|^2-\|y_{k+1}^*(\bar{x},\bar{\lambda})\|^2-\|y_{k+1}^*(\bar{x},\bar{\lambda})-y_{k}^*(x,\lambda)\|^2)$ and $\rho_{k+1}\le\rho_k$, \eqref{azoclem1:5} implies that
		\begin{align*}
			\|y_{k+1}^*(\bar{x},\bar{\lambda})-y_{k}^*(x,\lambda)\|^2
			\le\frac{2L^2}{\rho_{k+1}^2}\|\bar{x}-x\|^2+\frac{2L^2}{\rho_{k+1}^2}\|\bar{\lambda}-\lambda\|^2+\frac{\rho_k-\rho_{k+1}}{\rho_{k+1}}(\|y_{k+1}^*(\bar{x},\bar{\lambda})\|^2-\|y_{k}^*(x,\lambda)\|^2).
		\end{align*}
		The proof is then completed.
	\end{proof}

	\begin{lemma}\label{azoclem2}
		Suppose that Assumptions \ref{azoass:Lip} and \ref{azocrho} hold. Let $\{\left(x_k,y_k,\lambda_k\right)\}$ be a sequence generated by Algorithm \ref{zoalg:2},
		then $\forall k \ge 1$, 
		\begin{align}
			&\Psi_{k+1}(x_{k+1},\lambda_{k+1})-\Psi_k(x_{k},\lambda_k)\nonumber\\
			\le&4\eta_k\tilde{\alpha}_k L^2\|y_k-y_k^*(x_k,\lambda_k)\|^2+2\eta_k\tilde{\alpha}_k\|\widehat{\nabla} _x\tilde{\mathcal{L}}_{g,k}(x_{k},y_{k},\lambda_k)-v_k\|^2\nonumber\\
			&-(\frac{3\eta_k}{4\tilde{\alpha}_k}-\frac{L^2\eta_k^2}{\rho_k})\|\tilde{x}_{k+1}-x_k\|^2+4\eta_k\tilde{\alpha}_k L^2\|y_{k+1}-y_{k+1}^*(x_{k+1},\lambda_{k+1})\|^2\nonumber\\
			&-(\frac{\eta_k}{\tilde{\gamma}_k}-\frac{\|B\|^2\eta_k}{16\tilde{\alpha}_kL^2}-\frac{\eta_k^2L^2}{\rho_k})\|\tilde{\lambda}_{k+1}-\lambda_k\|^2+\eta_k\tilde{\alpha}_k d_xL^2\theta_{1,k}^2+\frac{\rho_k-\rho_{k+1}}{2}\sigma_y^2,\label{azoclem2:1}
		\end{align}
		where $\sigma_y=\max\{\|y\| \mid y\in\mathcal{Y}\}$.
	\end{lemma}
	
	\begin{proof}
		Similar to the proof of \eqref{azosclem1:6} in Lemma \ref{azosclem1}, by replacing $\tilde{\alpha}$ with $\tilde{\alpha}_k$, $\tilde{\gamma}$ with $\tilde{\gamma}_k$, $\mu$ with $\rho_k$ and $y^*(x_k,\lambda_k)$ with $y_k^*(x_k,\lambda_k)$, respectively, we have
		\begin{align}
			\Psi_k(x_{k+1},\lambda_{k})-\Psi_k(x_{k},\lambda_k)
			\le&4\eta_k\tilde{\alpha}_k L^2\|y_k-y_k^*(x_k,\lambda_k)\|^2+2\eta_k\tilde{\alpha}_k\|\widehat{\nabla} _x\tilde{\mathcal{L}}_{g,k}(x_{k},y_{k},\lambda_k)-v_k\|^2\nonumber\\
			&-(\frac{3\eta_k}{4\tilde{\alpha}_k}-\frac{L^2\eta_k^2}{\rho_k})\|\tilde{x}_{k+1}-x_k\|^2+\eta_k\tilde{\alpha}_k d_xL^2\theta_{1,k}^2.\label{azoclem2:2}
		\end{align}
		On the other hand, similar to the proof of \eqref{azosclem1:7}, \eqref{azosclem1:8} and \eqref{azosclem1:9}, by replacing $\tilde{\alpha}$ with $\tilde{\alpha}_k$, $\tilde{\gamma}$ with $\tilde{\gamma}_k$, $\mu$ with $\rho_k$ and $y^*(x_{k+1},\lambda_{k+1})$ with $y_{k+1}^*(x_{k+1},\lambda_{k+1})$, respectively, we have
		\begin{align} \label{azoclem2:3}
			&\Psi_{k+1}(x_{k+1},\lambda_{k+1})-\Psi_{k+1}(x_{k+1},\lambda_k)\nonumber\\
			\le&4\eta_k\tilde{\alpha}_k L^2\|y_{k+1}-y_{k+1}^*(x_{k+1},\lambda_{k+1})\|^2-(\frac{\eta_k}{\tilde{\gamma}_k}-\frac{\|B\|^2\eta_k}{16\tilde{\alpha}_kL^2}-\frac{\eta_k^2L^2}{\rho_k})\|\tilde{\lambda}_{k+1}-\lambda_k\|^2.
		\end{align}
		By the definition of $\Psi_k(x,\lambda)$ and $y_k^*(x,\lambda)$ and Assumption \ref{azocrho}, we get
		\begin{align}
			\Psi_{k+1}(x_{k+1},\lambda_{k})-\Psi_{k}(x_{k+1},\lambda_{k})
			=&\Psi_{k+1}(x_{k+1},\lambda_{k})-\tilde{\mathcal{L}}_{g,k}(x_{k+1},y_{k}^*(x_{k+1},\lambda_{k}),\lambda_{k})\nonumber\\
			\le&\Psi_{k+1}(x_{k+1},\lambda_{k})-\tilde{\mathcal{L}}_{g,k}(x_{k+1},y_{k+1}^*(x_{k+1},\lambda_{k}),\lambda_{k})\nonumber\\
			=&\frac{\rho_k-\rho_{k+1}}{2}\|y_{k+1}^*(x_{k+1},\lambda_{k})\|^2\le\frac{\rho_k-\rho_{k+1}}{2}\sigma_y^2.\label{azoclem2:4}
		\end{align}
		Combining \eqref{azoclem2:2},  \eqref{azoclem2:3} and  \eqref{azoclem2:4}, we complete the proof.
	\end{proof}

	\begin{lemma}\label{azoclem3}
		Suppose that Assumptions \ref{azoass:Lip} and \ref{azocrho} hold. Let $\{\left(x_k,y_k,\lambda_k\right)\}$ be a sequence generated by Algorithm \ref{zoalg:2}, if $0<\eta_k\le1$, $\rho_k\le L$ and  $0<\tilde{\beta}\le\frac{1}{6L}$
		then $\forall k \ge 1$, 
		\begin{align}
			&\|y_{k+1}-y_{k+1}^*(x_{k+1},\lambda_{k+1})\|^2\nonumber\\
			\le&(1-\frac{\eta_k\tilde{\beta}\rho_{k+1}}{4})\|y_{k}-y_k^*(x_k,\lambda_k)\|^2-\frac{3\eta_k}{4}\|\tilde{y}_{k+1} -y_k\|^2+\frac{10L^2\eta_k}{\rho_{k+1}^3\tilde{\beta}}\|\tilde{x}_{k+1}-x_k\|^2\nonumber\\
			&+\frac{10\eta_k\tilde{\beta}}{\rho_{k+1}}\|\widehat{\nabla} _{y}\tilde{\mathcal{L}}_{g,k}(x_{k},y_{k},\lambda_k)-w_k\|^2+\frac{5\eta_k\tilde{\beta}d_yL^2\theta_{2,k}^2}{2\rho_{k+1}}+\frac{10L^2\eta_k}{\rho_{k+1}^3\tilde{\beta}}\|\tilde{\lambda}_{k+1}-\lambda_k\|^2\nonumber\\
			&+\frac{5(\rho_k-\rho_{k+1})}{\eta_k\tilde{\beta}\rho_{k+1}^2}(\|y_{k+1}^*(x_{k+1},\lambda_{k+1})\|^2-\|y_{k}^*(x_k,\lambda_k)\|^2).\label{azoclem3:1}
		\end{align}
	\end{lemma}
	
	\begin{proof}
		Similar to the proof of \eqref{azosclem2:10} in Lemma \ref{azosclem2}, by replacing $\mu$ with $\rho_k$ and $y^*(x_k,\lambda_k)$ with $y_k^*(x_k,\lambda_k)$, respectively, we have
		\begin{align}
			&\|y_{k+1}-y_k^*(x_k,\lambda_k)\|^2\nonumber\\
			\le&(1-\frac{\eta_k\tilde{\beta}\rho_k}{2})\|y_{k}-y_k^*(x_k,\lambda_k)\|^2-\frac{3\eta_k}{4}\|\tilde{y}_{k+1} -y_k\|^2+\frac{8\eta_k\tilde{\beta}}{\rho_k}\|\widehat{\nabla} _{y}\tilde{\mathcal{L}}_{g,k}(x_{k},y_{k},\lambda_k)-w_k\|^2\nonumber\\
			&+\frac{2\eta_k\tilde{\beta}d_yL^2\theta_{2,k}^2}{\rho_k}.\label{azoclem3:2}
		\end{align}
		By the Cauchy-Schwarz inequality, Lemma \ref{azoclem1}, \eqref{zo2:update-x} and \eqref{zo2:update-lambda}, we have
		\begin{align}
			&\|y_{k+1}-y_{k+1}^*(x_{k+1},\lambda_{k+1})\|^2\nonumber\\
			=&\|y_{k+1}-y_{k}^*(x_{k},\lambda_{k})\|^2+2\langle y_{k+1}-y_k^*(x_{k},\lambda_{k}),y_k^*(x_{k},\lambda_{k})-y_{k+1}^*(x_{k+1},\lambda_{k+1}) \rangle\nonumber\\
			&+\|y_k^*(x_{k},\lambda_{k})-y_{k+1}^*(x_{k+1},\lambda_{k+1})\|^2\nonumber\\
			\le&(1+\frac{\eta_k\tilde{\beta}\rho_k}{4})\|y_{k+1}-y_k^*(x_{k},\lambda_{k})\|^2+(1+\frac{4}{\eta_k\tilde{\beta}\rho_k})\|y_k^*(x_{k},\lambda_{k})-y_{k+1}^*(x_{k+1},\lambda_{k+1})\|^2\nonumber\\
			\le&(1+\frac{\eta_k\tilde{\beta}\rho_k}{4})\|y_{k+1}-y_k^*(x_{k},\lambda_{k})\|^2+(2+\frac{8}{\eta_k\tilde{\beta}\rho_k})\frac{L^2\eta_k^2}{\rho_{k+1}^2}\|\tilde{x}_{k+1}-x_k\|^2\nonumber\\
			&+(1+\frac{4}{\eta_k\tilde{\beta}\rho_k})\frac{\rho_k-\rho_{k+1}}{\rho_{k+1}}(\|y_{k+1}^*(x_{k+1},\lambda_{k+1})\|^2-\|y_{k}^*(x_k,\lambda_k)\|^2)+(2+\frac{8}{\eta_k\tilde{\beta}\rho_k})\frac{L^2\eta_k^2}{\rho_{k+1}^2}\|\tilde{\lambda}_{k+1}-\lambda_k\|^2.\label{azoclem3:3}
		\end{align}
		Similar to the proof of \eqref{azosclem2:13}, by combining \eqref{azoclem3:2} and \eqref{azoclem3:3} and using Assumption \ref{azocrho}, we complete the proof.
	\end{proof}

	\begin{lemma}\label{azoclem4}
		Suppose that Assumptions \ref{azoass:Lip}, \ref{azoass:var}, \ref{azocrho} and \ref{azoscthetak} hold. Let $\{\left(x_k,y_k,\lambda_k\right)\}$ be a sequence generated by Algorithm \ref{zoalg:2}. Denote
		\begin{align*}
			F_{k+1}(x_{k+1},y_{k+1},\lambda_{k+1})
			=&\mathbb{E}[\Psi_{k+1}(x_{k+1},\lambda_{k+1})+\frac{32\tilde{\alpha}_kL^2}{\tilde{\beta}\rho_{k+1}}\|y_{k+1}-y_{k+1}^*(x_{k+1},\lambda_{k+1})\|^2\\
			&+D_k^{(1)}\|\widehat{\nabla} _{x}\tilde{\mathcal{L}}_{g,k+1}(x_{k+1},y_{k+1},\lambda_{k+1})-v_{k+1}\|^2\\
			&+D_k^{(2)}\|\widehat{\nabla} _{y}\tilde{\mathcal{L}}_{g,k+1}(x_{k+1},y_{k+1},\lambda_{k+1})-w_{k+1}\|^2],\\
			S_{k+1}(x_{k+1},y_{k+1},\lambda_{k+1})
			=&F_{k+1}(x_{k+1},y_{k+1},\lambda_{k+1})-4\eta_k\tilde{\alpha}_k L^2\mathbb{E}\|y_{k+1}-y_{k+1}^*(x_{k+1},\lambda_{k+1})\|^2\\
			&-\frac{160\tilde{\alpha}_{k+1}L^2(\rho_{k+1}-\rho_{k+2})}{\eta_{k+1}\tilde{\beta}^2\rho_{k+2}^3}\mathbb{E}\|y_{k+1}^*(x_{k+1},\lambda_{k+1})\|^2+\frac{12D_{k+1}^{(2)}\rho_{k+1}^2\sigma_y^2}{b},
		\end{align*}
		where $D_k^{(1)}>0$ and $D_k^{(2)}>0$. If $0<\eta_k\le1$, $\rho_k\le L$ and  $0<\tilde{\beta}\le\frac{1}{6L}$, then $\forall k \ge 2$, 
		\begin{align}\label{azoclem4:1}
			&S_{k+1}(x_{k+1},y_{k+1},\lambda_{k+1})-S_{k}(x_{k},y_{k},\lambda_{k})\nonumber\\
			\le&(-4\eta_k\tilde{\alpha}_kL^2+4\eta_{k-1}\tilde{\alpha}_{k-1}L^2 +\frac{32\tilde{\alpha}_kL^2}{\tilde{\beta}\rho_{k+1}}-\frac{32\tilde{\alpha}_{k-1}L^2}{\tilde{\beta}\rho_{k}})\mathbb{E}\|y_k-y_k^*(x_k,\lambda_k)\|^2\nonumber\\
			&+(2\eta_k\tilde{\alpha}_k-D_k^{(1)}\varrho_{k}+D_k^{(1)}-D_{k-1}^{(1)})\mathbb{E}\|\widehat{\nabla} _{x}\tilde{\mathcal{L}}_{g,k}(x_{k},y_{k},\lambda_k)-v_k\|^2\nonumber\\
			&+(\frac{320\eta_k\tilde{\alpha}_kL^2}{\rho_{k+1}^2}-D_k^{(2)}\iota_{k}+D_k^{(2)}-D_{k-1}^{(2)})\mathbb{E}\|\widehat{\nabla} _{y}\tilde{\mathcal{L}}_{g,k}(x_{k},y_{k},\lambda_k)-w_k\|^2\nonumber\\
			&-(\frac{3\eta_k}{4\tilde{\alpha}_k}-\frac{L^2\eta_k^2}{\rho_{k+1}}-\frac{320L^4\eta_k\tilde{\alpha}_k}{\rho_{k+1}^4\tilde{\beta}^2}-\frac{12L^2\eta_k^2(D_k^{(1)}+D_k^{(2)})}{b})\mathbb{E}\|\tilde{x}_{k+1}-x_k\|^2\nonumber\\
			&-(\frac{24\tilde{\alpha}_kL^2\eta_k}{\rho_{k+1}\tilde{\beta}}-\frac{12L^2\eta_k^2(D_k^{(1)}+D_k^{(2)})}{b})\mathbb{E}\|\tilde{y}_{k+1} -y_k\|^2+\eta_k\tilde{\alpha}_kd_xL^2\theta_{1,k}^2\nonumber\\
			&+(\frac{160\tilde{\alpha}_kL^2(\rho_k-\rho_{k+1})}{\eta_{k}\tilde{\beta}^2\rho_{k+1}^3}-\frac{160\tilde{\alpha}_{k+1}L^2(\rho_{k+1}-\rho_{k+2})}{\eta_{k+1}\tilde{\beta}^2\rho_{k+2}^3})\mathbb{E}\|y_{k+1}^*(x_{k+1},\lambda_{k+1})\|^2+\frac{\rho_k-\rho_{k+1}}{2}\sigma_y^2\nonumber\\
			&-(\frac{\eta_k}{\tilde{\gamma}_k}-\frac{\|B\|^2\eta_k}{16\tilde{\alpha}_kL^2}-\frac{\eta_k^2L^2}{\rho_{k+1}}-\frac{320L^4\eta_k\tilde{\alpha}_k}{\rho_{k+1}^4\tilde{\beta}^2}-\frac{12L^2\eta_k^2(D_k^{(1)}+D_k^{(2)})}{b})\mathbb{E}\|\tilde{\lambda}_{k+1}-\lambda_k\|^2+\frac{80\eta_k\tilde{\alpha}_kd_yL^4\theta_{2,k}^2}{\rho_{k+1}^2}\nonumber\\
			&+\frac{2\delta^2(\varrho_k^2D_k^{(1)}+\iota_k^2D_k^{(2)})}{b}+\frac{3L^2(d_x\theta_{1,k}^2D_k^{(1)}+d_y\theta_{2,k}^2D_k^{(2)})}{b}+\frac{12(D_{k+1}^{(2)}-D_{k}^{(2)})\rho_{k+1}^2\sigma_y^2}{b}.
		\end{align}
	\end{lemma}
	
	\begin{proof}
		Similar to the proof of Lemma \ref{azosclem3}, we have
		\begin{align}
			&\mathbb{E}\|\widehat{\nabla} _{x}\tilde{\mathcal{L}}_{g,k+1}(x_{k+1},y_{k+1},\lambda_{k+1})-v_{k+1}\|^2\nonumber\\
			\le&(1-\varrho_{k})\mathbb{E}\|\widehat{\nabla} _{x}\tilde{\mathcal{L}}_{g,k}(x_{k},y_{k},\lambda_k)-v_k\|^2+\frac{9L^2\eta_k^2}{b}\mathbb{E}[\|\tilde{x}_{k+1}-x_k\|^2+\|\tilde{y}_{k+1}-y_k\|^2+\|\tilde{\lambda}_{k+1}-\lambda_k\|^2]\nonumber\\
			&+\frac{2\varrho_k^2\delta^2}{b}+\frac{3d_xL^2\theta_{1,k}^2}{b},\label{azoclem4:2x}\\
			&\mathbb{E}\|\widehat{\nabla} _{y}\tilde{\mathcal{L}}_{g,k+1}(x_{k+1},y_{k+1},\lambda_{k+1})-w_{k+1}\|^2\nonumber\\
			\le&(1-\iota_{k})\mathbb{E}\|\widehat{\nabla} _{y}\tilde{\mathcal{L}}_{g,k}(x_{k},y_{k},\lambda_k)-w_k\|^2+\frac{12L^2\eta_k^2}{b}\mathbb{E}[\|\tilde{x}_{k+1}-x_k\|^2+\|\tilde{y}_{k+1}-y_k\|^2+\|\tilde{\lambda}_{k+1}-\lambda_k\|^2]\nonumber\\
			&+\frac{2\iota_k^2\delta^2}{b}+\frac{3d_yL^2\theta_{2,k}^2}{b}+\frac{12(\rho_k^2-\rho_{k+1}^2)\sigma_y^2}{b}.\label{azoclem4:2y}
		\end{align}
		Combining \eqref{azoclem2:1}, \eqref{azoclem3:1}, \eqref{azoclem4:2x}, \eqref{azoclem4:2y} and using Assumption \ref{azocrho}, we obtain
		\begin{align}\label{azoclem4:3}
			&F_{k+1}(x_{k+1},y_{k+1},\lambda_{k+1})-F_{k}(x_{k},y_{k},\lambda_{k})\nonumber\\
			\le&4\eta_k\tilde{\alpha}_k L^2\mathbb{E}\|y_{k+1}-y_{k+1}^*(x_{k+1},\lambda_{k+1})\|^2\nonumber\\
			&+(-4\eta_k\tilde{\alpha}_kL^2+\frac{32\tilde{\alpha}_kL^2}{\tilde{\beta}\rho_{k+1}}-\frac{32\tilde{\alpha}_{k-1}L^2}{\tilde{\beta}\rho_{k}})\mathbb{E}\|y_k-y_k^*(x_k,\lambda_k)\|^2\nonumber\\
			&+(2\eta_k\tilde{\alpha}_k-D_k^{(1)}\varrho_{k}+D_k^{(1)}-D_{k-1}^{(1)})\mathbb{E}\|\widehat{\nabla} _{x}\tilde{\mathcal{L}}_{g,k}(x_{k},y_{k},\lambda_k)-v_k\|^2\nonumber\\
			&+(\frac{320\eta_k\tilde{\alpha}_kL^2}{\rho_{k+1}^2}-D_k^{(2)}\iota_{k}+D_k^{(2)}-D_{k-1}^{(2)})\mathbb{E}\|\widehat{\nabla} _{y}\tilde{\mathcal{L}}_{g,k}(x_{k},y_{k},\lambda_k)-w_k\|^2\nonumber\\
			&-(\frac{3\eta_k}{4\tilde{\alpha}_k}-\frac{L^2\eta_k^2}{\rho_{k+1}}-\frac{320L^4\eta_k\tilde{\alpha}_k}{\rho_{k+1}^4\tilde{\beta}^2}-\frac{12L^2\eta_k^2(D_k^{(1)}+D_k^{(2)})}{b})\mathbb{E}\|\tilde{x}_{k+1}-x_k\|^2\nonumber\\
			&-(\frac{24\tilde{\alpha}_kL^2\eta_k}{\rho_{k+1}\tilde{\beta}}-\frac{12L^2\eta_k^2(D_k^{(1)}+D_k^{(2)})}{b})\mathbb{E}\|\tilde{y}_{k+1} -y_k\|^2+\eta_k\tilde{\alpha}_kd_xL^2\theta_{1,k}^2\nonumber\\
			&+\frac{160\tilde{\alpha}_kL^2(\rho_k-\rho_{k+1})}{\eta_{k}\tilde{\beta}^2\rho_{k+1}^3}(\mathbb{E}\|y_{k+1}^*(x_{k+1},\lambda_{k+1})\|^2-\mathbb{E}\|y_{k}^*(x_k,\lambda_k)\|^2)+\frac{\rho_k-\rho_{k+1}}{2}\sigma_y^2\nonumber\\
			&-(\frac{\eta_k}{\tilde{\gamma}_k}-\frac{\|B\|^2\eta_k}{16\tilde{\alpha}_kL^2}-\frac{\eta_k^2L^2}{\rho_{k+1}}-\frac{320L^4\eta_k\tilde{\alpha}_k}{\rho_{k+1}^4\tilde{\beta}^2}-\frac{12L^2\eta_k^2(D_k^{(1)}+D_k^{(2)})}{b})\mathbb{E}\|\tilde{\lambda}_{k+1}-\lambda_k\|^2+\frac{80\eta_k\tilde{\alpha}_kd_yL^4\theta_{2,k}^2}{\rho_{k+1}^2}\nonumber\\
			&+\frac{2\delta^2(\varrho_k^2D_k^{(1)}+\iota_k^2D_k^{(2)})}{b}+\frac{3L^2(d_x\theta_{1,k}^2D_k^{(1)}+d_y\theta_{2,k}^2D_k^{(2)})}{b}+\frac{12D_k^{(2)}(\rho_k^2-\rho_{k+1}^2)\sigma_y^2}{b}.
		\end{align}
		The proof is completed by the definition of $S_{k+1}(x_{k+1},y_{k+1},\lambda_{k+1})$.
	\end{proof}
	
	Similar to the proof of Theorem \ref{azoscthm1}, we then obtain the following theorem which provides a bound on $T_2(\varepsilon)$, where $T_2(\varepsilon):=\min\{k \mid \| \nabla \mathcal{G}_k^{\tilde{\alpha}_k,\tilde{\beta} ,\tilde{\gamma}_k}(x_k,y_{k},\lambda_k) \| \leq \varepsilon, k\geq 2\}$ and $\varepsilon>0$ is a given target accuracy.
	
	\begin{theorem}\label{azocthm1}
		Suppose that Assumptions \ref{fea}, \ref{azoass:Lip}, \ref{azoass:var} and \ref{azoscthetak} hold. Let $\{\left(x_k,y_k,\lambda_k\right)\}$ be a sequence generated by Algorithm \ref{zoalg:2}. Set  $\eta_k=\frac{1}{(k+2)^{5/13}}$, $\tilde{\alpha}_k=\frac{{\color{black}A_2}}{(k+2)^{4/13}}$, $\tilde{\gamma}_k=\frac{{\color{black}A_3}}{(k+2)^{4/13}}$, $\rho_{k+1}=\frac{L}{(k+2)^{2/13}}$, $\varrho_k=\frac{{\color{black}A_4}}{(k+2)^{\frac{12}{13}}}$,$\iota_k=\frac{{\color{black}A_5}}{(k+2)^{\frac{8}{13}}}$, $D_k^{(1)}={\color{black}A_6}(k+2)^{\frac{3}{13}}$, $D_k^{(2)}={\color{black}A_7}(k+2)^{\frac{3}{13}}$, $\theta_{1,k}=\frac{\varepsilon}{L(k+2)^{\frac{6}{13}}}\sqrt{\frac{{\color{black}D_1}}{d_x({\color{black}24D_1}+16+\frac{64{\color{black}A_6}}{b{\color{black}A_2}})}}$, $\theta_{2,k}=\frac{\varepsilon}{L(k+2)^{\frac{6}{13}}}\sqrt{\frac{{\color{black}D_1}}{d_y({\color{black}24D_1}+1280+\frac{64{\color{black}A_7}}{b{\color{black}A_2}})}}$ with  ${\color{black}A_2}\le\min\{\frac{1}{8L},\frac{\tilde{\beta}}{16\sqrt{10}},\frac{1}{\sqrt{3}\|A\|}\}$, ${\color{black}A_3}\le\min\{\frac{2L^2{\color{black}A_2}}{\|B\|^2},\frac{1}{4L},\frac{\tilde{\beta}^2}{1280{\color{black}A_2}}\}$, ${\color{black}A_4}\ge\frac{4{\color{black}A_2}}{{\color{black}A_6}}+\frac{12}{13}$, ${\color{black}A_5}\ge\frac{640{\color{black}A_2}}{{\color{black}A_7}}+\frac{12}{13}$, ${\color{black}A_6}\le\min\{\frac{b}{96{\color{black}A_2}L^2},\frac{11b{\color{black}A_2}}{12L\tilde{\beta}},\frac{b}{192{\color{black}A_3}L^2}\}$, ${\color{black}A_7}\le\min\{\frac{b}{96{\color{black}A_2}L^2},\frac{11b{\color{black}A_2}}{12L\tilde{\beta}},\frac{b}{192{\color{black}A_3}L^2}\}$. If $0<\tilde{\beta}\le\frac{1}{6L}$, then for any given $\varepsilon>0$, ${\color{black}T_2}(\varepsilon)$ satisfies that  
		\begin{align}
			\frac{{\color{black}7}\varepsilon^2}{8}\le\frac{C_1+C_2\ln({\color{black}T_2}(\varepsilon)+2)}{{\color{black}D_1A_2}(\frac{13}{4}({\color{black}T_2}(\varepsilon)+3)^{4/13}-\frac{13}{4}4^{4/13})},\label{azocthm1:1}
		\end{align}
		${\color{black}D_1}\le\min\{\frac{1}{{\color{black}32}},\frac{{\color{black}L}}{\frac{3}{\tilde{\beta}}+3\|B\|^2\tilde{\beta}},\frac{{\color{black}A_3}}{{\color{black}24A_2}},\frac{{\color{black}A_4A_6}}{{\color{black}12A_2}},\frac{{\color{black}A_5A_7}}{{\color{black}12A_2}}\}$, $C_1=S_{2}(x_{2},y_{2},\lambda_2)-\underbar{S}+16 L^2\eta_1\tilde{\alpha}_1\sigma_y^2+\frac{\rho_2}{2}\sigma_y^2+\frac{160\tilde{\alpha}_2L^2\rho_2}{\eta_{2}\tilde{\beta}^2\rho_3^3}\sigma_y^2$, $C_2=(\frac{2\delta^2({\color{black}A_4^2A_5+A_5^2A_7})}{b}+\frac{36{\color{black}A_7}d_yL^2\sigma_y^2}{13b}+{\color{black}2D_1A_2L^2\sigma_y^2})$ with $\underbar{S}=\min\limits_{\lambda\in\Lambda}\min\limits_{x\in\mathcal{X}}\min\limits_{y\in\mathcal{Y}}S_k(x,y,\lambda)$ and $\sigma_y=\max\{\|y\| \mid y\in\mathcal{Y}\}$.
	\end{theorem}
	
	\begin{proof}
		By the definition of $\tilde{\alpha}_k$, $\rho_k$ and $D_k^{(1)}$, we  have $\frac{\tilde{\alpha}_k}{\rho_{k+1}}\le\frac{\tilde{\alpha}_{k-1}}{\rho_{k}}$ and
		\begin{align*}
			D_k^{(1)}-D_{k-1}^{(1)}&{\color{black}= A_6 \!\int_{k+1}^{k+2} \frac{3}{13} x^{-\frac{10}{13}}\,dx 
				\le A_6\!\int_{k+1}^{k+2} \frac{3}{13}(k+1)^{-\frac{10}{13}}\,dx}\\
			& =\frac{3{\color{black}A_6}(k+1)^{-10/13}}{13}\le\frac{3{\color{black}A_6}\cdot2^{10/13}(2+k)^{-10/13}}{13}\le\frac{6{\color{black}A_6}(2+k)^{-9/13}}{13},
		\end{align*}
		{\color{black}where the first inequality is due to the monotonically decreasing property of $x^{-10/13}$ on $(0, +\infty)$.} Then, by the setting of ${\color{black}A_4}$, we get
		\begin{align*}
			2\eta_k\tilde{\alpha}_k-D_k^{(1)}\varrho_{k}+D_k^{(1)}-D_{k-1}^{(1)}&\le(2{\color{black}A_2}-{\color{black}A_4A_6}+\frac{6{\color{black}A_6}}{13})(2+k)^{-9/13}\le-\frac{{\color{black}A_4A_6}}{2}(2+k)^{-9/13}.
		\end{align*}
		Similarly, we obtain
		\begin{align*}
			\frac{320\eta_k\tilde{\alpha}_kL^2}{\rho_{k+1}^2}-D_k^{(2)}\iota_{k}+D_k^{(2)}-D_{k-1}^{(2)}&\le(320{\color{black}A_2}-{\color{black}A_5A_7}+\frac{6{\color{black}A_7}}{13})(2+k)^{-5/13}\le-\frac{{\color{black}A_5A_7}}{2}(2+k)^{-5/13}.
		\end{align*}
		By the settings of ${\color{black}A_2, A_3, A_6, A_7}$, we have
		\begin{align*}
			&-(\frac{3\eta_k}{4\tilde{\alpha}_k}-\frac{L^2\eta_k^2}{\rho_{k+1}}-\frac{320L^4\eta_k\tilde{\alpha}_k}{\rho_{k+1}^4\tilde{\beta}^2}-\frac{12L^2\eta_k^2(D_k^{(1)}+D_k^{(2)})}{b})\\
			\le&(-\frac{3}{4{\color{black}A_2}}+L+\frac{320{\color{black}A_2}}{\tilde{\beta}^2}+\frac{12L^2({\color{black}A_6+A_7})}{b})(2+k)^{-1/13}
			\le-\frac{1}{4{\color{black}A_2}}(2+k)^{-1/13},\\
			&-(\frac{24\tilde{\alpha}_kL^2\eta_k}{\rho_{k+1}\tilde{\beta}}-\frac{12L^2\eta_k^2(D_k^{(1)}+D_k^{(2)})}{b})\\
			\le&(-\frac{24L{\color{black}A_2}}{\tilde{\beta}}+\frac{12L^2({\color{black}A_6+A_7})}{b})(2+k)^{-7/13}
			\le-\frac{2L{\color{black}A_2}}{\tilde{\beta}}(2+k)^{-7/13},
		\end{align*}
		and
		\begin{align*}
			&-(\frac{\eta_k}{\tilde{\gamma}_k}-\frac{\|B\|^2\eta_k}{16\tilde{\alpha}_kL^2}-\frac{\eta_k^2L^2}{\rho_k}-\frac{320L^4\eta_k\tilde{\alpha}_k}{\rho_k^4\tilde{\beta}^2}-\frac{12L^2\eta_k^2(D_k^{(1)}+D_k^{(2)})}{b})\\
			\le&(-\frac{1}{{\color{black}A_3}}+\frac{\|B\|^2}{16L^2{\color{black}A_2}}+L+\frac{320{\color{black}A_2}}{\tilde{\beta}^2}+\frac{12L^2({\color{black}A_6+A_7})}{b})(2+k)^{-1/13}\le-\frac{1}{4{\color{black}A_3}}(2+k)^{-1/13}.
		\end{align*}
		Plugging all the above inequalities into \eqref{azoclem4:1}, we get
		\begin{align}\label{azocthm1:2}
			&S_{k+1}(x_{k+1},y_{k+1},\lambda_{k+1})-S_{k}(x_{k},y_{k},\lambda_{k})\nonumber\\
			\le&16L^2(\eta_{k-1}\tilde{\alpha}_{k-1}-\eta_k\tilde{\alpha}_k)\sigma_y^2-\frac{{\color{black}A_4A_6}}{2}(2+k)^{-9/13}\mathbb{E}\|\widehat{\nabla} _{x}\tilde{\mathcal{L}}_{g,k}(x_{k},y_{k},\lambda_k)-v_k\|^2\nonumber\\
			&-\frac{{\color{black}A_5A_7}}{2}(2+k)^{-5/13}\mathbb{E}\|\widehat{\nabla} _{y}\tilde{\mathcal{L}}_{g,k}(x_{k},y_{k},\lambda_k)-w_k\|^2+\eta_k\tilde{\alpha}_kd_xL^2\theta_{1,k}^2\nonumber\\
			&-\frac{1}{4{\color{black}A_2}}(2+k)^{-1/13}\mathbb{E}\|\tilde{x}_{k+1}-x_k\|^2-\frac{2L{\color{black}A_2}}{\tilde{\beta}}(2+k)^{-7/13}\mathbb{E}\|\tilde{y}_{k+1} -y_k\|^2\nonumber\\
			&+(\frac{160\tilde{\alpha}_kL^2(\rho_k-\rho_{k+1})}{\eta_{k}\tilde{\beta}^2\rho_{k+1}^3}-\frac{160\tilde{\alpha}_{k+1}L^2(\rho_{k+1}-\rho_{k+2})}{\eta_{k+1}\tilde{\beta}^2\rho_{k+2}^3})\sigma_y^2+\frac{\rho_k-\rho_{k+1}}{2}\sigma_y^2\nonumber\\
			&-\frac{1}{4{\color{black}A_3}}(2+k)^{-1/13}\mathbb{E}\|\tilde{\lambda}_{k+1}-\lambda_k\|^2+\frac{80\eta_k\tilde{\alpha}_kd_yL^4\theta_{2,k}^2}{\rho_{k+1}^2}\nonumber\\
			&+\frac{2\delta^2(\varrho_k^2D_k^{(1)}+\iota_k^2D_k^{(2)})}{b}+\frac{3L^2(d_x\theta_{1,k}^2D_k^{(1)}+d_y\theta_{2,k}^2D_k^{(2)})}{b}+\frac{36{\color{black}A_7}d_yL^2\sigma_y^2}{13b}(k+2)^{-1}.
		\end{align}
		{\color{black}Denote 
			\begin{equation*}
				\nabla \tilde{ \mathcal{G}}_k^{\tilde{\alpha}_k,\tilde{\beta} ,\tilde{\gamma}_k}\left( x,y,\lambda \right) :=\left( \begin{array}{c}
					\frac{1}{\tilde{\alpha}_k}\left( x-{\color{black}\mathcal{P}_{\mathcal{X}}}\left( x-\tilde{\alpha}_k\nabla _x\tilde{\mathcal{L}}_{g,k}\left( x,y,\lambda  \right) \right) \right)\\
					\frac{1}{\tilde{\beta}} \left( y-{\color{black}\mathcal{P}_{\mathcal{Y}} }\left( y+ \tilde{\beta}\nabla _y\tilde{\mathcal{L}}_{g,k}\left( x,y,\lambda \right) \right) \right)\\
					\frac{1}{\tilde{\gamma}_k}\left( \lambda-{\color{black}\mathcal{P}_{\Lambda}}\left(\lambda-\tilde{\gamma}_k\nabla_{\lambda}\tilde{\mathcal{L}}_{g,k}(x,y,\lambda)\right)\right)\\
				\end{array} \right) .
		\end{equation*}}
		Similar the proof of \eqref{azosclem5:1} and by $\eta_k\le1$, $3\|A\|^2\le\frac{1}{\tilde{\alpha}_k^2}$, we have
		\begin{align}\label{azocthm1:3}
			&\mathbb{E}\|\nabla \tilde{ \mathcal{G}}_k^{\tilde{\alpha}_k,\tilde{\beta} ,\tilde{\gamma}_k }(x_k,y_k,\lambda_k)\|^2\nonumber\\
			\le&\frac{4}{\tilde{\alpha}_k^2} \mathbb{E}\|\tilde{x}_{k+1}-x_k\|^2+(\frac{3}{\tilde{\beta}^2}+3\|B\|^2) \mathbb{E}\|\tilde{y}_{k+1}-y_k\|^2+\frac{3}{\tilde{\gamma}_k^2}\mathbb{E}\|\tilde{\lambda}_{k+1}-\lambda_k\|^2\nonumber\\
			&+3\mathbb{E}\|\widehat{\nabla}_x\tilde{\mathcal{L}}_{g,k}\left( x_k,y_{k},\lambda_k \right)-v_k\|^2+3\mathbb{E}\|\widehat{\nabla}_y\tilde{\mathcal{L}}_{g,k}\left( x_k,y_{k},\lambda_k \right)-w_k\|^2+\frac{3d_xL^2\theta_{1,k}^2}{4}+\frac{3d_yL^2\theta_{2,k}^2}{4}.
		\end{align}
		{\color{black}By the definition of $\nabla  \mathcal{G}_k^{\tilde{\alpha},\tilde{\beta} ,\tilde{\gamma}_k }(x_k,y_k,\lambda_k)$ and \eqref{azocthm1:3}, we have
			\begin{align}\label{azocthm1:3-1}
				&\mathbb{E}\|\nabla  \mathcal{G}^{\tilde{\alpha}_k,\tilde{\beta} ,\tilde{\gamma}_k }(x_k,y_k,\lambda_k)\|^2\nonumber\\
				=&\mathbb{E}\|\nabla  \mathcal{G}^{\tilde{\alpha}_k,\tilde{\beta} ,\tilde{\gamma}_k }(x_k,y_k,\lambda_k)-\nabla \tilde{ \mathcal{G}}_k^{\tilde{\alpha}_k,\tilde{\beta} ,\tilde{\gamma}_k }(x_k,y_k,\lambda_k)+\nabla \tilde{ \mathcal{G}}_k^{\tilde{\alpha}_k,\tilde{\beta} ,\tilde{\gamma}_k }(x_k,y_k,\lambda_k)\|^2\nonumber\\
				\le&2\mathbb{E}\|\nabla  \mathcal{G}^{\tilde{\alpha}_k,\tilde{\beta} ,\tilde{\gamma}_k }(x_k,y_k,\lambda_k)-\nabla \tilde{ \mathcal{G}}_k^{\tilde{\alpha}_k,\tilde{\beta} ,\tilde{\gamma}_k }(x_k,y_k,\lambda_k)\|^2+2\mathbb{E}\|\nabla \tilde{ \mathcal{G}}_k^{\tilde{\alpha}_k,\tilde{\beta} ,\tilde{\gamma}_k }(x_k,y_k,\lambda_k)\|^2\nonumber\\
				\le&\frac{8}{\tilde{\alpha}_k^2} \mathbb{E}\|\tilde{x}_{k+1}-x_k\|^2+(\frac{6}{\tilde{\beta}^2}+6\|B\|^2) \mathbb{E}\|\tilde{y}_{k+1}-y_k\|^2+\frac{6}{\tilde{\gamma}_k^2}\mathbb{E}\|\tilde{\lambda}_{k+1}-\lambda_k\|^2+2\rho_{k}^2\sigma_y^2\nonumber\\
				&+6\mathbb{E}\|\widehat{\nabla}_x\tilde{\mathcal{L}}_{g,k}\left( x_k,y_{k},\lambda_k \right)-v_k\|^2+6\mathbb{E}\|\widehat{\nabla}_y\tilde{\mathcal{L}}_{g,k}\left( x_k,y_{k},\lambda_k \right)-w_k\|^2+\frac{3d_xL^2\theta_{1,k}^2}{2}+\frac{3d_yL^2\theta_{2,k}^2}{2}.
		\end{align}}
		
		It follows from the definition of ${\color{black}D_1}$ and {\color{black}\eqref{azocthm1:3-1}},
		\begin{align}\label{azocthm1:4}
			&{\color{black}D_1}\eta_k\tilde{\alpha}_k\mathbb{E}\|\nabla {\color{black}\mathcal{G}}^{\tilde{\alpha}_k,\tilde{\beta} ,\tilde{\gamma}_k}(x_k,y_k,\lambda_k)\|^2\nonumber\\
			\le&S_{k}(x_{k},y_{k},\lambda_{k})-S_{k+1}(x_{k+1},y_{k+1},\lambda_{k+1})+\frac{3d_x{\color{black}D_1}L^2\theta_{1,k}^2\eta_k\tilde{\alpha}_k}{{\color{black}2}}+\frac{3d_y{\color{black}D_1}L^2\theta_{2,k}^2\eta_k\tilde{\alpha}_k}{{\color{black}2}}\nonumber\\
			&+\eta_k\tilde{\alpha}_k d_xL^2\theta_{1,k}^2+16L^2(\eta_{k-1}\tilde{\alpha}_{k-1}-\eta_k\tilde{\alpha}_k)\sigma_y^2+\frac{80\eta_k\tilde{\alpha}_kd_yL^4\theta_{2,k}^2}{\rho_{k+1}^2}\nonumber\\
			&+(\frac{160\tilde{\alpha}_kL^2(\rho_k-\rho_{k+1})}{\eta_{k}\tilde{\beta}^2\rho_{k+1}^3}-\frac{160\tilde{\alpha}_{k+1}L^2(\rho_{k+1}-\rho_{k+2})}{\eta_{k+1}\tilde{\beta}^2\rho_{k+2}^3})\sigma_y^2+\frac{\rho_k-\rho_{k+1}}{2}\sigma_y^2\nonumber\\
			&+\frac{2\delta^2(\varrho_k^2D_k^{(1)}+\iota_k^2D_k^{(2)})}{b}+\frac{3L^2(d_x\theta_{1,k}^2D_k^{(1)}+d_y\theta_{2,k}^2D_k^{(2)})}{b}+(\frac{36{\color{black}A_7}L^2\sigma_y^2}{13b}+{\color{black}2D_1A_2L^2\sigma_y^2})(k+2)^{-1}.
		\end{align}
		By summing both sides of \eqref{azocthm1:4} from $k=2$ to ${\color{black}T_2}(\varepsilon)$, we obtain
		\begin{align}
			&\sum_{k=2}^{{\color{black}T_2}(\varepsilon)}{\color{black}D_1}\eta_k\tilde{\alpha}_k\mathbb{E}\|\nabla {\color{black}\mathcal{G}}^{\tilde{\alpha}_k,\tilde{\beta} ,\tilde{\gamma}_k}(x_k,y_k,\lambda_k)\|^2\nonumber\\
			\le& S_{2}(x_{2},y_{2},\lambda_2)-S_{{\color{black}T_2}(\varepsilon)+1}(x_{{\color{black}T_2}(\varepsilon)+1},y_{{\color{black}T_2}(\varepsilon)+1},\lambda_{{\color{black}T_2}(\varepsilon)+1})+16 L^2\eta_1\tilde{\alpha}_1\sigma_y^2+\frac{\rho_2}{2}\sigma_y^2\nonumber\\
			&+\sum_{k=2}^{{\color{black}T_2}(\varepsilon)}\eta_k\tilde{\alpha}_k(\frac{3d_x{\color{black}D_1}L^2\theta_{1,k}^2}{{\color{black}2}}+d_xL^2\theta_{1,k}^2+\frac{3L^2d_xD_k^{(1)}\theta_{1,k}^2}{b\eta_k\tilde{\alpha}_k})\nonumber\\
			&+\sum_{k=2}^{{\color{black}T_2}(\varepsilon)}\eta_k\tilde{\alpha}_k(\frac{3d_y{\color{black}D_1}L^2\theta_{2,k}^2}{{\color{black}2}}+\frac{3L^2d_yD_k^{(2)}\theta_{2,k}^2}{b\eta_k\tilde{\alpha}_k}+\frac{80d_yL^4\theta_{2,k}^2}{\rho_{k+1}^2})\nonumber\\
			&+\sum_{k=2}^{{\color{black}T_2}(\varepsilon)}\frac{2\delta^2(\varrho_k^2D_k^{(1)}+\iota_k^2D_k^{(2)})}{b}+\frac{160\tilde{\alpha}_2L^2\rho_2}{\eta_{2}\tilde{\beta}^2\rho_3^3}\sigma_y^2+\sum_{k=2}^{{\color{black}T_2}(\varepsilon)}(\frac{36{\color{black}A_7}d_yL^2\sigma_y^2}{13b}+{\color{black}2D_1A_2L^2\sigma_y^2})(k+2)^{-1}.\label{azocthm1:5}
		\end{align}
		Denote $\underbar{S}=\min\limits_{\lambda\in\Lambda}\min\limits_{x\in\mathcal{X}}\min\limits_{y\in\mathcal{Y}}S_k(x,y,\lambda)$. By the definition of $\theta_{1,k}^2$ and $\theta_{2,k}^2$, we then conclude from \eqref{azocthm1:5} that
		\begin{align}
			&\sum_{k=2}^{{\color{black}T_2}(\varepsilon)}{\color{black}D_1}\eta_k\tilde{\alpha}_k\mathbb{E}\|\nabla {\color{black}\mathcal{G}}^{\tilde{\alpha}_k,\tilde{\beta} ,\tilde{\gamma}_k}(x_k,y_k,\lambda_k)\|^2\nonumber\\
			\le& S_{2}(x_{2},y_{2},\lambda_2)-\underbar{S}+16 L^2\eta_1\tilde{\alpha}_1\sigma_y^2+\frac{\rho_2}{2}\sigma_y^2+\frac{160\tilde{\alpha}_2L^2\rho_2}{\eta_{2}\tilde{\beta}^2\rho_3^3}\sigma_y^2+\sum_{k=2}^{{\color{black}T_2}(\varepsilon)}(\frac{36{\color{black}A_7}d_yL^2\sigma_y^2}{13b}+{\color{black}2D_1A_2L^2\sigma_y^2})(k+2)^{-1}\nonumber\\
			&+\sum_{k=2}^{{\color{black}T_2}(\varepsilon)}\frac{{\color{black}D_1}\varepsilon^2\eta_k\tilde{\alpha}_k}{16}+\sum_{k=2}^{{\color{black}T_2}(\varepsilon)}\frac{{\color{black}D_1}\varepsilon^2\eta_k\tilde{\alpha}_k}{16}+\sum_{k=2}^{{\color{black}T_2}(\varepsilon)}\frac{2\delta^2({\color{black}A_4^2A_5+A_5^2A_7})}{b}(k+2)^{-1}.\label{azocthm1:6}
		\end{align}
		Since $\sum_{k=2}^{{\color{black}T_2}(\varepsilon)}(k+2)^{-1}\le\ln({\color{black}T_2}(\varepsilon)+2)$ and $\sum_{k=2}^{{\color{black}T_2}(\varepsilon)}(k+2)^{-9/13}\ge\frac{13}{4}({\color{black}T_2}(\varepsilon)+3)^{4/13}-\frac{13}{4}4^{4/13}$, by the definition of $C_1$ and $C_2$, we get
		\begin{align}
			\frac{{\color{black}7}\varepsilon^2}{8}\le\frac{C_1+C_2\ln({\color{black}T_2}(\varepsilon)+2)}{{\color{black}D_1A_2}(\frac{13}{4}({\color{black}T_2}(\varepsilon)+3)^{4/13}-\frac{13}{4}4^{4/13})},\label{azocthm1:7}
		\end{align}
		which completes the proof. 
	\end{proof}
	
	\begin{remark}
		It is easily verified from \eqref{azocthm1:1} that $T_2(\varepsilon)=\tilde{\mathcal{O}}\left(\varepsilon ^{-6.5} \right)$ by Theorem \ref{azocthm1}, which means that the number of iterations for Algorithm \ref{zoalg:2} to obtain an $\varepsilon$-stationary point of problem \eqref{problem:s} is upper bounded by $\tilde{\mathcal{O}}\left(\varepsilon ^{-6.5} \right)$ for solving stochastic nonconvex-concave minimax problem with coupled linear constraints.  By Theorem \ref{azocthm1} we have that $b=\mathcal{O}(1)$ which implies that the total number of function value queries are bounded by $\tilde{\mathcal{O}}\left((d_x+d_y)\varepsilon ^{-6.5}  \right)$ for solving the stochastic nonconvex-concave minimax problem with coupled linear constraints.
	\end{remark}

	\begin{remark}
		If $A=B=c=0$, problem \eqref{problem:s} degenerates into problem \eqref{min-s}, Algorithm \ref{zoalg:2} can be simplified as Algorithm \ref{zoalg:3}. Theorem \ref{azocthm1} implies that the number of iterations for Algorithm \ref{zoalg:3} to obtain an $\varepsilon$-stationary point of problem \eqref{min-s} is bounded by $\tilde{\mathcal{O}}\left(\varepsilon ^{-6.5} \right)$ and the total number of function value queries are bounded by $\tilde{\mathcal{O}}\left((d_x+d_y)\varepsilon ^{-6.5}  \right)$ for solving stochastic nonconvex-concave minimax problems.
	\end{remark}

	\section{Numerical Results}\label{sec4}
	{\color{black}In this section, we compare the proposed ZO-PDAPG and ZO-RMPDPG algorithms with three first-order algorithms: PDAPG \cite{zhang2022primal}, MGD \cite{tsaknakis2021minimax}, and PGmsAD \cite{dai2022optimality} algorithms for solving two minimax problems with coupled constraints: adversarial attacks in network flow problems and data poisoning against logistic regression. All numerical tests are implemented using Python 3.9 and run on a laptop with an Apple M4 Pro processor and 24 GB RAM.
	}
	\subsection{Adversarial attack in network flow problems.}
	In this subsection, we consider the following adversarial attack in network flow problems. 
	{\color{black}The minimum cost network flow problem can be defined as finding the minimum cost path from a source to a sink so that we can successfully transmit a certain amount of traffic. \cite{tsaknakis2021minimax} considers an extension of this type of problem where in addition to regular network users, there is an adversary who injects traffic into the network, forcing regular users to use more expensive paths. The adversary problem can be formulated as follows:}
	\begin{align}
		\max _{y\in \mathcal{Y}}  \min _{x\in \mathcal{X}} & \quad \sum_{(i, j) \in E} q_{i j}\left(x_{i j}+y_{i j}\right) x_{i j}-\frac{\eta}{2}\|y\|^2\nonumber\\
		\mbox { s.t. } & \quad \mathbf{x}+\mathbf{y} \leq \mathbf{p}\label{testprob}
	\end{align}
	where $\mathcal{X}= \{x\mid 0 \leq \mathbf{x} \leq \mathbf{p}, \sum_{(i, t) \in E} x_{i t}=r_t, \sum_{(i, j) \in E} x_{i j}-\sum_{(j, k) \in E} x_{j k}=0, \forall j \in V \backslash\{s, t\}\}$, $\mathcal{Y}=\{y\mid 0 \leq \mathbf{y} \leq \mathbf{p}, \sum_{(i, j) \in E} y_{i j}=b\}$, $V$ and $E$ is the set of vertices and edges of a directed graph $G=(V, E)$, respectively, $s$ and $t$ denote the source and sink node respectively, and $r_t$ is the total flow to sink node $t$, $\mathbf{x}=\left\{x_e\right\}_{e \in E}$ and $\mathbf{p}=\left\{p_e\right\}_{e \in E}$ denote the vectors of flows and the capacities at edge $e$, respectively, and $\mathbf{y}=\left\{y_e\right\}_{e \in E}$  denote the set of flows controlled by the adversary, $q_{ij}>0$ is the unit cost at edge $(i,j)\in E$ and $b > 0$ is its total budget, $\eta>0$.  {\color{black} Similar to \cite{tsaknakis2021minimax}}, we define the following ``relative cost increase" as performance measure, i.e.,
	$$\omega=\frac{q_{att}-q_{cl}}{q_{cl}},$$
	where $q_{cl}$ and $q_{att}$ denote the minimum cost before and after the attack,  {\color{black} i.e., the minimum cost when $y=0$ and $y=\hat{y}$, respectively, with $\hat{y}$ being the solution obtained by several tested algorithms}. The higher the increase of the minimum cost the more successful the attack.

	{\color{black}We use the Erdos-Renyi model with parameter $prob$ (i.e., the probability of an edge appearing on the graph is $p$) to randomly generate a network (graph) consisting of $n$ nodes.}
	We compare the proposed ZO-PDAPG with three first-order algorithms, which are PDAPG \cite{zhang2022primal}, MGD \cite{tsaknakis2021minimax} and PGmsAD \cite{dai2022optimality}. We set $\eta=0.05$,  $1/\beta=0.8$, $1/\alpha=0.6$, $\gamma=0.5$ in ZO-PDAPG and PDAPG. 
	All the stepsizes for MGD are set to be 0.5 which is same as that in \cite{tsaknakis2021minimax}. We set $\alpha_y=0.8$, $\alpha_x=0.6 $ for PGmsAD. The capacity $p_{ij}$ and the cost coefficients $q_{ij}$ of the edges are generated uniformly at random in the interval [1, 2],  and total flow $r_t$ is $d\%$ of the sum of capacities of the edges exiting the source with $d$ being a parameter to be chosen. {\color{black}In the four tested algorithms, we use Python's built-in CVXOPT solver (the interior point algorithm) to solve the projection subproblems. In fact, projection onto $\mathcal{X}$ or $\mathcal{Y}$ can be regarded as the solution of a strongly convex quadratic programming problem, in which the interior point algorithm has a linear convergence rate  \cite{Gondzio}.}
	
	We compare the relative cost increase of the four tested algorithms for solving \eqref{testprob} with different $b$. For each given $b$, all the algorithms are run for 15 times with randomly chosen $p_{ij}$ and $q_{ij}$ which means 15 graphs with different capacity and cost vectors.

	\begin{figure}[t]
		\centering  
		\subfigure[$n=10, p=0.75, d=30$]{
			\includegraphics[width=0.45\textwidth]{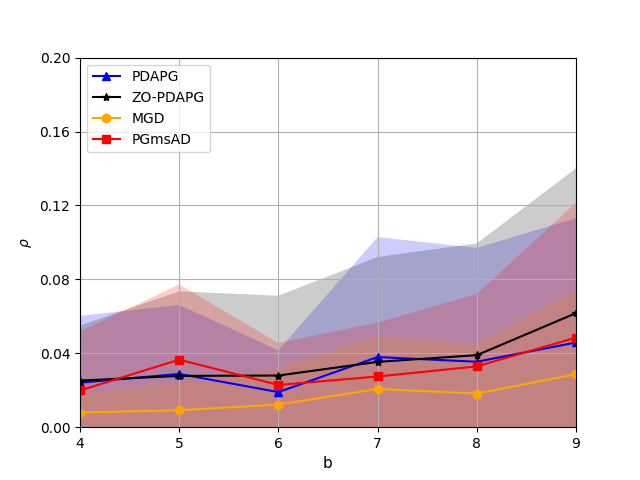}}
		\subfigure[$n=10, p=1, d=30$]{
			\includegraphics[width=0.45\textwidth]{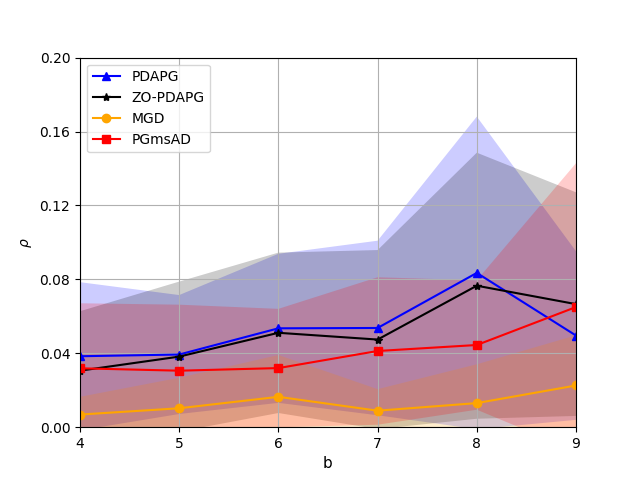}}
		\subfigure[$n=10, p=0.75, d=20$]{
			\includegraphics[width=0.45\textwidth]{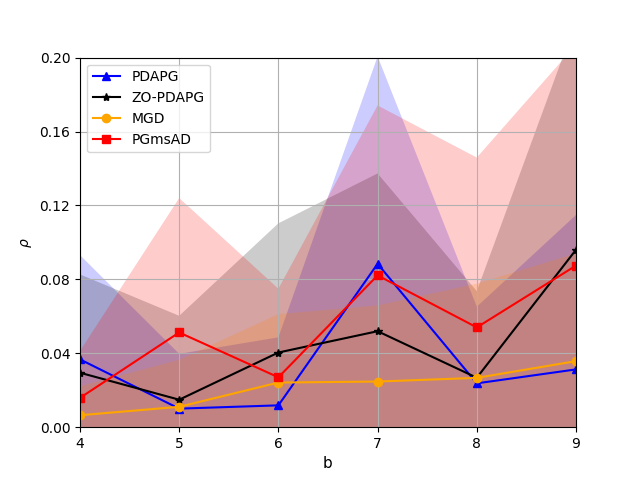}}
		\subfigure[$n=10, p=1, d=20$]{
			\includegraphics[width=0.45\textwidth]{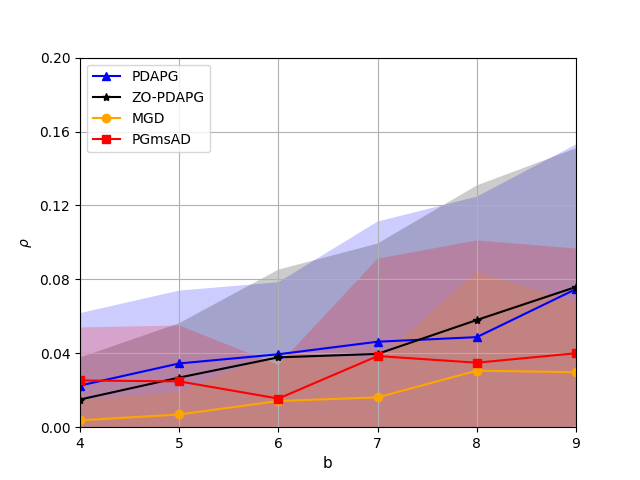}}
		\caption{{\color{black}The relative cost increase of the four algorithms. $n$ is the number of nodes in the graph, and $p$ is the probability of an edge appearing in the graph.}}
		\label{Fig.zo-cost}
	\end{figure}

	Figure \ref{Fig.zo-cost} shows the average of the relative cost increase of four tested algorithms over 15 independent runs. The shaded part around lines denotes the standard deviation over 15 independent runs.
	It shows that the proposed ZO-PDAPG performs similar to the state-of-art first order algorithms.
	
	{\color{black}
		\subsection{Data poisoning against logistic regression.}
		In this subsection, we study the data poisoning against logistic regression \eqref{app:dp}, which is also a nonconvex-concave minimax problem with coupled constraints.
		
		In the experiment, we generate a dataset containing $n=1000$ samples $\left\{a_i, b_i\right\}_{i=1}^n$, where $a_i \in \mathbb{R}^{100}$ is sampled from a Gaussian distribution $\mathcal{N}(\mathbf{0}, \mathbf{I})$. The dataset will be randomly split into a training dataset $\mathcal{D}_{\text {train}}$ and a test dataset $\mathcal{D}_{\text {test}}$. If $\frac{1}{1+e^{-\left(a_i^T \theta^*+s_i\right)}}>0.5$, then the label $b_i=1$, otherwise $b_i=0$, where $s_i \in \mathcal{N}\left(0,10^{-3}\right)$ is random noise, and we choose $\theta^*=\mathbf{1}$ as the true model parameters. Since we cannot accurately know the distribution of the real data set in practical computation, we choose the empirical risk on the training set instead of the expected risk in \eqref{app:dp}, that is, we select $g(x,y)= \hat{h}\left(x, y; \mathcal{D}_{\mathrm{tr}, p}\right)+ \hat{h}\left(\mathbf{0}, y; D_{\mathrm{tr}, c}\right)$, where
		$$
		\hat{h}(x,y; \mathcal{D})=-\frac{1}{|D|} \sum_{\left(a_i, b_i\right) \in \mathcal{D}}\left[b_i \log \left(h\left(x, y; a_i\right)\right)+\left(1-b_i\right) \log \left(1-h\left(x, y; a_i\right)\right)\right],
		$$
		$h\left(x, y; a_i\right)=\frac{1}{1+e^{-(x+a_i)^{T}y}}$. The poisoning rate is $\frac{|\mathcal{D}_{\mathrm{tr}, p}|}{|\mathcal{D}_{\mathrm{train}}|}=10\%$.
		
		We compare the proposed ZO-PDAPG and ZO-RMPDPG with PDAPG, MGD, and PGmsAD. In ZO-PDAPG and PDAPG, we set $1/\beta=0.01$, $1/\alpha_k=\frac{0.01}{10+\sqrt{k}}$, $\gamma_k=\frac{0.01}{10+\sqrt{k}}$, $\rho_{k}=\frac{0.01}{k^{0.25}}$. In ZO-RMPDPG, we set $\beta=0.015$, $\tilde{\alpha}_k=\frac{0.001}{(k+2)^{4/13}}$, $\tilde{\gamma}_k=\frac{0.01}{10+(k+2)^{4/13}}$, $\eta_k=\frac{10}{10+(k+2)^{5/13}}$, $\varrho_k=\frac{1}{(k+2)^{12/13}}$, $\iota_k=\frac{2}{(k+2)^{8/13}}$, and choose the mini-batch size $b=5$. In MGD, the step size of $x$ and $\lambda$ is set to 0.0001, and the step size of $y$ is set to 0.01. For PGmsAD, we set $\alpha_y=0.01$, $\alpha_x=0.0001$.
		\begin{figure}[t]
			\centering  
			\subfigure[]{
				\label{figa}
				\includegraphics[width=0.45\textwidth]{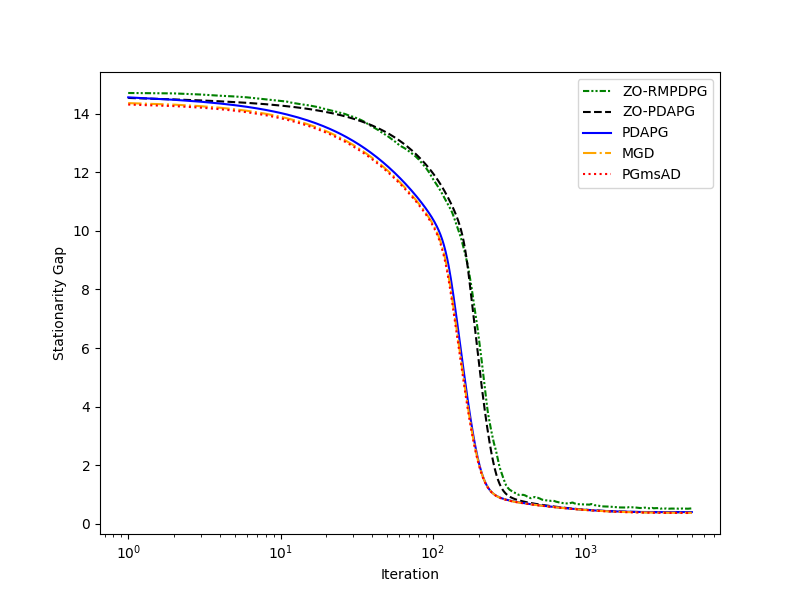}}
			\subfigure[]{
				\label{figb}
				\includegraphics[width=0.45\textwidth]{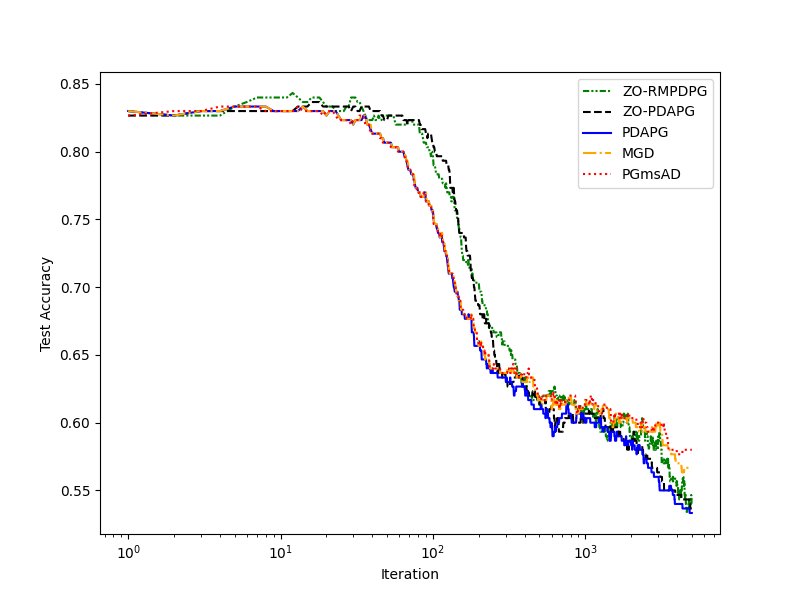}}
			\caption{Performance of five algorithms in data poisoning against logistic regression.}
			\label{Fig.poison}
		\end{figure}
		Figure \ref{figa} shows the stationary gap of the five tested algorithms. The stationary gap of the ZO-PDAPG and ZO-RMPDPG algorithms are comparable to the other three first-order algorithms. Figure \ref{figb} shows the test accuracy of the five tested algorithms. The performance of the ZO-PDAPG and ZO-RMPDPG algorithms is similar to that of three first-order algorithms.
	}
	
	\section{Conclusions}
	In this paper, we have investigated zeroth-order algorithms and their complexity analysis for solving deterministic and stochastic nonconvex-(strongly) concave minimax problems with coupled linear constraints. Zeroth-order methods have gained increasing attention in black-box optimization settings, such as certain machine learning models, where gradient information is unavailable and only function value queries are possible.
	
	To address this class of problems, we propose two novel zeroth-order algorithms: a zeroth-order primal-dual alternating projection gradient (ZO-PDAPG) algorithm and a zeroth-order regularized momentum primal-dual projected gradient algorithm (ZO-RMPDPG). We establish that, under the deterministic nonconvex-strongly concave setting, the iteration complexity to reach an $\varepsilon$-stationary point is $\mathcal{O}\left( \varepsilon^{-2} \right)$ for ZO-PDAPG and $\tilde{\mathcal{O}}(\varepsilon^{-3})$ for ZO-RMPDPG. For the nonconvex-concave case, the corresponding complexities are $\mathcal{O}\left( \varepsilon^{-4} \right)$ and $\tilde{\mathcal{O}}(\varepsilon^{-6.5})$, respectively. Analogous results are derived for their stochastic counterparts. To our knowledge, these are the first zeroth-order algorithms with guaranteed iteration complexity for minimax problems of this class featuring coupled linear constraints. Furthermore, for stochastic nonconvex-concave minimax problems without coupled constraints, the ZO-RMPDPG algorithm achieves the best-known complexity among all existing zeroth-order methods.
	
	Numerical experiments on adversarial attacks in network flow problems and data poisoning attacks against logistic regression demonstrate the practical efficiency of the proposed algorithms.

	\bibliographystyle{amsplain}

\end{document}